\documentclass[12pt,a4paper]{amsart}

\usepackage{euscript,amsfonts,amssymb,amsmath,amscd,amsthm,enumerate,hyperref}

\usepackage{comment}

\usepackage{mathrsfs}

\usepackage{graphicx}

\usepackage{color}

\usepackage[margin=1.1in]{geometry}

\usepackage{bbm}

\newtheorem{theorem}{Theorem}[section]
\newtheorem*{theorem*}{Theorem}
\newtheorem{lemma}[theorem]{Lemma}
\newtheorem{definition}[theorem]{Definition}
\newtheorem{proposition}[theorem]{Proposition}

\newtheorem{corollary}[theorem]{Corollary}
\newtheorem{assumption}[theorem]{Assumption}

\theoremstyle{remark}
\newtheorem{rmk}[theorem]{Remark}
\theoremstyle{remark}
\newtheorem{example}[theorem]{Example}

\def \eps{\varepsilon}
\def \lb{\left(} \def \rb{\right)}

\def \dist{\mathrm{dist}}

\newcommand{\E}{\mathbb E}
\newcommand{\EE}{\mathbb E}

\newcommand{\bZ}{\mathbf Z}
\newcommand{\cC}{\mathcal C}

\newcommand{\cD}{\mathcal D}
\newcommand{\cP}{\mathcal P}

\newcommand{\pp}{\mathbb{P}}
\newcommand{\qq}{\mathbb{Q}}
\newcommand{\rr}{\mathbb{R}}

\newcommand{\nn}{\mathbb{N}}

\newcommand{\eq}{\begin{equation}}
\newcommand{\en}{\end{equation}}

\newcommand{\bone}{\mathbf{1}}
\newcommand{\PP}{\mathbb{P}}
\newcommand{\RR}{\mathbb{R}}
\newcommand{\CC}{\mathbb{C}}
\newcommand{\wt}{\widetilde}
\newcommand{\ZZ}{\mathbb{Z}}
\newcommand{\NN}{\mathbb{N}}

\numberwithin{equation}{section}

\newcommand{\squeezeup}{\vspace{-3mm}}

\begin{document}

\title[Scaling limits of external multi-particle DLA]{Scaling limits of external multi-particle DLA on the plane and the supercooled Stefan problem}
\author{Sergey Nadtochiy, Mykhaylo Shkolnikov, and Xiling Zhang}
\address{Department of Applied Mathematics, Illinois Institute of Technology, Chicago, IL 60616.}
\email{snadtochiy@iit.edu}
\address{ORFE Department, Bendheim Center for Finance, and Program in Applied \& Computational Mathematics, Princeton University, Princeton, NJ 08544.}
\email{mshkolni@gmail.com}
\address{Department of Applied Mathematics, Illinois Institute of Technology, Chicago, IL 60616.}
\email{xzhang135@iit.edu}

\footnotetext[1]{S.~Nadtochiy is partially supported by the NSF CAREER grant DMS-1855309.}
\footnotetext[2]{M.~Shkolnikov is partially supported by the NSF grant DMS-1811723.}
\footnotetext[3]{The authors are grateful to A. Negron for conducting the numerical experiment which, in particular, produced Figure \ref{fig:2}.}

\begin{abstract}
We consider (a variant of) the external multi-particle diffusion-limited aggregation (MDLA) process of \textsc{Rosenstock} and \textsc{Marquardt} on the plane. Based on the recent findings of \cite{DT}, \cite{DNS19} in one space dimension it is natural to conjecture that the scaling limit of the growing aggregate in such a model is given by the growing solid phase in a suitable ``probabilistic'' formulation of the single-phase supercooled Stefan problem for the heat equation. To address this conjecture, we extend the probabilistic formulation from \cite{DNS19} to multiple space dimensions. We then show that the equation that characterizes the growth rate of the solid phase in the supercooled Stefan problem is satisfied by the scaling limit of the external MDLA process with an inequality, which can be strict in general. In the course of the proof, we establish two additional results interesting in their own right: (i) the stability of a ``crossing property'' of planar Brownian motion and (ii) a rigorous connection between the probabilistic solutions to the supercooled Stefan problem and its classical and weak solutions.
\end{abstract}

\maketitle

\section{Introduction and main results}
\label{se:intro}

Growth processes of \textit{diffusion-limited aggregation} (DLA) type have attracted much attention since their introduction by \textsc{Rosenstock} and \textsc{Marquardt} in \cite{RoMa} and by \textsc{Witten} and \textsc{Sander} in \cite{WiSa,WiSa2}. The former let an aggregate grow in $\mathbb{Z}^d$ by starting with $\{0\}$ and attaching a neighboring site whenever a particle from a system evolving outside of the aggregate attempts to enter the aggregate from that site, which is referred to as the multi-particle DLA (MDLA). In \cite{WiSa,WiSa2}, the aggregate grows by the same rule, except that the particles do not evolve simultaneously: each particle is generated according to a chosen distribution at the moment when the previous one is absorbed by the aggregate.
The immediate interest in the DLA type processes stemmed from the dendritic nature of the growing aggregates therein which, at least qualitatively, agrees with the shapes observed experimentally in crystallization, electrodeposition, and bacteria colony growth among many others (see \cite{Sa} for more details and pointers to the experimental literature). This has led to many numerical simulation studies of the aggregates resulting from DLA, including \cite{Vo,Vo2}, \cite{Me2} where the fractal dimension of the aggregate is estimated as $\approx\!1.7$ in two ambient space dimensions, \cite{Me1}, \cite{HaLe} where the same is discussed for somewhat modified growth mechanisms, and \cite{Vi}, \cite{BDLP} where parallels and contrasts to Laplacian growth models (such as the Hele-Shaw process of \cite{SaTa}) are drawn. We refer to the books \cite{Ka} and \cite{Me3} for a placement of DLA into the broader context of random growth processes. 

\medskip

Despite the considerable interest in DLA type processes their mathematical theory in multiple ambient space dimensions remains in its infancy, to the extent that \textsc{Sander} characterizes the DLA process of \cite{WiSa,WiSa2} as a ``devilishly difficult model to solve, even approximately'' in \cite{Sa}. For the aggregate of \cite{WiSa,WiSa2}, only an upper bound of the order $n^{\frac{2}{d+1}}$ on its radius upon the attachment of $n$ sites has been established (see \cite{Ke2}, and also \cite{Ke1}), as well as the almost sure convergence to infinity of the number of holes in it when $d=2$ (see \cite{EW}). For the aggregate in the MDLA process of \cite{RoMa}, only the linear growth of the radius when the initial particle density is sufficiently high has been shown (see \cite{Sly} and, for $d=2$, also \cite{SiSt}). Meanwhile, much more detailed results have been proved for various simplifications of the original DLA processes from \cite{WiSa,WiSa2} and \cite{RoMa}: e.g., for the DLA process of \cite{WiSa,WiSa2} in regular trees (see \cite{BPP}, and also \cite{Pi}, \cite{Dev}, \cite{AlSh}), the DLA process of \cite{WiSa,WiSa2} in infinite cylinders (see \cite{BeYa}), the DLA process of \cite{WiSa,WiSa2} in the hyperbolic plane (see \cite{El}), the directed DLA process in the plane (see \cite{Ma}), and the Hastings-Levitov aggregation process of \cite{HaLe} (see \cite{CaMa}, \cite{NoTu}, \cite{VST}, \cite{Si}). 

\medskip

Moreover, the MDLA process in one ambient space dimension, where the aggregate amounts to a set of consecutive integers, has been analyzed rigorously in the recent works \cite{KeSi2}, \cite{Sly}, \cite{ElNaSl}. Specifically, consider a system of particles that are placed at the sites of $\mathbb{Z}\backslash\{0\}$ initially as the atoms of a Poisson random measure with a mean measure that assigns mass $p$ to each element of $\mathbb{Z}\backslash\{0\}$, and each particle performs an independent continuous time simple symmetric random walk (SSRW). Then, as time $t$ goes to infinity, the endpoints of the aggregate have the order $t^{\frac{1}{2}}$ for $p<1$ (see \cite{KeSi2}), the order $t^{\frac{2}{3}}$ for $p=1$ (see \cite{ElNaSl}), and the order $t$ for $p>1$ (see \cite{Sly}). Prior to \cite{ElNaSl}, the work \cite{DT} has studied such a version of the MDLA process in one ambient space dimension, with $p=1$ and with the additional modification that the aggregate advances according to the number of particles it absorbs, so that the location of its right endpoint always gives the number of particles absorbed thus far. The paper \cite{ChSw} considers a version of the one-dimensional MDLA where the initial locations of the particles are drawn from an arbitrary symmetric distribution that puts at most one particle at each site of $\mathbb{Z}\backslash\{0\}$, and the particles evolve according to a symmetric continuous time bond exclusion process (described in detail below), so that no two particles ever occupy the same site.
It is shown in \cite{DT} that the path of the right endpoint of the aggregate on the $t^{2/3}$ scale tends to the free boundary $\Lambda$ in the \textit{single-phase supercooled Stefan problem} (1SSP) for the heat equation. A similar connection to the supercooled Stefan equation is established in \cite{ChSw}.

\medskip

The 1SSP for the heat equation refers to a choice of signs in the traditional Stefan problem for the heat equation (see \cite{Stefan1,Stefan2,Stefan3,Stefan4}, and also \cite{Bri}) which physically corresponds to the freezing of a supercooled liquid (i.e., a liquid cooled below its equilibrium freezing temperature while kept liquid). In the supercooled regime, the Stefan problem for the heat equation develops singularities already in one space dimension, in the sense that the velocity of the free boundary diverges in finite time for generic initial conditions (see \cite{Sher}). Qualitatively, this is in line with the rapid freezing of supercooled liquids noted experimentally. Despite the singularities, for the appropriate notion of a solution of the 1SSP, referred to as the ``probabilistic solution'', the global well-posedness has been recently established in \cite{DNS19}. Specifically, for appropriate initial conditions, the free boundary $\Lambda:\,[0,\infty)\to[0,\infty)$ satisfying the probabilistic formulation of the 1SSP
\begin{equation}\label{prob_sol_one}
\Lambda_t=\pp(\tau\le t),\;\; t\ge0\quad\text{for}\quad \tau:=\inf\{t\ge0:\,X_{0}+B_t\le\Lambda_t\},
\end{equation}
where $X_{0}>0$ is a random variable of prescribed distribution (representing the initial temperature distribution in the liquid phase) and $B$ is a standard Brownian motion independent thereof, exists and is unique under the minimality condition
\begin{equation}\label{eq.intro.prob_sol_one.jumpSize}
\Lambda_t-\Lambda_{t-}:=\Lambda_t-\lim_{s\uparrow t} \Lambda_s
=\inf\big\{x>0:\;\pp\big(\tau\ge t,\,X_{0}+B_t\in(\Lambda_{t-},\Lambda_{t-}+x]\big)<x\big\}.
\end{equation}

Motivated by the connection between the MDLA and the 1SSP, established by \cite{ChSw} and \cite{DT} in one space dimension, we investigate herein the connection between a natural version of MDLA and the 1SSP for the heat equation in \textit{two} space dimensions. To this end, we take $N\in\mathbb{N}$ and start with an initial aggregate $\Gamma^N_{0-}\subset\rr^d$ given by a union of translates of $[-N^{-\frac{1}{d}}/2,N^{-\frac{1}{d}}/2]^d$ centered at elements of $\mathbb{Z}^d/N^{\frac{1}{d}}$. 
In addition, we place $N$ particles at distinct points in $(\mathbb{Z}^d/N^{\frac{1}{d}})\cap(\rr^d\backslash\Gamma^N_{0-})$ according to an initial distribution that is symmetric across particles. Regarding $\mathbb{Z}^d/N^{\frac{1}{d}}$ as a hypercubic lattice, we let the particles perform a symmetric continuous time bond exclusion process. That is, we assign i.i.d.~ Poisson processes of rate $N^{\frac{2}{d}}/2$ to the bonds, and at the arrival times of these processes, any particles located at the endpoints of the respective bond jump to the other endpoint of the bond. When a particle finds itself at distance $N^{-\frac{1}{d}}/2$ from the aggregate, it is absorbed and the $N^{-\frac{1}{d}}\times\cdots\times N^{-\frac{1}{d}}$ hypercube around it is attached to the aggregate. If additional particles are at distance $N^{-\frac{1}{d}}/2$ from the aggregate upon the attachment of the new hypercube, they are absorbed and the $N^{-\frac{1}{d}}\times\cdots\times N^{-\frac{1}{d}}$ hypercubes around them are attached to the aggregate as well. This procedure continues until all non-absorbed particles are further than $N^{-\frac{1}{d}}/2$ from the aggregate. The non-absorbed particles resume the continuous time bond exclusion process on $\mathbb{Z}^d/N^{\frac{1}{d}}$, etc. Figure \ref{fig:2} depicts snapshots of the growing aggregate for $d=2$ and two values of $N$. In fact, our results hold for more general growth processes, whose (slightly technical) definition is deferred to Section \ref{se:system}. 

\begin{figure}
	\begin{center}
		\begin{tabular} {cc}
			{
				\includegraphics[width = 0.45\textwidth]{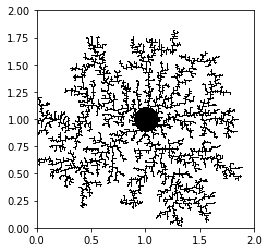}
			} & {
				\includegraphics[width = 0.45\textwidth]{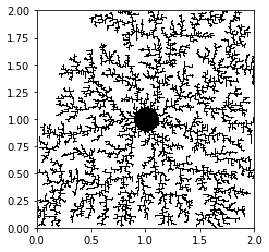}
			}\\
			{
				\includegraphics[width = 0.45\textwidth]{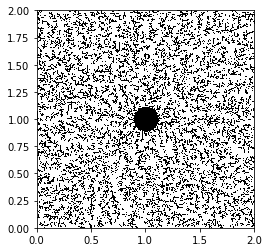}
			} & {
				\includegraphics[width = 0.45\textwidth]{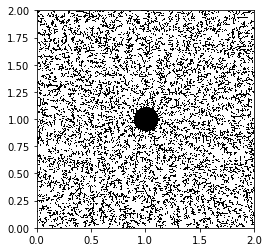}
			}
		\end{tabular}
	    
		\caption{Simulated aggregates in the MDLA variant on a $2\times2$ square. For simplicity, the particles are reflected at the boundary of the square. In all cases, the particles are uniformly distributed outside of the closed ball $\Gamma^N_{0-}$ initially, with approximately $20\%$ of the initially available sites being occupied. Top left: $N=9,900$, $T=0.01$. Top right: $N=9,900$, $T=0.015$. Bottom left: $N=61,875$, $T=0.01$. Bottom right: $N=61,875$, $T=0.015$.}
		\squeezeup		
		\label{fig:2}
	\end{center}
\end{figure}

\medskip

In view of the findings of \cite{ChSw} and \cite{DT} in one space dimension, it is natural to conjecture that, in the $N\to\infty$ scaling limit, the growing aggregate $\Gamma^N$ described in the preceding paragraph tends to the growing solid phase in the 1SSP for the heat equation in any space dimension. To address this conjecture we introduce herein a probabilistic formulation of the 1SSP for the heat equation in multiple space dimensions. In the latter, one seeks a non-decreasing family $(\Gamma_t)_{t\ge0}$ of closed subsets of $\rr^d$ such that $\Gamma_{0-}\subset\Gamma_0$ for a given closed $\Gamma_{0-}\subset\rr^d$, one has $\Gamma_t=\bigcap_{s>t} \Gamma_s$ for all $t\ge0$, and
\begin{align}\label{prob_sol_multi}
& \int_{\RR^d} \left(\chi(t,x)-\bone_{\Gamma_{0-}}(x)\right) \varphi(x)\,\mathrm{d}x 
= \mathbb{E}\big[\varphi\big(X_{0}+B_\tau\big)\,\mathbf{1}_{\{\tau\le t\}}\big],\;\varphi\!\in\! C_c^\infty(\rr^d,\rr),\;t\!\ge\!0 \\
&\,\text{for}\quad\tau:=\inf\{t\ge0:\,X_{0}+B_t\in\Gamma_t\}, \notag
\end{align} 
where $X_{0}\in\rr^d\backslash\Gamma_{0-}$ is a random variable of prescribed distribution (representing the initial temperature distribution in the liquid phase), $B$ is a standard Brownian motion independent thereof, and $\chi$ is a ``phase'' function satisfying
\begin{align}
& \chi(t,x)\,\bone_{\RR^d\setminus \bigcup_{0\le s\leq t}\Delta\Gamma_s}(x)
=\bone_{\Gamma_t\setminus \bigcup_{0\le s\leq t}\Delta\Gamma_s}(x),\quad
\int_{\Delta\Gamma_s} \chi(t,x)\,\mathrm{d}x = \text{Leb}(\Delta\Gamma_s),\;\;s\in[0,t],
\label{eq.SecStefan.weakStefan.chi.0}
\end{align}
where $\Delta\Gamma_s:=\Gamma_s\setminus \bigcup_{0\le r<s} \Gamma_r$ and the union $\bigcup_{0\le s\leq t}\Delta\Gamma_s$ is taken over all $s\in[0,t]$ such that $\text{Leb}(\Delta\Gamma_s)>0$ (the set of such $s$ is at most countable). Note that if the function $t\mapsto \text{Leb}(\Gamma_t)$ is continuous, the condition \eqref{eq.SecStefan.weakStefan.chi.0} amounts to $\chi(t,x)=\bone_{\Gamma_t}(x)$.

\medskip

For $d=1$, $\Gamma_{0-}=(-\infty,0]$ and $\Gamma_t=(-\infty,\Lambda_t]$, $t\ge0$, the formulation \eqref{prob_sol_one}--\eqref{eq.intro.prob_sol_one.jumpSize} implies \eqref{prob_sol_multi}--\eqref{eq.SecStefan.weakStefan.chi.0}, with $\chi(t,\cdot\,)$ defined as $0$ on $(\Lambda_t,\infty)$, as $1$ on $(-\infty,0]$, and as the density of the restriction of the distribution of $X_{0}+B_{\tau}$ to $(0,\Lambda_t]$ on the latter set. Indeed, condition \eqref{prob_sol_multi} can be verified directly from the definition of $\chi$. To check \eqref{eq.SecStefan.weakStefan.chi.0}, on any sub-interval of $\{t\ge0:\,\Lambda_{t-}= \Lambda_t\}$\footnote{It follows from the results of \cite{DNS19} that this set is a union of a countable number of disjoint intervals.} we use the first equation in \eqref{prob_sol_one} to get
\begin{align*}
\Lambda_{t}-\Lambda_{s} = \PP(\tau\in(s,t]) = \PP\left(X_{\tau}\in(\Lambda_{s},\Lambda_t]\right),\;\; s\in[0,t),
\end{align*}
which yields the first equation in \eqref{eq.SecStefan.weakStefan.chi.0}. For the second equation in \eqref{eq.SecStefan.weakStefan.chi.0}, we compute
\begin{align*}
\pp\big(\tau\ge s,\,X_{0}+B_s\in(\Lambda_{s-},\Lambda_{s-}+x]\big) = \int_{\Lambda_{s-}}^{\Lambda_{s-}+x} \chi(t,y)\,\mathrm{d}y,
\;\;s\in[0,t),\;\;x\in[0,\Delta\Lambda_s]
\end{align*}
and apply \eqref{eq.intro.prob_sol_one.jumpSize}.
In Subsection \ref{se:prob_sol} we discuss the probabilistic formulation \eqref{prob_sol_multi}--\eqref{eq.SecStefan.weakStefan.chi.0} in detail and show that it leads to solutions of the 1SSP for the heat equation which are weaker than the classical solutions but stronger than the weak solutions.

\medskip

In the main result (Theorem \ref{thm:main}), we fix a $T\in(0,\infty)$ and encode the paths of the particles in (our version of) MDLA on $[0,T]$ by their empirical measure $\mu^{X,N}$, viewed as a random element of $\mathcal{P}(\mathcal{D}([0,T+1],\rr^d))$. The latter denotes the space of probability measures on $\mathcal{D}([0,T+1],\rr^d)$, the set of c\`{a}dl\`{a}g paths $[0,T+1]\to\rr^d$ carrying the Skorokhod M1 metric, endowed with the topology of weak convergence. Here, for technical reasons, we need to consider the dynamics of the particles and of the aggregate on the extended interval $[-1,T+1]$, assuming that the particle paths are constant on $[-1,0]$ and the aggregate is equal to $\Gamma^N_{0-}$ on $[-1,0)$. The growing aggregate $\Gamma^N$ is captured by the Euclidean distance function $D^N_t(x):=\min_{y\in\Gamma^N_t} |x-y|$, regarded as a random element of $C(\rr^d,\mathcal{D}([-1,T+1],\rr))$, the space of continuous functions $\rr^d\to\mathcal{D}([-1,T+1],\rr)$ with the topology of uniform convergence.
Finally, we introduce the following assumption. 

\begin{assumption}\label{asmp:main}
\begin{enumerate}[(a)]
\item The empirical measures $(\mu^{X,N}_0)_{N\in\nn}$ of the initial particle locations converge in distribution on $\mathcal{P}(\rr^d)$ as $N\to\infty$ to a deterministic limit $\mu^X_0$.
\item The deterministic distance functions $D^N_{0-}(\cdot):=\min_{y\in\Gamma^N_{0-}} |\cdot-\,y|$, $N\in\nn$ converge in $C(\rr^d,\rr)$ to a limit $D_{0-}(\cdot)$ as $N\to\infty$. 
\item The diameters of all connected components of all $\Gamma^N_{0-}$, $N\in\nn$ are bounded below by a positive constant.
\item For the closed set $\Gamma_{0-}:=\{x\in\rr^d:\,D_{0-}(x)=0\}$, the Lebesgue measure of the symmetric difference $\Gamma_{0-}\Delta \Gamma^N_{0-}$ tends to $0$ as $N\to\infty$.
\end{enumerate}
\end{assumption}

Our main result, Theorem \ref{thm:main}, follows from Propositions \ref{prop:existence}, \ref{prop:limDLA.probSol.ineq}, \ref{prop:brownian_limit}, Theorem \ref{thm:char}.

\begin{theorem}\label{thm:main}
\begin{enumerate}[(a)]
\item Any subsequence of $(\mu^{X,N},D^N)_{N\in\nn}$ has a further subsequence converging in distribution on $\mathcal{P}(\mathcal{D}([0,T+1],\rr^d))\times C(\rr^d,\mathcal{D}([-1,T+1],\rr))$ under Assumption \ref{asmp:main}(a)--(b).~Moreover, for almost every realization $(\mu^X(\omega),D(\omega))$ of a weak limit point of $(\mu^{X,N},D^N)_{N\in\nn}$, the non-decreasing family $(\Gamma_t(\omega))_{t\in[0,T]}:=(\{x\in\rr^d:\,D_t(\omega;x)=0\})_{t\in[0,T]}$ of closed subsets of $\rr^d$ satisfies $\Gamma_{0-}\subset\Gamma_0(\omega)$ and $\Gamma_t(\omega)=\bigcap_{t<s\le T} \Gamma_s(\omega)$, $t\in[0,T)$. 
\item For $d=2$ and for almost every realization $(\mu^X(\omega),D(\omega))$ of a limit point,
\begin{enumerate}
\item[(b.1)] under Assumption \ref{asmp:main}(a)--(c), the canonical process under $\mu^X(\omega)$ has the form $X_{0}+B_{t\wedge\tau}$, $t\in[0,T]$, where $\tau:=\inf\{t\ge0\!:X_{0}+B_t\in\Gamma_t(\omega)\}$, $B$~is a Brownian motion, and $X_0$ is an independent random variable with the distribution $\mu^X_0$,
\item[(b.2)] under Assumption \ref{asmp:main}(a)--(d), $\int_{\Gamma_t(\omega)\backslash\Gamma_{0-}} \varphi(x)\,\mathrm{d}x \ge \E\big[\varphi\big(X_{0}+B_\tau\big)\,\mathbf{1}_{\{\tau\le t\}}\big]$ for all $\varphi\in C^\infty_c(\rr^2,[0,\infty))$, $t\in[0,T]$.
\end{enumerate}
\end{enumerate}
\end{theorem}
	
\begin{rmk}\label{rem:mainRes.explain}
Theorem \ref{thm:main}(b) says that, for $d=2$ and under Assumption \ref{asmp:main}(a)--(d), the growing aggregate $\Gamma(\omega)$ associated with almost every realization  $(\mu^X(\omega),D(\omega))$ satisfies the probabilistic formulation of the 1SSP for the heat equation, where $\chi(t,x)=\bone_{\Gamma_t}(x)$, except that \eqref{prob_sol_multi} holds with an inequality instead of an equality. In Subsection~\ref{subse:DLA.vs.Stefan} we give an example which illustrates that, for a general limit point of MDLA, one cannot expect to find a $\chi$ satisfying \eqref{eq.SecStefan.weakStefan.chi.0} and such that \eqref{prob_sol_multi} holds with an equality, even for densities of $X_{0}$ that are uniformly close to $0$. In such cases, the density of absorbed particles in the aggregate is too low, which is due to the multitude of microscopic holes formed when the $N^{-\frac{1}{2}}\times N^{-\frac{1}{2}}$ squares are attached to the aggregate in a ``disorderly" fashion. The density of the resulting ``solid'' crystal is lower and its enthalpy is higher than normal. The resulting object is sometimes referred to as a ``mushy region'' (see, e.g., \cite{Vis}). 
\end{rmk}

The rest of the paper is structured as follows. In Section \ref{se:system}, we define the general growth processes studied herein. Sections \ref{se:LargePop.lim}, \ref{se:Prop}, and \ref{se:Stefan} then provide the proofs of Theorem \ref{thm:main}(a), (b.1), and (b.2), respectively, for these processes. One of the main technical contributions of the paper is Proposition \ref{prop:stabCrossing}, which shows the stability of a certain ``crossing property'' of planar Brownian motion, interesting on its own and explaining why Theorem \ref{thm:main}(b) is restricted to $d=2$. Another result, interesting in its own right, is stated in Propositions \ref{prop:classic.is.prob} and \ref{prop:prob.is.weak}, which show that a classical solution to the single-phase supercooled Stefan problem must be a probabilistic solution and that a probabilistic solution must be a weak solution (in the PDE sense). The example mentioned in Remark \ref{rem:mainRes.explain} is described at the end of Subsection \ref{subse:DLA.vs.Stefan}. We close with an appendix, which contains the verification that the version of MDLA described in this introduction falls into the framework of Section \ref{se:system}. 

\section{Preliminaries on a class of external MDLA processes}
\label{se:system}


In the definition of our class of external MDLA processes we use the following notations. We fix a $d\in\nn$, take for each $N\in\nn$ the hypercubic lattice $\bZ_N:=\mathbb{Z}^d/N^{\frac{1}{d}}$, and write $\mathcal C_N$ for the set of translates of the hypercube $[-N^{-\frac{1}{d}}/2,N^{-\frac{1}{d}}/2]^d$ by elements of $\bZ_N$. Clearly, the union of all elements of $\mathcal C_N$ is $\RR^d$, and the intersection of any two elements of $\mathcal C_N$ is a Lebesgue null set. Consider now an initial aggregate $\Gamma^N_{0-}$ (``seed'') which is a union of elements of $\mathcal C_N$. Without loss of generality we assume that the Lebesgue measure of $\RR^d\setminus\Gamma^N_{0-}$ is at least $1$. 
Finally, for a fixed $T\in(0,\infty)$, we let
$$ 
\big(\widetilde{X}^{1,N}_t,\,\widetilde{X}^{2,N}_t,\,\ldots,\,\widetilde{X}^{N,N}_t\big)_{t\in[0,T+1]}
$$
be a $(\bZ_N)^N$-valued c\`{a}dl\`{a}g process. We refer to it as the underlying particle system.

\begin{definition}\label{def:extDLA}
For a given underlying particle system and seed $\Gamma^N_{0-}\subset\RR^d$ which is a union of hypercubes from $\mathcal C_N$, the associated external MDLA process consists of the $(\bZ_N)^N$-valued stochastic (``particles'') process
$$ 
\big(X^{1,N}_t,\,X^{2,N}_t,\,\ldots,\,X^{N,N}_t\big)_{t\in[0,T]}
$$
and a random non-decreasing family $\{\Gamma^N_t\}_{t\in[-1,T]}$ of closed subsets of $\RR^d$ (``aggregate'') such that
\begin{itemize}
\item $X^{i,N}_t = \widetilde{X}^{i,N}_{t\wedge\tau^{i,N}}$, $t\in[0,T]$, $i=1,\,2,\,\ldots,\,N$, where
\begin{equation}\label{eq.tau.i.N.def}
\tau^{i,N} := \inf\{t\geq0:\,\mathrm{dist}(\widetilde{X}^{i,N}_t,\Gamma^N_{t-})\leq N^{-\frac{1}{d}}/2\},\; i=1,\,2,\,\ldots,\,N\;\;\;\text{and}\;\;\;\Gamma^N_{t-}:=\bigcup_{0\le s<t} \Gamma^N_s,
\end{equation}
\item $\Gamma^N_t=\Gamma^N_{0-}$, $t\in[-1,0)$,
\item at every hitting time $\tau^{i,N}$, the set $\Gamma^N_{\tau^{i,N}}$ is given by the union of $\Gamma^N_{\tau^{i,N}-}$ with all sets $C$, such that $C$ is a connected union of cubes in $\cC_N$ whose centers belong to $\{X^{1,N}_{\tau^{i,N}},\,X^{2,N}_{\tau^{i,N}},\,\ldots,\,X^{N,N}_{\tau^{i,N}}\}$ and $C\cap\Gamma^N_{\tau^{i,N}-}\neq \varnothing$, and the set $\Gamma^N_\cdot$ remains unchanged between these hitting times.
\end{itemize}
\end{definition}

\begin{rmk}
One could let $\tau^{i,N} := \inf\{t\geq0:\,\mathrm{dist}(\widetilde{X}^{i,N}_t,\Gamma^N_{t-})=0\}$ in Definition \ref{def:extDLA} and, at any such hitting time, grow the aggregate by only adding the hypercube from which the hitting particle entered it. This would not affect our results, except that 
Example \ref{counterexample} would be more difficult to analyze.
\end{rmk}

\begin{rmk}
One could also study an internal MDLA process in which particles are absorbed upon exiting the aggregate: 
$
\tau^{i,N} := \inf\{t\geq0:\,\mathrm{dist}(\widetilde{X}^{i,N}_t,\Gamma^N_{t-})\geq N^{-\frac{1}{d}}/2\}.
$
The main results of Sections \ref{se:LargePop.lim} and \ref{se:Prop} hold for internal MDLA processes (as well as for other growth processes, see Remarks \ref{rmk:ext.GenGrowthModels} and \ref{rmk:ext.GenGrowthModels.2}). However, the results of Section \ref{se:Stefan} would change. The limit points of internal MDLA processes are associated with the regular, rather than supercooled, version of the Stefan problem. The former is much better understood than the latter and, naturally, much more is known about scaling limits of internal MDLA processes, so we focus on external MDLA processes herein. 
\end{rmk}

It is clear that for any given underlying particle system and seed, Definition \ref{def:extDLA} uniquely determines the aggregate $\Gamma^N$ on $[-1,T]$. Informally speaking, whenever a particle finds itself in a hypercube adjacent to the aggregate, the hypercube is instantaneously added to the aggregate. The choice of hypercubes as the ``building blocks" of the aggregate is made for convenience -- we expect our results to apply for any other natural family of partitions of $\RR^d$, provided the diameters of the partition elements decay at the rate of $N^{-\frac{1}{d}}$. It is clear as well that the sets $\{\Gamma^N_t\}_{t\in[-1,T]}$ increase by jumps only, with a total of at most $N$ jumps, and that the right-continuity $\Gamma^N_t = \bigcap_{t<s\le T} \Gamma^N_s$, $t\in[-1,T)$ holds. The domain of the time index of $\Gamma^N$ is extended from $[0,T]$ to $[-1,T]$ in order to yield a precise meaning to $\Gamma^N_{0-}$. Also for technical reasons, we consider the underlying particle system on the extended time interval $[0,T+1]$.  

\medskip

While Definition \ref{def:extDLA} amounts to a very general class of external MDLA processes, our results require further assumptions. To state the first assumption, we introduce the empirical measure of the particle paths:
\begin{equation}
\mu^{\widetilde{X},N}:= \frac{1}{N}\sum_{i=1}^N \delta_{\widetilde{X}^{i,N}},
\end{equation}
which is a random probability measure on the Skorokhod space $\mathcal{D}([0,T+1],\rr^d)$ of c\`{a}dl\`{a}g functions $[0,T+1]\to\RR^d$. As stated in Section \ref{se:intro}, we equip $\mathcal{D}([0,T+1],\rr^d)$ with the Skorokhod M1 metric, making it a complete separable metric space (see, e.g., \cite{Whitt}, \cite{JacodShiryaev}), and denote by $\mathcal{P}(\mathcal{D}([0,T+1],\rr^d))$ the space of probability measures on $\mathcal{D}([0,T+1],\rr^d)$ endowed with the topology of weak convergence. Thus, $\mu^{\widetilde{X},N}$ is a random element of $\mathcal{P}(\mathcal{D}([0,T+1],\rr^d))$. Our first assumption says that the input of the external MDLA process, encoded by $\mu^{\widetilde{X},N}$ and $\Gamma^N_{0-}$, converges as $N\rightarrow\infty$. This is the minimal assumption (up to the distinction between the full sequence and its subsequences) needed for the existence of limit points of external MDLA (Section \ref{se:LargePop.lim}).

\begin{assumption}\label{ass:0}
The seed $\Gamma^N_{0-}$ tends to a closed set $\Gamma_{0-}\subset\RR^d$ in the distance sense\footnote{Such a convergence is also known as the Painlev\'e-Kuratowski convergence, see \cite[Corollary 4.7]{Rocka}.}{\em :}
\begin{equation}\label{aggIC}
\lim_{N\rightarrow\infty}\mathrm{dist}(x,\Gamma^N_{0-}) = \mathrm{dist}(x,\Gamma_{0-}),\;\; x\in\RR^d.
\end{equation}
The empirical measure $\mu^{\wt X,N}$ converges in distribution as a random element of $\mathcal{P}(\mathcal{D}([0,T+1],\rr^d))$.
\end{assumption}

An example of a particle system satisfying Assumption \ref{ass:0} is the system of i.i.d.~continuous time simple symmetric random walks with jump rates $N^{\frac{2}{d}}$. More generally, if the underlying particle system $\{\wt X^{i,N}\}_{i=1}^N$ has a symmetric distribution and satisfies certain regularity conditions, then the propagation of chaos holds, i.e., the convergence of the empirical measure of $\{\wt X^{i,N}_0\}_{i=1}^N$ to a deterministic limit implies the convergence of $\mu^{\wt X,N}$.
We refer to \cite{Sznitman} for a detailed study of the chaoticity property.

\medskip

The next assumption ensures that the underlying particles make only nearest neighbor jumps and that their empirical measure converges to a Wiener measure.~It is crucial for our description of the limiting particles as absorbed Brownian motions (Section \ref{se:Prop}).

\begin{assumption}\label{ass:1}
For any large enough $N\in\nn$, with probability one, we have
\begin{equation}\label{eq.smallJumpSize.def}
\big|\wt X^{i,N}_t-\lim_{s\uparrow t} \wt X^{i,N}_s\big|
\in\{0,N^{-\frac{1}{d}}\},\;\;t\in[0,T+1],\;\; i=1,\,2,\,\ldots,\,N,
\end{equation}
and the diameters of all connected components of all $\Gamma^N_{0-}$, $N\in\nn$ are bounded below by a positive constant.
In addition, Assumption \ref{ass:0} holds, and $\mu^{\wt X}:=\lim_{N\to\infty} \mu^{\wt X,N}$ almost surely equals to the distribution of the process $(\xi+B_t)_{t\in[0,T+1]}$, where $\xi$ is an $\rr^d$-valued random variable and $B$  is a standard Brownian motion independent of $\xi$. 
\end{assumption}

\begin{rmk}\label{rmk:relax.Ass1}
Assumption \ref{ass:1} can, in fact, be relaxed by replacing $B$ with $(B_{L_t})_{t\in[0,T+1]}$, for any strictly increasing stochastic process $L$.
\end{rmk}

An example of a particle system fulfilling Assumption \ref{ass:1} is the system of i.i.d.~continuous time simple symmetric random walks with jump rates $N^{\frac{2}{d}}$.

\begin{rmk}
As the volume of the aggregate grows by $N^{-1}$ every time a particle is absorbed, it makes sense to refer to $N^{-1}$ as the size of a particle. Thus, the scaling of an MDLA process is determined by the asymptotic behavior of: (i) the size of the particles, (ii) the size of their jumps (``space scaling''), and (iii) their speed (``time scaling''). We scale these quantities as $N^{-1}$, $N^{-\frac{1}{d}}$, and $N^{\frac{2}{d}}$, respectively. This choice is motivated by our goal to study the connection between the scaling limits of MDLA processes and the Stefan problem (Section \ref{se:Stefan}). The latter is a parabolic problem, which explains the parabolic space-time scaling, and the size of the particles is scaled to ensure that their total size is of order $1$. However, other scaling choices are possible, e.g., the choice of $N^{-2}$, $N^{-\frac{1}{d}}$, and $N^{1+\frac{2}{d}}$, with the additional assumption that the particles are generated from a remote source at a rate of order $N$, is more natural for connecting the scaling limits to the (elliptic) Hele-Shaw problem, which has been oftentimes alluded to (implicitly and heuristically) in the physics literature.  
\end{rmk}

The following assumption is used in Section \ref{se:Stefan} to establish a connection between the limit points of external MDLA processes and the supercooled Stefan problem.

\begin{assumption}\label{ass:2}
As $N\rightarrow\infty$, the Lebesgue measure of the symmetric difference $\Gamma_{0-}\Delta \Gamma^N_{0-}$ tends to $0$.
In addition, for any large enough $N\in\nn$, with probability one,
\begin{eqnarray}
&& \wt X^{i,N}_0\notin\Gamma^N_{0-},\;\;i=1,\,2,\,\ldots,\,N, \\
&& \wt X^{i,N}_t\neq \wt X^{j,N}_t,\;\;t\in[0,T+1],\;\; i\neq j. 
\end{eqnarray}
\end{assumption}

Assumption \ref{ass:2} guarantees that the particles start outside of the seed and that no two particles can occupy the same site (``exclusion rule''). Together with \eqref{eq.smallJumpSize.def}, it implies that, at any given time, the number of particles absorbed by the aggregate coincides with the number of hypercubes from $\mathcal C_N$ added to the aggregate at that time. The latter can be interpreted as an energy (or, density) conservation property and is discussed in more detail in connection with the Stefan problem, in Subsection \ref{subse:DLA.vs.Stefan}. Note that systems of i.i.d.~continuous time simple symmetric random walks violate Assumption \ref{ass:2}. Nonetheless, there are particle systems that satisfy Assumptions \ref{ass:1} and \ref{ass:2}. An example is provided by the ``bond exclusion process'', which is described briefly in Section \ref{se:intro}.~Below we provide its detailed description (see also \cite[Chapter I, Section 3), Example d)]{Sznitman}).

\medskip

Consider $N$ particles that, at time $0$, are placed at distinct locations in $\bZ_N\cap(\rr^d\backslash\Gamma^N_{0-})$ according to an initial distribution that is symmetric across particles. A bond of $\bZ_N$ is a pair of sites $x,y\in\bZ_N$ such that all coordinates of $x$ and $y$ are equal, except for one which differs by exactly $N^{-\frac{1}{d}}$ between the two sites. To each bond we associate an independent Poisson process of rate $N^{\frac{2}{d}}/2$. Every particle $\wt X^{i,N}$, located at a site $x$, remains there until the first of the $2d$ Poisson processes associated with the $2d$ bonds emanating from $x$ jumps, at which time the particle jumps across the corresponding bond. The particle at the other end of the bond, if present, is forced to also traverse the bond, and, therefore, no two particles can occupy the same site. Thus, this underlying particle system satisfies Assumption \ref{ass:2}, provided the seed $\Gamma_{0-}$ is approximated with sufficient precision by $\Gamma^N_{0-}$. It is shown in Proposition \ref{prop:brownian_limit} that the bond exclusion process fulfills Assumption~\ref{ass:1} as well, provided that $\mu^{\wt X,N}_0:=\frac{1}{N}\sum_{i=1}^N \delta_{\wt X^{i,N}_0}$ tends to a deterministic limit. 

\medskip

We note that the physically relevant quantities in an external MDLA process are given by the empirical measure of the particle paths
\begin{equation}
\mu^{X,N}:= \frac{1}{N}\sum_{i=1}^N \delta_{X^{i,N}}
\end{equation}
and by the aggregate $\Gamma^N$. Herein, we establish the existence of limit points (in law) of the pair $(\mu^{X,N},\Gamma^N)$ as $N\rightarrow\infty$, and describe several important properties of those. In particular, we analyze the relationship between the particles and the aggregate in a limiting process, and their connection to the supercooled Stefan problem.

\begin{rmk}
One can, of course, encode an external MDLA process by $(\mu^{\wt X,N},\Gamma^N)$, as the absorption times $\{\tau^{i,N}\}^N_{i=1}$ and the empirical measure of the particle paths $\mu^{X,N}$ can then be reconstructed. However, a priori, the functional dependence of $\mu^{X,N}$ on $(\mu^{\wt X,N},\Gamma^N)$ may be lost when we pass to the $N\to\infty$ limit (in law). In fact, one of the main conclusions of our work (Theorem \ref{thm:char}) is that, under Assumption~\ref{ass:1} and for $d\in\{1,2\}$, this functional dependence is not lost in the limit, so that almost every realization of $\mu^{X}:=\lim_{N\to\infty} \mu^{X,N}$ is the distribution of the canonical process under $\mu^{\wt X}:=\lim_{N\to\infty} \mu^{\wt X,N}$ stopped upon hitting the aggregate. The latter follows from the convergence of the absorption times which, in turn, relies on Assumption \ref{ass:1} and on the stability of a ``crossing property'' of Brownian motion, known before for $d=1$ and shown in Subsection \ref{subse:crossing} of this paper for $d=2$.
\end{rmk}

\section{Existence of limit points}
\label{se:LargePop.lim}

In this section, we prove that, under Assumption \ref{ass:0}, any sequence of external MDLA processes has a limit point in distribution. Before proceeding to the result we discuss the topology used to define the convergence in distribution of MDLA processes.
First, we recall that, for $a_1<a_2\in\RR$ and $d\in\!\nn$, the space $\cD([a_1,a_2],\rr^d)$ of right-continuous functions $[a_1,a_2]\to\rr^d$ with left limits is equipped with the Skorokhod M1 metric. For technical reasons, we extend the particle paths to $(T,T+1]$ constantly and make their empirical measure $\mu^{X,N}$ a random element of $\cP(\cD([0,T+1],\rr^d))$, the space of probability measures on $\cD([0,T+1],\rr^d)$ with the topology of weak convergence.

\medskip

Next, again for technical reasons, we allow the time index of the aggregate $\Gamma^N$ to vary in $[-1,T+1]$ by setting
\begin{equation}
\Gamma^N_t = \Gamma^N_T,\;\; t\in(T,T+1].
\end{equation}
We also notice that, at any time $t\in[-1,T+1]$, the aggregate $\Gamma^N_t$ is uniquely characterized by the associated distance function 
\begin{equation}
D^N_t(x):=\dist(x,\Gamma_t^N)=\min_{y\in\Gamma^N_t} |x-y|,\;\; x\in\RR^d.
\end{equation}
The non-decreasing nature of the family $\{\Gamma^N_t\}_{t\in[-1,T+1]}$ and its right-continuity imply that the functions $t\mapsto D^N_t(x)$, $x\in\RR^d$ are elements of $\cD([-1,T+1],\RR)$. On the other hand, the triangle inequality yields the $1$-Lipschitz property of $x\mapsto D^N_t(x)$, $t\in[-1,T+1]$. Thus, the random field 
\begin{equation}
[-1,T+1]\times\RR^d\to\RR,\;\;(t,x)\mapsto D^N_t(x)
\end{equation}
can be viewed as a random element of $C(\RR^d,\cD([-1,T+1],\RR))$. We equip the latter space with the topology of uniform convergence on compacts induced by the M1 metric on $\cD([-1,T+1],\RR)$. This, in turn, induces a topology on the space of paths of $\Gamma^N$. With this choice of the topology, the limit points of $(\Gamma^N)_{N\in\nn}$ correspond to the limit points of $(D^N)_{N\in\nn}$. As it is more convenient to work with distance functions, we analyze the limit points (in distribution) of external MDLA processes directly as the limit points (in distribution) of $(\mu^{X,N},D^N)_{N\in\nn}$. Hereby, each $(\mu^{X,N},D^N)$ is regarded as a random element of $\cP(\cD([0,T+1],\RR^d))\times C(\RR^d,\cD([-1,T+1],\RR))$, and we put the product topology thereon.

\begin{rmk}
The criterion for M1 convergence of monotone functions (see, e.g., \cite[Theorem 4.2]{DIRT2}) and the $1$-Lipschitz property of every $x\mapsto D^N_t(x)$ imply the following. For any subsequence of $(D^N)_{N\in\nn}$ converging almost surely in $C(\RR^d,\cD([-1,T+1],\RR))$, there almost surely exists a set of times $\mathbb T\subset[-1,T+1]$ of full Lebesgue measure such that $x\mapsto D^N_t(x)$ converges pointwise for all $t\in\mathbb T$. This shows the convergence of the corresponding subsequence of $(\Gamma^N_t)_{N\in\nn}$ in the distance sense (or, equivalently, in the Painlev\'e-Kuratowski topology) for all $t\in\mathbb T$ almost surely. 
\end{rmk}

Proposition \ref{prop:existence} is the main result of this section. It states that every subsequence of $(\mu^{X,N},D^N)_{N\in\nn}$ admits a further subsequence converging in distribution.

\begin{proposition}\label{prop:existence}
Under Assumption \ref{ass:0}, any increasing infinite sequence in $\NN$ has a subsequence along which $(\mu^{X,N},D^N)$ has a limit $(\mu^X,D)$ in distribution. Moreover, if $\lim_{N\to\infty} \mu^{\wt X,N}$ assigns probability one to $C([0,T+1],\RR^d)$ almost surely, then $\mu^X$ assigns probability one to $C([0,T+1],\RR^d)$ almost surely.
\end{proposition}

\begin{proof}
Due to Prokhorov's Theorem, it suffices to prove that the sequence of the respective laws of $(\mu^{X,N},D^N)_{N\in\nn}$ is tight, i.e., that for any $\varepsilon>0$, there exist compact sets $K_1\subset \mathcal P(\cD([0,T+1],\rr^d))$ and $K_2\subset C(\RR^d,\cD([-1,T+1],\rr))$ with the property
\begin{equation}\label{eq.existPf.eq2}
	\PP(\mu^{X,N}\in K_1,\,D^N\in K_2)\geq 1-\varepsilon,\;\;N\in\nn.
\end{equation}
To construct $K_1$, we first infer from Assumption \ref{ass:0} that the respective laws of $(\mu^{\wt X,N})_{N\in\nn}$ form a tight sequence. Consequently, for any $\varepsilon>0$, there exists a compact set $\wt K_1\subset \mathcal P(\cD([0,T+1],\rr^d))$ such that
$$
\PP(\mu^{\wt X,N}\in \wt K_1)\geq 1-\varepsilon/2,\;\;N\in\nn.
$$
Let $\Omega^N_0:= \{\omega\in\Omega:\,\mu^{\wt X,N}(\omega)\in \wt K_1\}$. Then, Prokhorov's Theorem implies that, for any $\varepsilon'>0$, there exists a compact set $K'_1\subset \mathcal \cD([0,T+1],\rr^d)$ for which
$$
\mu^{\wt X,N}(\omega)(K'_1)\geq 1-\varepsilon',\;\;\omega\in\Omega^N_0,\;\; N\in\nn.
$$
Due to the characterization of compactness in $\cD([0,T+1],\rr^d)$ of \cite[Theorem 12.12.2]{Whitt}, the latter gives a uniform boundedness and a uniform oscillation estimate on at least $N(1-\varepsilon')$ paths of the underlying particles $\{\wt X^{i,N}(\omega)\}_{i=1}^N$, for all $\omega\in\Omega^N_0$ and all $N\in\nn$. Clearly, the corresponding stopped particle paths $X^{i,N}(\omega)$ satisfy the same uniform boundedness and uniform oscillation estimates. Applying the reverse direction of Prokhorov's Theorem we conclude that there is a compact $K_1\subset \mathcal P(\cD([0,T+1],\rr^d))$ such that $\mu^{X,N}(\omega)\in K_1$ for all $\omega\in\Omega_0^N$ and all $N\in\nn$. This finishes the construction of $K_1$, as we recall that $\PP(\Omega^N_0)\geq 1-\varepsilon/2$, $N\in\nn$. 

\medskip

It remains to show that, for any $\varepsilon>0$, we can select a compact set $K_2\subset C(\RR^d,\cD([-1,T+1],\rr))$ so that $\pp(D^N\in K_2)\ge 1-\varepsilon/2$, $N\in\nn$. Since the space $C(\RR^d,\cD([-1,T+1],\rr))$ is equipped with the topology of uniform convergence on compacts, it suffices to verify that, for any compact set $K\subset\RR^d$ and any $\widehat{\varepsilon}>0$, we can find a compact set $\widehat{K}_2\subset C(K,\cD([-1,T+1],\rr))$ such that $\pp(D^N|_K\in\widehat{K}_2)\ge 1-\widehat{\varepsilon}$, $N\in\nn$. We claim that the closure of $\{D^N(\omega)|_K:\,\omega\in\Omega,\,N\in\nn\}$ in $C(K,\cD([-1,T+1],\rr))$ can serve as $\widehat{K}_2$ for any $\widehat{\varepsilon}>0$. By the Arzel\`{a}-Ascoli Theorem for general topological spaces (see, e.g., \cite[Chapter 7, Theorem 17]{Kelley}), it is enough to check that the functions $x\mapsto D^N_\cdot(\omega)(x)$ from $K$ to $\cD([-1,T+1],\rr)$ are equicontinuous and that the set $\{D^N_\cdot(\omega)(x):\,\omega\in\Omega,\,N\in\nn\}$ is precompact in $\cD([-1,T+1],\rr)$ for every $x\in K$. The former is a consequence of the fact that the M1-distance between any $D^N_\cdot(\omega)(x)$, $D^N_\cdot(\omega)(x')$ is at most $|x-x'|$, as can be inferred by choosing synchronous parametrizations for the two functions and noticing that they differ by at most $|x-x'|$ pointwise. On the other hand, the precompactness of the set $\{D^N_\cdot(\omega)(x):\,\omega\in\Omega,\,N\in\nn\}$, for every $x\in K$, follows from the constancy on $[-1,0]$, $[T,T+1]$ and the monotonicity of this set's elements, $\sup_{N\in\nn,\,\omega\in\Omega}\, D^N_{-1}(\omega)(x)=\sup_{N\in\nn}\,\mathrm{dist}(x,\Gamma^N_{0-})<\infty$ (cf.~\eqref{aggIC}), and
the precompactness criterion for $\cD([-1,T+1],\rr)$ of \cite[Theorem 12.12.2]{Whitt}. 

\medskip

Thus far, we have shown that the sequence of the respective laws of $(\mu^{X,N},D^N)_{N\in\nn}$ is tight. Now, suppose that $\lim_{N\to\infty} \mu^{\wt X,N}$ assigns probability one to $C([0,T+1],\RR^d)$ almost surely. Thanks to the tightness of the respective laws of $(\mu^{\wt X,N},\mu^{X,N})_{N\in\nn}$ and to the Skorokhod Representation Theorem (see \cite[Chapter 3, Theorem 5.1]{Dudley}), we may assume that there exists an $\Omega_1\subset\Omega$ such that $\pp(\Omega_1)=1$ and for all $\omega\in\Omega_1$, the limits $\lim_{N\to\infty} \mu^{\wt X,N}(\omega)$, $\mu^X(\omega)=\lim_{N\to\infty} \mu^{X,N}(\omega)$ exist and the former assigns probability one to $C([0,T+1],\RR^d)$. Here, to simplify notation, we have relabeled the convergent subsequence of $(\mu^{\wt X,N},\mu^{X,N})_{N\in\nn}$. Since the M1-convergence to a continuous limit implies the uniform convergence to that limit (see, e.g., \cite[Theorem 4.2]{DIRT2}), it is easy to check (e.g., by an application of the Skorokhod Representation Theorem for $\cD([0,T+1],\rr^d)$) that the sequence $(\mu^{\wt X,N}(\omega))_{N\in\nn}$ is $C$-tight in the sense of \cite[Chapter~VI, Definition 3.25]{JacodShiryaev} for all $\omega\in\Omega_1$. As the oscillations and the absolute values of a stopped path do not exceed those of the original path, we deduce that the sequence $(\mu^{X,N}(\omega))_{N\in\nn}$ is also $C$-tight for all $\omega\in\Omega_1$. (Formally, the latter follows from \cite[Chapter~VI, Theorem 3.21 and Proposition 3.26, (i)$\Leftrightarrow$(iii)]{JacodShiryaev}.) In particular, its limit $\mu^X(\omega)$ assigns probability one to $C([0,T+1],\rr^d)$ for all $\omega\in\Omega_1$.
\end{proof}

\begin{rmk}\label{rmk:ext.GenGrowthModels}
The proof of Proposition \ref{prop:existence} reveals that Assumption \ref{ass:0} can be slightly relaxed. Instead of requiring that $\mathrm{dist}(\,\cdot\,,\Gamma^N_{0-})$ and $\mu^{\wt X,N}$ converge, it is enough to require that the sequence $(\mathrm{dist}(x,\Gamma^N_{0-}))_{N\in\nn}$ is bounded for each $x\in\rr^d$ and that any subsequence of $(\mu^{\wt X,N})_{N\in\nn}$ has a limit point. Moreover, the proof of Proposition~\ref{prop:existence} does not fully utilize the structure of an external MDLA process. Namely, as long as Assumption \ref{ass:0} (or the relaxation just mentioned) holds for $\{\wt X^{i,N}\}_{i=1}^N$, Proposition \ref{prop:existence} applies to any family of processes $\{X^{i,N}\}_{i=1}^N$ whose paths are constructed by stopping the respective paths of $\{\wt X^{i,N}\}_{i=1}^N$ in some manner (possibly not even at stopping times) and to any non-decreasing right-continuous family $\{\Gamma^N_t\}_{t\in[0,T]}$ of closed subsets of $\RR^d$.
\end{rmk}

Note that every limit point $D$ of $(D^N)_{N\in\nn}$ defines a limiting aggregate: 
\begin{equation}\label{def:aggr}
\Gamma_t:= \{x\in\RR^d:\,D_t(x)=0\},\;\; t\in[-1,T+1].
\end{equation}
The Skorokhod Representation Theorem for $C(\RR^d,\cD([-1,T+1],\RR))$ and the characterization of M1 convergence for monotone functions in \cite[Theorem 4.2]{DIRT2} yield that, with probability one, all functions $t\mapsto D_t(x)$, $x\in\RR^d$ are non-increasing and all functions $x\mapsto D_t(x)$, $t\in[-1,T+1]$ are $1$-Lipschitz. This observation implies that, with probability one, the family $\{\Gamma_t\}_{t\in[-1,T+1]}$ consists of closed subsets of $\rr^d$ and is non-decreasing. It is also clear that $\Gamma_\cdot$ is constant on $[-1,0)$ and on $[T,T+1]$. Finally, we claim that $\{\Gamma_t\}_{t\in[-1,T+1]}$ is right-continuous, in the sense that $\Gamma_t = \bigcap_{t<s\le T+1} \Gamma_s$, $t\in[-1,T+1)$. Indeed, $\Gamma_t\subset\bigcap_{t<s\le T+1} \Gamma_s$, $t\in[-1,T+1)$. Moreover, if there existed an $x\in\bigcap_{t<s\le T+1} \Gamma_s$ such that $x\notin\Gamma_t$, then $D_t(x)>0$ while $D_s(x)=0$, $t<s\le T+1$, contradicting the right-continuity of $D_\cdot(x)$.  

\begin{rmk}\label{rem:discont.Dbar}
We point out that, with probability one, the union of the discontinuity times of $t\mapsto D_t(x)$ over all $x\in\rr^d$ is countable. This is due to the $1$-Lipschitz property of $x\mapsto D_t(x)$, $t\in[-1,T+1]$ which lets us take the union over all $x\in\qq^d$ only. We refer to the complement of that union as the set of \textit{continuity times} of $D$. If a subsequence of $(D^N)_{N\in\nn}$ converges to $D$ in distribution (as in Proposition~\ref{prop:existence}), then the Skorokhod Representation Theorem allows us to find realizations of the random variables involved such that, with probability one, for all compact sets $K\subset\rr^d$ and all continuity times $t$ of $D$, it holds
$\sup_{x\in K} |D^N_t(x)-D_t(x)|\to0$ along the subsequence~in~consideration. (Here, we use the $1$-Lipschitz property of $x\mapsto D^N_t(x)$.)
\end{rmk}


\section{Limit points as absorbed Brownian motions}
\label{se:Prop}

Pick a limit point $(\mu^X,D)$ in distribution of a sequence of external MDLA processes, as in Proposition \ref{prop:existence}, and define $\Gamma$ as the aggregate associated with $D$ via \eqref{def:aggr}. Recall that, under Assumption \ref{ass:1}, the empirical measure of the underlying particle paths converges to the law of a standard Brownian motion with an independent initial condition $\xi$. The goal of this section is to show that, under Assumption \ref{ass:1} and with $d\in\{1,2\}$, almost every realization of $\mu^X$ is the distribution of a standard Brownian motion started from $\xi$ and absorbed upon hitting the aggregate $\Gamma$. Theorem~\ref{thm:char} gives the formal statement.

\begin{theorem}\label{thm:char}
For $d\in\{1,2\}$ and under Assumption \ref{ass:1}, let $(\mu^X,D)$, defined on a probability space $(\Omega,\mathcal{F},\PP)$, be a limit point in distribution of a sequence of MDLA processes as in Proposition \ref{prop:existence}. Then, for $\PP$-almost every $\omega\in\Omega$, the projection of $\mu^{X}(\omega)$ onto $\mathcal{P}(C([0,T],\RR^d))$, i.e., the restriction of this measure to shorter paths, coincides with the distribution of the stochastic process
\begin{equation}
\widehat{\omega}\mapsto \big(\xi(\widehat{\omega})+B_{t\wedge\tau(\omega,\widehat{\omega})}(\widehat{\omega})\big)_{t\in[0,T]},
\end{equation}
where $\xi$ and $B$ are instances of the random objects in Assumption \ref{ass:1}, defined on another probability space $(\widehat{\Omega},\widehat{\mathcal{F}},\widehat{\PP})$, and
\begin{equation}
\tau(\omega,\widehat{\omega}):=\inf\big\{t\in[0,T]:\,D_t(\omega)\big(\xi(\widehat{\omega})+B_t(\widehat{\omega})\big)=0\big\}.
\end{equation}
\end{theorem}

The proof of Theorem \ref{thm:char} is provided in Subsection \ref{subse:proof}. It relies on the convergence of the absorption times involved, which, in turn, follows from the stability of a ``crossing property'' of Brownian motion in dimensions one and two, shown in Subsection \ref{subse:crossing}. The latter result is well-known in dimension one, but its analogue for a planar Brownian motion is new and constitutes the main technical contribution of this paper, interesting on its own. The proof of the stability of the crossing property explains why we restrict the scope of Theorem \ref{thm:char} to $d\in\{1,2\}$ and to Brownian limit points. 

\begin{rmk}\label{rmk:ext.GenGrowthModels.2}
As alluded to in Remark \ref{rmk:relax.Ass1}, Theorem \ref{thm:char} and its proof hold also for limiting particles that are time-changed Brownian motions, since the findings of Subsection~\ref{subse:crossing} hold for such processes. Moreover, similarly to the case of Proposition~\ref{prop:existence}, the proof of Theorem \ref{thm:char} does not utilize the full structure of external MDLA processes (cf.~Remark \ref{rmk:ext.GenGrowthModels}). In particular, provided the underlying particle system $\{\wt X^{i,N}\}_{i=1}^N$ satisfies Assumption \ref{ass:1} (or its aforementioned relaxation), Theorem \ref{thm:char} applies to any non-decreasing right-continuous family $\Gamma^N=\{\Gamma^N_t\}_{t\in[0,T]}$ of closed subsets of $\RR^d$, in which all points of every $\Gamma^N_t$ can be connected to $\Gamma^N_{0-}$ by a continuous curve in $\Gamma^N_t$, and to any family of processes $\{X^{i,N}\}_{i=1}^N$ whose paths are obtained from the respective paths of $\{\wt X^{i,N}\}_{i=1}^N$ by stopping them upon hitting $\Gamma^N$.
\end{rmk}

Among other things, Theorem \ref{thm:char} allows us to describe the limit points of external MDLA processes in terms of the associated PDEs. To this end, we fix an $\omega\in\Omega$ and consider the time-$t$ distributions of the canonical process on $C([0,T],\rr^d)$ under $\mu^X(\omega)$.~By Theorem \ref{thm:char}, those equal the time-$t$ distributions of a Brownian motion absorbed upon hitting $\Gamma(\omega)$, for $\PP$-almost every $\omega$. A comparison with the unabsorbed Brownian motion reveals that the restrictions of the latter time-$t$ distributions to $\RR^d\setminus\Gamma_t(\omega)$ admit density functions, which we denote by $u(t,\cdot)$, dropping the dependence on $\omega$ for brevity. A straightforward application of It\^o's formula then shows that $u$ is a weak solution (in the PDE sense) of the heat equation
\begin{equation}\label{eq.linHeatEq.u}
\partial_t u=\frac{1}{2}\Delta u\quad\text{on}\quad Q_T(\omega):=\{(t,x)\in (0,T]\times\RR^d:\,x\notin \Gamma_t(\omega)\}
\end{equation}
with the restriction of the law of $\xi$ to $\rr^d\backslash\Gamma_0(\omega)$ as the initial condition and a zero lateral boundary condition. In fact, Weyl's Lemma (see, e.g., \cite[p.~90, step 4]{McK}) yields $u\in C^\infty(Q_T(\omega),\rr)$ and that it fulfills \eqref{eq.linHeatEq.u} pointwise. We can also ensure that the initial condition holds classically. However, in general, the zero lateral boundary condition only holds in a weak sense. The described Cauchy-Dirichlet problem (or, equivalently, its probabilistic formulation in terms of an absorbed Brownian motion) uniquely determines $u$ for each $\Gamma(\omega)$. One might conjecture that, in addition, the limit points of external MDLA processes satisfy the Stefan free-boundary condition which uniquely determines $\Gamma(\omega)$, as in the case $d=1$ (see \cite{DNS19}). We address this conjecture in Subsection \ref{subse:DLA.vs.Stefan}.

\subsection{Crossing property and its stability} \label{subse:crossing}

Subsection \ref{subse:crossing} is devoted to the stability of the crossing property of Brownian motion, employed in the proof of Theorem \ref{thm:char} and valuable by itself. The precise result is provided below, in Proposition \ref{prop:stabCrossing} and Corollary \ref{cor:stabCrossing}. Since its proof is trivial for $d=1$, we focus on the case $d=2$ in most of the subsequent statements. 

\medskip

We begin with the notion of the \emph{winding number} and its connection to the first enclosing time of a point by a continuous curve on the plane.\footnote{The upcoming preliminaries on the winding number are essentially known. However, the authors are unaware of a reference containing the exact facts needed herein, and thus give detailed proofs.} To this end, for $x,y,z\in\RR^2$, with $x\neq z\neq y$, we write $\texttt{arg}(z,x;y)\in[-\pi,\pi)$ for the angular coordinate of $y$ in the polar coordinate system centered at $z$ and rotated so that $\texttt{arg}(z,x;x)=0$.

\begin{definition}\label{def:crossing.Wnd.def}
For a point $z=(z_1,z_2)\in\RR^2$, a continuous curve $\gamma\!:[a_1,a_2]\rightarrow\RR^2$, and times $a_1\le t_1<t_2\le a_2$ such that $z\notin\gamma_{[t_1,t_2]}:=\{\gamma(t):\,t\in[t_1,t_2]\}$, the winding number with respect to $z$ of the arc of $\gamma$ over $[t_1,t_2]$ is defined by
\begin{equation}\label{eq.crossing.Wnd.def.1}
\mathtt{Wnd}(z,\gamma;t_1,t_2)=\sum_{n=1}^\infty \theta^{(n)},
\end{equation}
where $\theta^{(n)}:=\mathtt{arg}\big(z,\gamma(\tau^{(n-1)});\gamma(\tau^{(n)})\big)$, $n\ge1;$ $\tau^{(0)}:=t_1;$ and, for $n\ge1$,
\begin{equation*}
\begin{split}
\tau^{(n)}\!\!:=\! \inf\Big\{t\!\in\![\tau^{(n-1)},t_2]\!:\mathtt{arg}\big(z,(z_1\!+\!1,z_2);\gamma(t)\big)\!\in\!\frac{\pi}{2}\ZZ,
\mathtt{arg}\big(z,\gamma(\tau^{(n-1)});\gamma(t)\big)\!\neq\!0\Big\}\!\wedge\! t_2.
\end{split}
\end{equation*}
\end{definition}

The uniform continuity of $\gamma$ and $z\notin\gamma_{[t_1,t_2]}$ imply that the terms of the sequence $(\tau^{(n)})_{n\in\nn}$ equal $t_2$ from some $n_0\in\nn$ on, and the series in \eqref{eq.crossing.Wnd.def.1} has a finite number of non-zero summands. It is also clear from the definition that $\texttt{Wnd}(z,\gamma;t_1,t_2)=\texttt{Wnd}(z,\gamma\circ\lambda^{-1};\lambda(t_1),\lambda(t_2))$ for any continuous increasing function $\lambda$ on $[t_1,t_2]$, i.e., the winding number is stable with respect to reparametrizations that preserve the curve orientation. In addition, the winding number with respect to $0$ of a curve $\gamma$ does not change if $\gamma$ is scaled by a positive multiple
, i.e., the winding number is stable with respect to space scaling
. Lastly, we note that, for any $\alpha\in(0,2\pi)$, the winding number $\texttt{Wnd}(z,\gamma;t_1,t_2)$ belongs to $[-\alpha,\alpha]$ when $\gamma_{[t_1,t_2]}$ is contained in a cone of angle $\alpha$ centered at $z$. We use these observations repeatedly. 

\medskip

Next, we recall that the winding number of a piecewise smooth closed curve on the plane can be defined in terms of a contour integral. The following lemma shows, in particular, that Definition \ref{def:crossing.Wnd.def} is consistent with the contour integral definition, and, in some sense, can be reduced to it even if $\gamma$ is not piecewise smooth, provided $\gamma|_{[t_1,t_2]}$ is closed. To state this result, we identify $\RR^2$ with the complex plane $\CC$.

\begin{lemma}\label{le:crossing.Wnd.vs.Int}
Let $\gamma\!:[a_1,a_2]\rightarrow\CC$ be a continuous curve whose arc over $[t_1,t_2]$ is closed, i.e., $\gamma(t_1)=\gamma(t_2)$, and let $z\in\CC$ be such that $z\notin\gamma_{[t_1,t_2]}$. Write $\{\tau^{(n)}\}_{n=0}^\infty$ for the stopping times associated with $(z,\gamma;t_1,t_2)$, as in Definition \ref{def:crossing.Wnd.def}. Then, there exists an $\varepsilon_0>0$, which can be chosen as a non-decreasing function of $\min_{t\in[t_1,t_2]}|\gamma(t)-z|$ such that, for any $\varepsilon\in(0,\varepsilon_0)$ and any piecewise smooth curve $\zeta:\,[t_1,t_2]\to\CC$ satisfying $\zeta(\tau^{(n)})=\gamma(\tau^{(n)})$, $n\ge0$ and $\max_{t\in[t_1,t_2]} |\zeta(t)-\gamma(t)|<\varepsilon$, we have
\begin{equation}\label{eq.crossing.WndLemma.eq0}
\mathtt{Wnd}(z,\gamma;t_1,t_2) = \frac{1}{i} \oint_{\zeta} \frac{1}{w-z}\,\mathrm{d}w.
\end{equation}
\end{lemma}

\begin{proof}
We assume without loss of generality that $z\!=\!0$, set $\varepsilon_0=\min_{t\in[t_1,t_2]}|\gamma(t)|/2>0$, and denote by $\partial O(\varepsilon_0)$ the circle of radius $\varepsilon_0$ around $0$. We then claim that, for any $\varepsilon\in(0,\varepsilon_0)$ and any $\zeta$ satisfying the conditions in the lemma, it holds
\begin{equation}\label{eq.crossing.WndLemma.eq1}
\int_{\tau^{(n-1)}}^{\tau^{(n)}} \frac{1}{\zeta(t)}\,\mathrm{d}\zeta(t) =
i\theta^{(n)} + \int_{\ell^{(n-1)}} \frac{1}{w}\,\mathrm{d}w - \int_{\ell^{(n)}} \frac{1}{w}\,\mathrm{d}w,\;\; n\geq 1,
\end{equation}
where each $\ell^{(n)}$ is the oriented segment of the line going through  $\gamma(\tau^{(n)})\!=\!\zeta(\tau^{(n)})$ and $0$ that connects the former point to $\partial O(\varepsilon_0)$. The claim is readily verified by considering the contour obtained by the orientation reversal of $\ell^{(n-1)}$ and its concatenation with $\zeta_{[\tau^{(n-1)},\tau^{(n)}]}$, $\ell^{(n)}$ and the appropriate arc of $\partial O(\varepsilon_0)$. Indeed, Cauchy's Integral Theorem implies that the integral of $w\mapsto\frac{1}{w}$ along this contour vanishes (because $0$ lies on the outside of the contour by the definition of $\tau^{(n)}$ and the choice of $\zeta$).

\medskip

Summing \eqref{eq.crossing.WndLemma.eq1} over $n=1,\,2,\,\ldots,\,n_0$, where $n_0$ is the smallest index such that $\tau^{(n)}=\tau^{(n_0)}=t_2$, $n\geq n_0$, and observing that $\ell^{(0)}$ and $\ell^{(n_0)}$ refer to the same oriented segment due to $\gamma(\tau^{(0)})=\gamma(t_1)=\gamma(t_2)=\gamma(\tau^{(n_0)})$, we deduce
\begin{equation*}
\int_{t_1}^{t_2} \frac{1}{\zeta(t)}\,\mathrm{d}\zeta(t) =
i \sum_{n=1}^{n_0} \theta^{(n)} = i\,\texttt{Wnd}(0,\gamma;t_1,t_2),
\end{equation*}
from which \eqref{eq.crossing.WndLemma.eq0} follows immediately. 
\end{proof}

The right-hand side of \eqref{eq.crossing.WndLemma.eq0} belongs to $2\pi\ZZ$ (see, e.g., \cite[Theorem~10.10]{Rudin}). Hence, since every closed arc of a continuous curve $\gamma$ admits a piecewise smooth approximation $\zeta$ as described in Lemma \ref{le:crossing.Wnd.vs.Int}, we conclude that the winding number of a closed arc of a continuous curve with respect to a point not on this arc is also an element of $2\pi\ZZ$.

\medskip

Next, we show the continuity of the winding number, used in the subsequent analysis.

\begin{lemma}\label{le:crossing.Wnd.cont}
Let $\gamma\!:[a_1,a_2]\rightarrow\CC$ be a continuous curve and let $z\in\CC\backslash\gamma_{[a_1,a_2]}$. Then, the function $s\mapsto\mathtt{Wnd}(z,\gamma;s,t)$ $($resp.~$t\mapsto\mathtt{Wnd}(z,\gamma;s,t)$$)$ is continuous on $[a_1,t)$ $($resp.~$(s,a_2]$$)$ for every $t\in(a_1,a_2]$ $($resp.~$s\in[a_1,a_2)$$)$. In addition, $\widehat{\gamma}\mapsto\mathtt{Wnd}(z,\widehat{\gamma};a_1,a_2)$ is continuous on the set $\{\widehat{\gamma}\in C([a_1,a_2],\CC)\!:\widehat{\gamma}(a_1)\!=\!\widehat{\gamma}(a_2),\,z\!\notin\!\widehat{\gamma}_{[a_1,a_2]}\}$ in the topology of uniform convergence.  
\end{lemma}
\begin{proof}
The first statement follows from the additivity of the winding number:
\begin{equation}\label{additivity}
\texttt{Wnd}(z,\gamma;t_1,t_2) + \texttt{Wnd}(z,\gamma;t_2,t_3) = \texttt{Wnd}(z,\gamma;t_1,t_3),\;\; a_1\leq t_1<t_2<t_3\leq a_2.
\end{equation}
The above additivity is verified directly from Definition \ref{def:crossing.Wnd.def} for sufficiently small $t_3-t_2$, such that $\gamma_{[t_2,t_3]}$ is contained in a cone of angle $\pi/2$ centered at $z$. The general case, then, follows by partitioning $[t_1,t_3]$ into sufficiently small intervals.

\medskip

The continuity of $\texttt{Wnd}(z,\cdot\,;a_1,a_2)$ results from Lemma \ref{le:crossing.Wnd.vs.Int}.~Indeed, take $(\gamma^{(k)})_{k\in\nn}$, $\gamma^{(\infty)}$ in $\{\widehat{\gamma}\in C([a_1,a_2],\CC)\!:\widehat{\gamma}(a_1)\!=\!\widehat{\gamma}(a_2),\,z\!\notin\!\widehat{\gamma}_{[a_1,a_2]}\}$ with $\max_{[a_1,a_2]} |\gamma^{(k)}-\gamma^{(\infty)}|\to0$ as $k\rightarrow\infty$. For every $k\in\nn$, we choose a finite interval partition $\Pi^{(k)}$ of $[a_1,a_2]$ whose break points include the stopping times $\{\tau^{(k,n)}\}_{n=0}^\infty$ associated with $\gamma^{(k)}$ per Definition~\ref{def:crossing.Wnd.def}, as well as the corresponding stopping times $\{\tau^{(\infty,n)}\}_{n=0}^\infty$ associated with~$\gamma^{(\infty)}$. Moreover, let $\zeta^{(k)}$ be the closed continuous curve interpolating linearly, on the intervals of~$\Pi^{(k)}$, between the values of $\gamma^{(k)}$. We can ensure that the mesh of $\Pi^{(k)}$ tends to $0$ as $k\rightarrow\infty$ and that it holds
\begin{equation}\label{eq.crossing.Wnd.cont.Pf.eq1}
\texttt{Wnd}(z,\gamma^{(k)};a_1,a_2) = \frac{1}{i} \oint_{\zeta^{(k)}} \frac{1}{w-z}\,\mathrm{d}w,\;\;k\in\nn
\end{equation}
(see Lemma \ref{le:crossing.Wnd.vs.Int}). 
If needed, we refine $\Pi^{(k)}$, without altering $\zeta^{(k)}$, so that there exists a break point of $\Pi^{(k)}$ between any two neighboring $\tau^{(\infty,n)}$ and $\tau^{(\infty,n')}$. Hereby, since all curves involved are closed, we treat the time domain $[a_1,a_2]$ as a circle (after identifying $a_1$ with $a_2$) when determining which break points of $\Pi^{(k)}$ are neighboring. 

\medskip

Further, for every $k\in\nn$, we define $\widetilde{\zeta}^{(k)}$ by interpolating linearly between the values of $\gamma^{(\infty)}$ at the points $\{\tau^{(\infty,n)}\}_{n=0}^\infty$ and the values of $\zeta^{(k)}$ at the other break points of~$\Pi^{(k)}$. Then, upon associating $\varepsilon_0\!>\!0$ with $\gamma^{(\infty)}$ via Lemma \ref{le:crossing.Wnd.vs.Int}, $\max_{[a_1,a_2]} |\widetilde{\zeta}^{(k)}-\gamma^{(\infty)}|<\varepsilon_0$ for all $k\in\nn$ large enough, and for such $k$,
\begin{equation}\label{eq.crossing.Wnd.cont.Pf.eq2}
\texttt{Wnd}(z,\gamma^{(\infty)};a_1,a_2) = \frac{1}{i} \oint_{\widetilde{\zeta}^{(k)}} \frac{1}{w-z}\,\mathrm{d}w
\end{equation}
by Lemma \ref{le:crossing.Wnd.vs.Int}. Note that $\zeta^{(k)}$ and $\widetilde{\zeta}^{(k)}$ differ only on the time intervals bounded by a $\tau^{(\infty,n)}$ and a neighboring break point of $\Pi^{(k)}$. We write $I^{(k,n)}_-$ and $I^{(k,n)}_+$ for these intervals. The concatenation of the arc of $\zeta^{(k)}$ over $I^{(k,n)}_-\cup I^{(k,n)}_+$ with the orientation-reversed arc of $\widetilde{\zeta}^{(k)}$ over $I^{(k,n)}_-\cup I^{(k,n)}_+$ leads to a closed contour along which the function $w\mapsto\frac{1}{w-z}$ integrates to $0$, for all $k\in\nn$ large enough, by Cauchy's Integral Theorem (recall that $z\notin\gamma^{(\infty)}_{[a_1,a_2]}$). Thus, the right-hand sides of \eqref{eq.crossing.Wnd.cont.Pf.eq1} and \eqref{eq.crossing.Wnd.cont.Pf.eq2} coincide for all $k\in\nn$ large enough, implying the same for the left-hand sides. 
\end{proof}

The additivity property \eqref{eq.crossing.Wnd.cont.Pf.eq2} reveals that the winding number with respect to $z$ of a continuous curve $\gamma$ does not change if $\gamma$ is rotated around $z$. This can be seen by splitting $\gamma$ into arcs, each contained in a cone of angle less than $\frac{\pi}{2}$ centered at $z$, and obtaining the rotational invariance of the winding numbers of these arcs directly from Definition \ref{def:crossing.Wnd.def}. It is also worth recording that the continuous functional $\texttt{Wnd}(z,\cdot\,;a_1,a_2)$ with values in $2\pi\ZZ$ must be constant in a neighborhood of every closed continuous curve on $[a_1,a_2]$ that does not pass through $z$.

\medskip

We can now turn our attention to the connection between the winding number and the question of whether a point is enclosed by a curve. The following corollary of Lemma \ref{le:crossing.Wnd.vs.Int} shows that, if the winding number with respect to $z$ of a closed arc of a continuous curve is non-zero, then $z$ is enclosed by this arc.

\begin{corollary}\label{cor:crossing.Wnd.vsClosing}
Let $\gamma\!:[a_1,a_2]\rightarrow\RR^2$ be a continuous curve whose arc over $[t_1,t_2]$ is closed, i.e., $\gamma(t_1)=\gamma(t_2)$, and let $z\in\RR^2$ be so that $z\notin\gamma_{[t_1,t_2]}$. If $\,\mathtt{Wnd}(z,\gamma;t_1,t_2)\neq0$, then $z$ belongs to a bounded connected component of $\RR^2\setminus \gamma_{[t_1,t_2]}$. 
\end{corollary}

\begin{proof}
We argue by contradiction and assume that $z$ belongs to the unbounded connected component of $\RR^2\setminus \gamma_{[t_1,t_2]}$. Then, there exists a semi-infinite continuous path $\lambda$ connecting $z$ to infinity without intersecting $\gamma_{[t_1,t_2]}$. Next, we consider any $\zeta$ satisfying the conditions of Lemma \ref{le:crossing.Wnd.vs.Int} that is sufficiently close to $\gamma$, so that $\zeta_{[t_1,t_2]}$ does not intersect $\lambda$. Then, $z$ is in the unbounded connected component of $\RR^2\setminus \zeta_{[t_1,t_2]}$, and the right-hand side of \eqref{eq.crossing.WndLemma.eq0} vanishes by \cite[Theorem~10.10]{Rudin}. Therefore, $\texttt{Wnd}(z,\gamma;t_1,t_2)=0$ by Lemma \ref{le:crossing.Wnd.vs.Int}, which gives us the desired contradiction.
\end{proof}

Motivated by Corollary \ref{cor:crossing.Wnd.vsClosing}, we introduce the following definition.

\begin{definition}\label{def:crossing.FirstEnclosingTime}
The first enclosing time of a point $z\in\RR^2$ by the arc over $[t_1,t_2]$ of a continuous curve $\gamma:\,[a_1,a_2]\rightarrow\RR^2$, with $z\notin\gamma_{[t_1,t_2]}$, is defined by
$$
\mathtt{T}_c(z,\gamma;t_1,t_2):=\inf\{t\in[t_1,t_2]:\,\exists\,s\in[t_1,t)\text{ with }\gamma(s)=\gamma(t)\text{ and }\,\mathtt{Wnd}(z,\gamma;s,t)\neq0\}.
$$
\end{definition}

The continuity of $\gamma$ and $\texttt{Wnd}(z,\gamma;\cdot\,,\cdot\,)$ (cf.~\eqref{additivity}) implies that, if $\texttt{T}_c(z,\gamma;t_1,t_2)<\infty$ (as usual, we set $\inf\varnothing=\infty$), then
\begin{align*}
\exists\,s\in[t_1,\texttt{T}_c(z,\gamma;t_1,t_2)):\,\gamma(s)=\gamma(\texttt{T}_c(z,\gamma;t_1,t_2))
\text{ and }\texttt{Wnd}(z,\gamma;s,\texttt{T}_c(z,\gamma;t_1,t_2))\neq0.
\end{align*}

\begin{rmk}\label{rem:crossing.Wnd.vsClosing}
The latter observation together with Corollary \ref{cor:crossing.Wnd.vsClosing} reveal that, whenever $\texttt{T}_c:=\texttt{T}_c(z,\gamma;t_1,t_2)<\infty$, there is an $s\in[t_1,\texttt{T}_c)$ such that $\gamma(s)=\gamma(\texttt{T}_c)$ and $z$ lies in a bounded connected component of $\RR^2\setminus \gamma_{[s,\texttt{T}_c]}$.
\end{rmk}

We can now proceed to the crossing property of planar Brownian motion and to the stability of this property. The next lemma is similar to \cite[Section VII.1, Lemma 1]{LeGall}. It shows that almost every path of a planar Brownian motion immediately encloses its initial point. The main difference between our lemma and \cite[Section VII.1, Lemma~1]{LeGall} is in the definition of the first enclosing time. Herein, we require a non-zero winding number to declare that a point is enclosed (see Definition \ref{def:crossing.FirstEnclosingTime}), while \cite[Section VII.1, Lemma~1]{LeGall} labels a point as enclosed when it lies in a bounded connected component of the complement of the curve.  The former implies the latter by Corollary~\ref{cor:crossing.Wnd.vsClosing}. In the remainder of this subsection, we denote by $B^z$ the planar Brownian motion started from $z\in\RR^2$, constructed on a probability space with a probability measure $\PP$, and let
\begin{equation}
\tau^z_r:=\inf\{t\ge0:\,|B^z_t-z|=r\},\;\; r>0.
\end{equation}
We also regard the realizations of $B^z$ as curves in $\RR^2$ whenever needed.

\begin{lemma}\label{le:crossing}
For all $z\in\RR^2$ and $\delta>0$, it holds\footnote{The measurability of the event $\{\texttt{T}_c(z,B^z;\tau^z_{\delta'},\tau^z_{\delta})<\infty\}$ follows from the measurability of the winding number, which, in turn, can be verified directly from Definition \ref{def:crossing.Wnd.def}.}
\begin{equation}
\lim_{\delta'\downarrow0}\,\PP(\mathtt{T}_c(z,B^z;\tau^z_{\delta'},\tau^z_{\delta})<\infty)=1.
\end{equation}
\end{lemma}

\begin{proof}
By the translation invariance of Brownian motion we may assume that $z=0$. Further, the strong Markov property and the scale invariance of Brownian motion yield
\begin{equation*}
\begin{split}
\PP(\texttt{T}_c(0,B^0;\tau^0_{\delta/2^n},\tau^0_\delta)=\infty)
&\leq \PP\bigg(\bigcap_{m=0}^{n-1} \{\texttt{T}_c(0,B^0;\tau^0_{\delta/2^{m+1}},\tau^0_{\delta/2^m})=\infty\}\!\bigg) \\
&= \PP(\texttt{T}_c(0,B^0;\tau^0_1,\tau^0_2)=\infty)^n,\;\;n\ge1. 
\end{split}
\end{equation*}
Thus, it suffices to prove that $\PP(\texttt{T}_c(0,B^0;\tau^0_1,\tau^0_2)=\infty)<1$.

\medskip

\begin{figure}
	\begin{center}
		\begin{tabular} {cc}
			{
				\includegraphics[width = 0.45\textwidth]{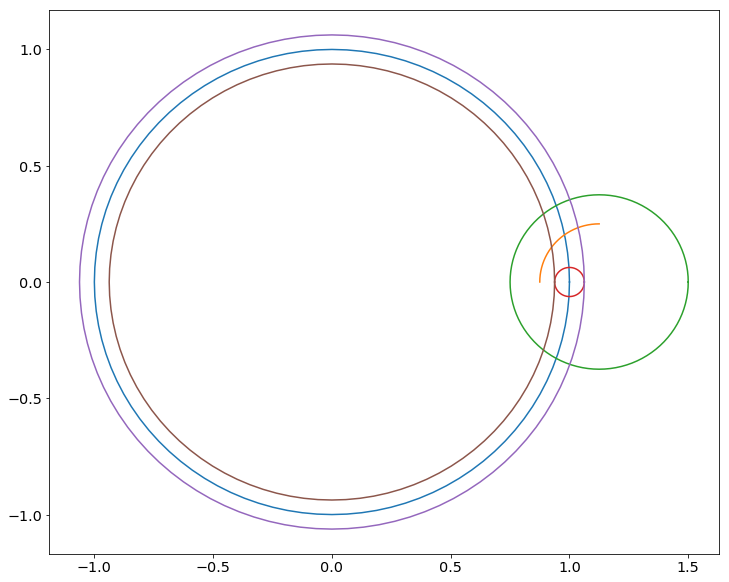}
			} & {
				\includegraphics[width = 0.45\textwidth]{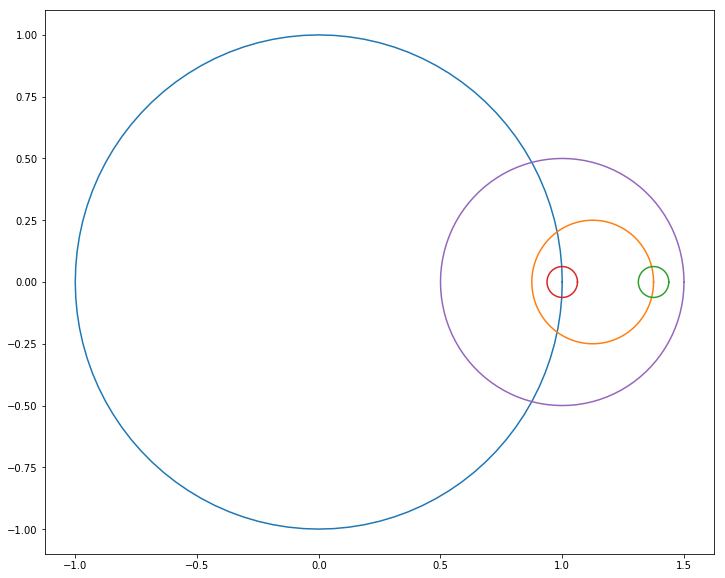}
			}\\
		\end{tabular}
		\caption{The left panel visualizes the proof of Lemma \ref{le:crossing}: $z=0$, the blue circle is $\widetilde{\gamma}_{[0,1]}$, the orange arc is $\widehat{\gamma}_{[2,3]}$, the green circle is $\partial O_1$, the red circle is $\partial O_2$, the purple circle is $\partial O_3$, and the brown circle is $\partial O_4$. The right panel visualizes the proof of Lemma \ref{le:crossing.mainProp.pf.lastLemma}: $z=0$, $B^z_{t_0}=B^z_{\tau}=(1,0)$, $B^z_{t_1}=B^z_{t_2}=(11/8,0)$, the blue circle represents $B^z_{[t_0,\tau]}$ (not actually a circle in the proof), the orange circle represents $B^z_{[t_1,t_2]}$ (not actually a circle in the proof), the green circle is $\partial O_1$, the red circle is $\partial O_2$, and the purple circle is $\partial O_3$.}
		\squeezeup		
		\label{fig:1}
	\end{center}
\end{figure}

A visualization of the remainder of the proof is given by the left panel of Figure~\ref{fig:1}. For convenience, we take $B^0_{\tau^0_1}=(1,0)\in\RR^2$. We further introduce a continuously differentiable curve $\gamma\!:[0,3]\to\RR^2$ obtained by the concatenation of $\widetilde{\gamma}\!:[0,1]\to\RR^2$ and $\widehat{\gamma}\!:[1,3]\to\RR^2$, where 
\begin{itemize}
\item $\widetilde{\gamma}(t)=(\cos(2\pi t),\sin(2\pi t))$, $t\in[0,1]$,
\item $\widehat{\gamma}_{(1,2)}\subset O_1$, with $O_1$ being the open ball of radius $3/8$ around $(9/8,0)$, 
\item $\widehat{\gamma}(t) = (9/8+\cos(\pi(t-1)/2)/4,\sin(\pi(t-1)/2)/4)$, $t\in[2,3]$.
\end{itemize}
We note that $\widehat{\gamma}(2)=(9/8,1/4)$; $\widehat{\gamma}(3)=(7/8,0)$; $\widehat{\gamma}_{[2,3]}\cap O_2=\varnothing$, where $O_2$ is the open ball of radius $1/16$ around $(1,0)$; and $\widehat{\gamma}_{[1,3]}\subset O_1$. 

\medskip

Let us show that $\gamma$ admits a neighborhood in the topology of uniform convergence such that any curve therein encloses $0$. To this end, we consider the curve $\overline{\gamma}$ given by the constant extension of $\widetilde{\gamma}$ to the time interval $[0,2]$, i.e., $\overline{\gamma}(t)=\widetilde{\gamma}(t)$, $t\in[0,1]$ and $\overline{\gamma}(t)=\widetilde{\gamma}(1)$, $t\in(1,2]$. Lemma \ref{le:crossing.Wnd.cont} yields the existence of an $\varepsilon_1>0$ such that  $\texttt{Wnd}(0,\overline{\zeta};0,2)=\texttt{Wnd}(0,\overline{\gamma};0,2)=2\pi$ for any closed continuous curve $\overline{\zeta}\!:[0,2]\to\RR^2$ with $\max_{t\in[0,2]} |\overline{\zeta}(t)-\overline{\gamma}(t)|<\varepsilon_1$. Then, for all $\varepsilon_2\in(0,\varepsilon_1\wedge(1/16))$ small enough and any continuous curve $\zeta\!:[0,3]\to\RR^2$ with $\max_{t\in[0,3]} |\zeta(t)-\gamma(t)|<\varepsilon_2$, we have: 
\begin{itemize}
\item $\zeta(2)\notin O_3$, where $O_3$ is the open ball of radius $17/16$ around $0$; 
\item $|\zeta(3)|<15/16$;
\item $\zeta_{[2,3]}\cap O_2=\varnothing$; 
\item $\zeta_{[1,3]}\subset O_1$; 
\item $\zeta(0),\zeta(1)\in O_2$;
\item $\zeta_{[0,1]}\cap O_4=\varnothing$, where $O_4$ is the open ball of radius $15/16$ around $0$;
\item $\zeta_{[1/4,1]}\cap\zeta_{[2,3]}=\varnothing$;
\item $\zeta_{[0,1/4]}$ and $\zeta_{[1,3]}$ are each contained in a cone of angle $5\pi/8$ centered at $0$.
\end{itemize}

\smallskip

For any $\zeta$ as described, we let $\overline{\zeta}\!:[0,2]\to\RR^2$ be an extension of the arc of $\zeta$ over $[0,1]$ which linearly interpolates between $\zeta(1)$ and $\zeta(0)$ on the time interval $[1,2]$. Note that $\overline{\zeta}$ is a closed continuous curve in the $\varepsilon_1$-neighborhood of $\overline{\gamma}$, thus, $\texttt{Wnd}(0,\overline{\zeta};0,2)=2\pi$, and $0$ lies in a bounded connected component of $\RR^2\setminus\overline{\zeta}_{[0,2]}$ by Corollary~\ref{cor:crossing.Wnd.vsClosing}. Moreover, the line segment connecting $0$ and $\zeta(3)$ belongs to $O_4$ and, hence, cannot intersect $\overline{\zeta}_{[0,2]}$, so that $\zeta(3)$ lies in a bounded connected component of $\RR^2\setminus\overline{\zeta}_{[0,2]}$ as well.~On the other hand, $\zeta(2)\notin O_3$ while $\overline{\zeta}_{[0,2]}\subset O_3$, which implies that $\zeta(2)$ lies in the unbounded connected component of $\RR^2\setminus\overline{\zeta}_{[0,2]}$. Thus, $\zeta_{[2,3]}$ intersects $\overline{\zeta}_{[0,2]}$. As $\zeta(0),\zeta(1)\in O_2$ and $\zeta_{[2,3]}\cap O_2=\varnothing$, we deduce that $\zeta_{[2,3]}\cap\overline{\zeta}_{[1,2]}=\varnothing$. Consequently, $\zeta_{[2,3]}$ intersects $\overline{\zeta}_{[0,1]}=\zeta_{[0,1]}$. Recalling $\zeta_{[1/4,1]}\cap\zeta_{[2,3]}=\varnothing$, we conclude that $\zeta_{[2,3]}$ intersects $\zeta_{[0,1/4]}$. We write $u_2\in[2,3]$ and $u_1\in[0,1/4]$ for two time coordinates of the aforementioned intersection point. Since $\texttt{Wnd}(0,\overline{\zeta};0,2)=2\pi$, the additivity of the winding number (see~\eqref{additivity}) and the fact that $\overline{\zeta}_{[1,2]}$ is contained in a cone of angle $\pi/2$ centered at $0$ yield $\texttt{Wnd}(0,\zeta;0,1)=\texttt{Wnd}(0,\overline{\zeta};0,1)\geq 3\pi/2$.
Finally, $\zeta_{[0,u_1]}$ and $\zeta_{[1,u_2]}$ are each contained in a cone of angle $5\pi/8$ centered at $0$, and we infer that $\texttt{Wnd}(0,\zeta;u_1,u_2)\geq \pi/4$. This proves $\texttt{T}_c(0,\zeta;0,3)<\infty$ for any continuous curve $\zeta$ in the $\varepsilon_2$-neighborhood of $\gamma$.

\medskip

More generally, for $z'\in\RR^2$ with $|z'|=1$, we rotate $\gamma$ by $\texttt{arg}(0,(1,0);z')$ to get~$\gamma^{z'}$. Then, conditionally on $B^0_{\tau^0_1}=z'$, the curve $[0,3]\to\RR^2$, $t\mapsto B^0_{\tau^0_1+t}$ falls into the $\varepsilon_2$-neighborhood of $\gamma^{z'}$ with a positive probability independent of $z'$, as can be seen from Girsanov's Theorem and the rotational invariance of Brownian motion. Hence, the conditional probability of $\{\texttt{T}_c(0,B^0;\tau^0_1,\tau^0_1+3)=\texttt{T}_c(0,B^0_{\tau^0_1+\cdot};0,3)<\infty\}\cap\{\tau_2^0>\tau^0_1+3\}$ given $B^0_{\tau^0_1}=z'$ admits a positive lower bound independent of $z'$. We conclude that $\PP(\texttt{T}_c(0,B^0;\tau^0_1,\tau^0_2)=\infty)<1$, as desired.
\end{proof}

Lemma \ref{le:crossing} has a simple but useful corollary.

\begin{corollary}\label{cor:crossing.cor1}
Fix any $z\in\RR^2$, $\delta>0$. Then, for almost every Brownian path $B^z$, there is a bounded open neighborhood $U=U(\delta,B^z)$ of $z$ such that $\partial U\subset B^z_{[0,\tau^z_\delta]}$.
\end{corollary}
\begin{proof}
Lemma \ref{le:crossing} implies that, for almost every Brownian path $B^z$, there exists a $\delta'\in(0,\delta)$ for which $\texttt{T}_c(z,B^z;\tau^z_{\delta'},\tau^z_{\delta})<\infty$.
Remark \ref{rem:crossing.Wnd.vsClosing} then yields the existence of a bounded open neighborhood $U$ of $z$ such that $\partial U\subset B^z_{[\tau^z_{\delta'},\tau^z_\delta]}\subset B^z_{[0,\tau^z_\delta]}$.
\end{proof}

The fact that each point of any $\Gamma^N_t$ can be connected to some other point in $\Gamma^N_t$ by a continuous curve, the strong Markov property of Brownian motion and Corollary \ref{cor:crossing.cor1} reveal that almost every path of a planar Brownian motion intersects $\Gamma^N$ infinitely often in each open right neighborhood of its first hitting time. The proof of Corollary~\ref{cor:stabCrossing} shows how to obtain the same conclusion with $\Gamma$ in place of $\Gamma^N$. We therefore refer to Corollary \ref{cor:crossing.cor1} as the crossing property of planar Brownian motion, in analogy to the terminology of \cite[Lemma 5.6]{DIRT2} in the one-dimensional setting. To prove Theorem~\ref{thm:char}, we need to additionally verify that the crossing property is stable, i.e., that it holds for any continuous path close to a Brownian path, as captured by the following proposition.  

\begin{proposition}\label{prop:stabCrossing}
Fix any $z\in\RR^2$, $\delta>0$. Then, for almost every Brownian path~$B^z$, there is an $\varepsilon\!=\!\varepsilon(\delta,B^z)\!>\!0$ such that, for all $b\!\in\! C([0,\tau^z_\delta],\RR^2)$ with $\max_{[0,\tau^z_\delta]} |b-B^z|\!<\!\varepsilon$, there exists a bounded open neighborhood $V=V(\delta,B^z,b)$ of $z$ such that $\partial V\subset b_{[0,\tau^z_\delta]}$. 
\end{proposition}

\begin{proof}
By Remark \ref{rem:crossing.Wnd.vsClosing}, it is enough to establish the subsequent claim.

\medskip

\noindent\textbf{Claim.} For almost every Brownian path $B^z$, there is an $\varepsilon=\varepsilon(\delta,B^z)>0$ such that, for all $b\in C([0,\tau^z_\delta],\RR^2)$ with $\max_{t\in[0,\tau^z_\delta]} |b(t)-B^z_t|<\varepsilon$, there exists a $\delta'\in(0,\delta)$ such that $\texttt{T}_c(z,b;\tau^z_{\delta'},\tau^z_{\delta})<\infty$.

\medskip

\noindent\textbf{Step 1.} For $\delta_0\in(0,\delta)$, let 
$$
E_{\delta_0}:= \{\texttt{T}_c(z,B^z;\tau^z_{\delta_0},\tau^z_{\delta})<\infty\}
\cap\{z\notin B^z_{(0,\tau^z_\delta]}\}.
$$
The non-increasing family $\{E_{\delta_0}\}_{\delta_0\in(0,\delta)}$ tends to an event of probability one as $\delta_0\downarrow0$, thanks to Lemma \ref{le:crossing}. Thus, it suffices to prove the claim on $E_{\delta_0}$, for an arbitrary fixed $\delta_0\in(0,\delta)$. To ease the notation, we set
$$
\tau=\lim_{n\rightarrow\infty}\texttt{T}_c(z,B^z;\tau^z_{\delta_0},n).
$$
The latter limit is well-defined and almost surely finite, due to the scale invariance of Brownian motion and Lemma \ref{le:crossing}. We also remark that $\tau^z_{\delta_0}<\tau<\tau^z_\delta$ on $E_{\delta_0}$. (Observe that, by definition, $\tau^z_\delta$ cannot be a self-intersection time of a Brownian path.) Further, we decrease $E_{\delta_0}$ by a zero probability event according to 
$$
E'_{\delta_0}:=E_{\delta_0}\cap\{B^z_\tau\notin B^z_{(\tau,\infty)}\}.
$$

\smallskip

Next, we take a $\delta_1>0$, expand the underlying probability space so it supports a uniform $(0,1)$-valued random variable $\eta$ independent of $B^z$, and let
\begin{align*}
&\sigma_1:=\inf\{t\geq0:\,|B^z_{\tau+t}-B^z_\tau|\geq \delta_1\eta\},\\
&\iota:=\lim_{n\rightarrow\infty} \texttt{T}_c(B^z_{\tau},B^z_{\tau+\cdot};\sigma_1,n).
\end{align*}
The strong Markov property of Brownian motion, its scale invariance and Lemma \ref{le:crossing} show that $\iota$ is well-defined and almost surely finite. 

\medskip

We claim that, almost surely, the conditional distribution of $B^z_{\tau+\iota}$ given $B^z_{[0,\tau]}$ is absolutely continuous with respect to the Lebesgue measure. The latter follows from the strong Markov property of Brownian motion and the observation that, for any $z'\in\RR^2$ and with $B^{z'}:=B^z-z+z'$, the pairs 
$$
\big(\eta,\widehat{B}^{z'} := z'+\eta(B^{z'}_{\cdot/\eta^2}-z')\big)\quad\text{and}\quad
(\eta,B^{z'})
$$
have the same distribution.
Indeed, the conditional distribution of $(B^z_{\tau+\cdot},\iota)$ given $B^z_{[0,\tau]}$ almost surely coincides with the distribution of $(\widehat{B}^{z'},\widehat{\iota}^{z'})$ evaluated at $z'=B^z_\tau$, where
\begin{align*}
&\widehat{\iota}^{z'}
:=\lim_{n\rightarrow\infty} \texttt{T}_c(z',\widehat{B}^{z'};\widehat{\sigma}^{z'}_1,n) 
=\lim_{n\rightarrow\infty} \texttt{T}_c(z',B^{z'}_{\cdot/\eta^2};\widehat{\sigma}^{z'}_1,n)
= \eta^2 \lim_{n\rightarrow\infty} \texttt{T}_c(z',B^{z'};\widehat{\sigma}^{z'}_1/\eta^2,n),\\
&\widehat{\sigma}_1^{z'}:=\inf\{t\geq0:\, |\widehat{B}^{z'}_{t}-z'|\geq \delta_1\eta\}
= \inf\{t\geq0:\,|B^{z'}_{t/\eta^2}-z'|\geq \delta_1\}\\
&\qquad\qquad\qquad\qquad\qquad\qquad\qquad\quad\, = \eta^2\inf\{t\geq0:\,|B^{z'}_{t}-z'|\geq \delta_1\}.
\end{align*}
These imply $\widehat{\iota}^{z'} = \eta^2\widehat{\tau}^{z'}$, with
\begin{align*}
\widehat\tau^{z'}:=\lim_{n\rightarrow\infty} \texttt{T}_c(z',B^{z'};\widetilde{\sigma}^{z'}_1,n)\quad\text{and}\quad
\widetilde{\sigma}^{z'}_1:=\inf\{t\geq0:\,|B^{z'}_{t}-z'|\geq\delta_1\}.
\end{align*}
Thus, the conditional distribution of $B^z_{\tau+\iota}$ given $B^z_{[0,\tau]}$ almost surely coincides with the distribution of $z'+\eta(B^{z'}_{\widehat{\tau}^{z'}}-z')$ evaluated at $z'=B^z_\tau$.
It is only left to notice that the distribution of $B^{z'}_{\widehat{\tau}^{z'}}$ is invariant with respect to rotations around $z'$ and that $\eta$ is independent of $B^{z'}_{\widehat{\tau}^{z'}}$, to conclude that, almost surely, the conditional distribution of $B^z_{\tau+\iota}$ given $B^z_{[0,\tau]}$ is absolutely continuous with respect to the Lebesgue measure. 

\medskip

\noindent \textbf{Step 2.} We proceed by letting
\begin{eqnarray*}
&& \sigma_2:=\inf\big\{t\geq0:\,|B^z_{\tau+t}-B^z_\tau|\geq \big(|B^z_\tau-z|\wedge(\delta-|B^z_\tau-z|)\big)/2\big\}, \\
&& E_{\delta_0,\delta_1}:=E'_{\delta_0}\cap\{\iota<\sigma_2\}. 
\end{eqnarray*}
Using the strong Markov property of Brownian motion and Lemma \ref{le:crossing} we deduce that the non-increasing family $\{E_{\delta_0,\delta_1}\}_{\delta_1>0}$ tends to $E'_{\delta_0}$ as $\delta_1\downarrow0$, up to a zero probability event. 
It follows that
$$
\PP\bigg(\bigcup_{n=1}^\infty\bigcup_{m=1}^\infty E_{1/n,1/m}\bigg)=1.
$$
Recall also that, almost surely, the set $B^z_{[0,\tau]}$ has zero Lebesgue measure. Combining this with the conclusion of Step 1 we infer that, almost surely, $B^z_{\tau+\iota} \notin B^z_{[0,\tau]}$. Putting these together with the observation made after Definition \ref{def:crossing.FirstEnclosingTime}, we see that, almost surely, there exist $\delta_0,\delta_1>0$ such that $\tau^z_{\delta_0}<\tau<\tau+\sigma_2<\tau^z_\delta$ and $\iota<\sigma_2$, as well as $t_0,s_0,t_1,t_2$  satisfying $\tau_{\delta_0}^z<t_0<s_0<\tau<t_1<t_2<\tau+\sigma_2$ and
\begin{eqnarray}
&&\quad\;\; B^z_{t_0}=B^z_\tau,\quad\texttt{Wnd}(z,B^z;t_0,\tau) \neq 0, 
\label{eq.crossing.mainProp.Pf.finalPropOfB.eq1} \\
&&\quad\;\; |B^z_{t_0}-z|/2<|B^z_{s_0}-B^z_{t_0}|< \frac{7}{8}|B^z_{t_0}-z|,\quad
|B^z_t-B^z_{t_0}|<\frac{7}{8}|B^z_{t_0} - z|,\;\; t\in[t_0,s_0],
\label{eq.crossing.mainProp.Pf.finalPropOfB.eq2} \\
&&\quad\;\; B^z_{t_1}=B^z_{t_2},\quad \texttt{Wnd}(B^z_\tau,B^z;t_1,t_2) \neq 0, 
\label{eq.crossing.mainProp.Pf.finalPropOfB.eq3} \\
&&\quad\;\; |B^z_t-B^z_\tau|<|B^z_\tau - z|/2,\;\; t\in[\tau,t_2],\label{eq.crossing.mainProp.Pf.finalPropOfB.eq4} \\
&&\quad\;\; B^z_\tau\notin B^z_{[t_1,t_2]},\quad B^z_{t_2}\notin B^z_{[0,\tau]}.\label{eq.crossing.mainProp.Pf.finalPropOfB.eq5}
\end{eqnarray}
Indeed, a $t_0\in(\tau^z_{\delta_0},\tau)$ fulfilling \eqref{eq.crossing.mainProp.Pf.finalPropOfB.eq1} exists by the definition of $\tau$. The time $s_0\in(t_0,\tau)$ can be chosen as $\inf\{t>t_0\!:|B^z_t-B^z_{t_0}|=3|B^z_{t_0}-z|/4\}$, which is strictly less than~$\tau$ because $B^z_{[t_0,s_0]}$ is then contained in a cone of angle $\pi$ centered at $z$. The time $t_2$ can be taken as $\tau+\iota$; and $t_1\in(\tau,t_2)$ satisfying \eqref{eq.crossing.mainProp.Pf.finalPropOfB.eq3} then exists by the definition of $\iota$. In the described context, we proceed to the following lemma. 

\begin{lemma}\label{le:crossing.mainProp.pf.lastLemma}
There exists an $\varepsilon_1>0$ such that, for any continuous $b\!:[\tau^z_{\delta_0},\tau^z_\delta]\to\RR^2$ with $\max_{t\in[\tau^z_{\delta_0},\tau^z_\delta]}|b(t)-B^z_t|<\varepsilon_1$, it holds $b_{[t_0,s_0]}\cap b_{[t_1,t_2]}\neq\varnothing$.
\end{lemma}

Let us show how Lemma \ref{le:crossing.mainProp.pf.lastLemma} yields the claim at the beginning of the proof. Denote by $\gamma$ the extension of $B^z$ from $[t_0,\tau]$ to $[t_0,\tau+1]$ defined by $\gamma(t)=B^z_\tau$, $t\in(\tau,\tau+1]$. Then, $\gamma$ is a closed continuous curve and $\texttt{Wnd}(z,\gamma;t_0,\tau+1) =\texttt{Wnd}(z,B^z;t_0,\tau)\neq0$. By Lemma \ref{le:crossing.Wnd.cont}, there exists an $\varepsilon_2>0$ such that $\texttt{Wnd}(z,\widetilde{\gamma};t_0,\tau+1) \neq 0$ for any closed continuous curve $\widetilde{\gamma}$ in the $\varepsilon_2$-neighborhood of $\gamma$. Fixing such an $\varepsilon_2>0$, for any continuous $b\!:[\tau^z_{\delta_0},\tau^z_\delta]\to\RR^2$ with $\max_{t\in[\tau^z_{\delta_0},\tau^z_\delta]} |b(t)-B^z_t|<\varepsilon_1\wedge\varepsilon_2\wedge (|B^z_{t_0}-z|/8)$, we extend $b$ from $[t_0,\tau]$ to $[t_0,\tau+1]$ by linearly interpolating between $b(\tau)$ and $b(t_0)$ on $[\tau,\tau+1]$. We write $\widetilde{b}$ for the resulting extension. Then, $\widetilde{b}$ is a closed continuous curve which belongs to the $\varepsilon_2$-neighborhood of $\gamma$ and, hence, $\texttt{Wnd}(z,\widetilde{b};t_0,\tau+1) \neq 0$. Since the winding number is additive (see \eqref{additivity}) and $\widetilde{b}_{[\tau,\tau+1]}$ is contained in a cone of angle $\pi/2$ centered at $z$, we conclude that $\texttt{Wnd}(z,b;t_0,\tau) \notin [-3\pi/2,3\pi/2]$. Next, we apply Lemma \ref{le:crossing.mainProp.pf.lastLemma} and denote by $u_1\in[t_0,s_0]$ and $u_2\in[t_1,t_2]$ two time coordinates of the intersection point of $b_{[t_0,s_0]}$ and $b_{[t_1,t_2]}$.
Using the additivity of the winding number again and that $b_{[t_0,u_1]}$, $b_{[\tau,u_2]}$ are each contained in a cone of angle $\pi$ centered at $z$ (recall \eqref{eq.crossing.mainProp.Pf.finalPropOfB.eq2}, \eqref{eq.crossing.mainProp.Pf.finalPropOfB.eq4}), we infer that $\texttt{Wnd}(z,b;u_1,u_2) \notin [-\pi/2,\pi/2]$. As the arc of $b$ over $[u_1,u_2]$ is closed, $\texttt{Wnd}(z,b;u_1,u_2) \neq0$ shows the claim at the beginning of the proof.
\end{proof}

\noindent\textit{Proof of Lemma \ref{le:crossing.mainProp.pf.lastLemma}.} The idea of this proof is somewhat similar to that of the proof of Lemma \ref{le:crossing}. For a visualization, please refer to the right panel of Figure \ref{fig:1}. Relying on \eqref{eq.crossing.mainProp.Pf.finalPropOfB.eq5}, we pick a sufficiently small $r>0$ such that the distance between $B^z_{[t_0,s_0]}$ and the open ball $O_1$ of radius $r$ around $B^z_{t_2}=B^z_{t_1}$ is greater than $r$.
Decreasing $r>0$ if necessary, we also ensure that the distance between $B^z_{[t_1,t_2]}$ and the open ball $O_2$ of radius $r$ around $B^z_{\tau}=B^z_{t_0}$ is greater than $r$. In addition, notice that the arc of $B^z$ over the time interval $[t_1,t_2]$ is closed and that $B^z_{[t_1,t_2]}$ is contained in $O_3$, the open ball of radius $|B^z_\tau - z|/2$ around $B^z_\tau$. Let $\varepsilon_3>0$ be the distance between $B^z_{[t_1,t_2]}$ and $\partial O_3$. Further, $B^z_{s_0}\notin O_3\cup\partial O_3$ by \eqref{eq.crossing.mainProp.Pf.finalPropOfB.eq2}. Put $\varepsilon_4>0$ for the distance between $B^z_{s_0}$ and $\partial O_3$. 

\medskip

Write $\gamma$ for the extension of $B^z$ from $[t_1,t_2]$ to $[t_1,t_2+1]$ defined by $\gamma(t)=B^z_{t_2}$, $t\in(t_2,t_2+1]$. Then, $\gamma$ is a closed continuous curve with $\texttt{Wnd}(B^z_\tau,\gamma;t_1,t_2+1) =\texttt{Wnd}(B^z_\tau,B^z;t_1,t_2) \neq 0$. Thus, Lemma \ref{le:crossing.Wnd.cont} yields the existence of an $\varepsilon_5>0$ such that $\texttt{Wnd}(B^z_\tau,\widetilde{\gamma};t_1,t_2+1) \neq 0$ for any closed continuous curve $\widetilde{\gamma}$ in the $\varepsilon_5$-neighborhood of~$\gamma$.
Choosing $\varepsilon_1=r\wedge\varepsilon_3\wedge\varepsilon_4\wedge\varepsilon_5>0$, we have for any continuous $b\!:[\tau^z_{\delta_0},\tau^z_\delta]\to\RR^2$ which belongs to the $\varepsilon_1$-neighborhood of $B^z_{[\tau^z_{\delta_0},\tau^z_\delta]}$,
\begin{itemize}
\item $b(t_1),b(t_2)\in O_1$ and $b_{[t_0,s_0]}\cap O_1=\varnothing$, 
\item $b(t_0),b(\tau) \in O_2$ and $b_{[t_1,t_2]}\cap O_2=\varnothing$, 
\item $b_{[t_1,t_2]}\subset O_3$ and $b(s_0)\notin O_3$.
\end{itemize}

\smallskip

Finally, we introduce the extension $\widetilde{b}$ of the arc of $b$ over the time interval $[t_1,t_2]$ by a linear interpolation between $b(t_2)$ and $b(t_1)$ on $[t_2,t_2+1]$. Then, $\widetilde{b}$ is a closed continuous curve in the $\varepsilon_5$-neighborhood of $\gamma$ and, hence, $\texttt{Wnd}(B^z_\tau,\widetilde{b};t_1,t_2+1) \neq 0$. It follows from Corollary \ref{cor:crossing.Wnd.vsClosing} that $B^z_\tau$ is contained in a bounded connected component of $\RR^2\setminus \widetilde{b}_{[t_1,t_2+1]}$.~In addition, the line segment $\ell$ connecting $B^z_\tau$ and $b(t_0)$ is contained in~$O_2$, which shows that $\ell$ does not intersect $\widetilde{b}_{[t_1,t_2]}=b_{[t_1,t_2]}$. Since $\widetilde{b}_{[t_2,t_2+1]}\subset O_1$ and $O_1\cap O_2=\varnothing$, the segment $\ell$ does not intersect $\widetilde{b}_{[t_2,t_2+1]}$ either, so that $b(t_0)$ is contained in a bounded connected component of $\RR^2\setminus \widetilde{b}_{[t_1,t_2+1]}$ as well. Thus, $\widetilde{b}_{[t_1,t_2+1]}\subset O_3$ and $b(s_0)\notin O_3$ reveal that $b(t_0)$ and $b(s_0)$ belong to different connected components of $\RR^2\setminus \widetilde{b}_{[t_1,t_2+1]}$. Therefore, $b_{[t_0,s_0]}$ and $\widetilde{b}_{[t_1,t_2+1]}$ admit an intersection point. It only remains to observe that this intersection point cannot lie on $\widetilde{b}_{[t_2,t_2+1]}$, as $\widetilde{b}_{[t_2,t_2+1]}\subset O_1$ and $b_{[t_0,s_0]}\cap O_1=\varnothing$. We infer that $\varnothing\neq b_{[t_0,s_0]}\cap\widetilde{b}_{[t_1,t_2]}=b_{[t_0,s_0]}\cap b_{[t_1,t_2]}$, as desired. \qed


\medskip

Applying Proposition \ref{prop:stabCrossing} in the setting of Theorem \ref{thm:char}, we obtain the following corollary, which is used in the proof of Theorem \ref{thm:char}. 

\begin{corollary}\label{cor:stabCrossing}
For $d\in\{1,2\}$ and under Assumption \ref{ass:1}, let $(\mu^X,D)$ be a limit point in distribution of a sequence of MDLA processes (cf.~Proposition \ref{prop:existence}), constructed on a probability space $(\Omega,\mathcal{F},\PP)$. Then, for $\PP$-a.e. $\omega\in\Omega$, the set
\begin{equation*}
\begin{split}
&\big\{b\in C([0,T+1],\RR^d):\,\forall\,\delta>0\;\exists\,\varepsilon>0\;\forall\,b'\in C([0,T+1],\RR^d)\,\,\max_{[0,T+1]} \!|b'\!-\!b|\!<\!\varepsilon \\ 
&\Longrightarrow\,
\inf\{t\!\in\![0,T\!+\!1]\!:D_t(\omega;b'(t))\!=\!0\}
\leq\delta+\inf\{t\!\in\![0,T]\!:D_t(\omega;b(t))\!=\!0\}\!\big\},
\end{split}
\end{equation*}
with the convention $\inf\varnothing=\infty$, has probability one under $\mu^{\wt X}$, where the latter measure is defined in Assumption \ref{ass:1}.
\end{corollary}

\begin{proof}
By conditioning on $\xi=z$ for $z\in\RR^2$ and subsequently shifting the coordinate system, the corollary can be reduced to the case $\xi\equiv0$, so we only consider that case. For $d=1$, the corollary results from $D_\tau(b(\tau))=0$ for $\tau:=\inf\{t\in[0,T]\!:D_t(b(t))\!=\!0\}$, the strong Markov property of Brownian motion and $\max_{[0,\delta]} b>0$, $\min_{[0,\delta]} b<0$ for all $\delta>0$ and almost every standard Brownian path $b$. From here on, we take $d=2$. In view of the Skorokhod Representation Theorem, we may pick a sequence of MDLA processes, which we denote by $(\mu^{X,N},D^N)_{N\in\NN}$ for convenience, that converges to the limit $(\mu^X,D)$ almost surely. Moreover, if $\inf\{t\in[0,T]:\,D_t(b(t))=0\}=\infty$, then the desired inequality holds for all $b'$ automatically, and we only need to analyze the paths $b$ with $\min_{t\in[0,T]} D_t(b(t)) = 0$. The strong Markov property of Brownian motion and Proposition \ref{prop:stabCrossing} imply that, for almost every standard Brownian path $b$ on $[0,T+1]$ satisfying $\min_{t\in[0,T]} D_t(b(t)) = 0$ and all $\delta>0$, there exists an $\varepsilon=\varepsilon(\delta,b)>0$ such that for all $b'\in C([0,T+1],\RR^2)$ with $\max_{[0,T+1]} |b'-b|<\varepsilon$, there exists a bounded open neighborhood $V=V(\delta,b,b')$ of $b(\tau)$ fulfilling
\begin{equation}
\partial V\subset b_{[\tau,\tau+\delta]}',\quad\text{where}\quad\tau:=\inf\{t\in[0,T]:\,D_t(b(t))=0\}\leq T.
\end{equation}
The main claim in the proof of Proposition \ref{prop:stabCrossing} shows that, in fact, there is a $\delta_0\in(0,\delta)$ with $\partial V\subset b_{[\tau+\delta_0,\tau+\delta]}'$. Further, we can choose $\delta_0\in(0,\delta)$ so that $\tau+\delta_0$ is a continuity point of $D$ (recall Remark \ref{rem:discont.Dbar}). Note also that $D_{\tau+\delta_0}(b(\tau))=0$. The convergence of $D^N_{\tau+\delta_0}(b(\tau))$ to $D_{\tau+\delta_0}(b(\tau))=0$ then yields $V\cap\Gamma^N_{\tau+\delta_0}\neq\varnothing$ for all $N\in\NN$ large enough. 

\medskip

Next, we use Assumption \ref{ass:1} to deduce the existence of a universal constant $\delta_1>0$ such that, for all $N\geq1$, $t\in[0,T+1]$ and $x\in\Gamma^N_t$, there exists $y\in\Gamma^N_t$, with $|y-x|\geq\delta_1$, that can be connected to $x$ by a continuous curve in $\Gamma^N_t$. In addition, we deduce from Proposition \ref{prop:stabCrossing} that we can pick $V$ to be contained in the open ball of radius $\delta_1$ around $b(\tau)$ (by possibly decreasing $\delta_0\in(0,\delta)$). It follows that, for all $N\in\NN$ large enough, one can find an $x_N\in\partial V\cap\Gamma^N_{\tau+\delta_0}$. Taking $N\to\infty$, we find a limit point $x_\infty\in\partial V\subset b_{[\tau+\delta_0,\tau+\delta]}'$. The $1$-Lipschitz property of $x\mapsto D^N_{\tau+\delta_0}(x)$ and the convergence of $D^N_{\tau+\delta_0}(x_\infty)$ to $D_{\tau+\delta_0}(x_\infty)$ then imply $D_{\tau+\delta_0}(x_\infty)=0$ and, since $x_\infty\in b_{[\tau+\delta_0,\tau+\delta]}'$, also $\inf\{t\in[0,T+1]\!:D_t(b'(t))=0\}
\le\tau+\delta$, as desired.
\end{proof}

\subsection{Proof of Theorem \ref{thm:char}}\label{subse:proof}

Denote by $(\mu^{X,N},D^N)_{N\in\NN}$ a sequence of MDLA processes that converges to $(\mu^X,D)$ in distribution along a subsequence. Using this fact and the convergence of $(\mu^{\wt X,N})_{N\in\NN}$, we conclude that the joint distributions of the triplets $(\mu^{X,N},D^N,\mu^{\wt X,N})_{N\in\NN}$ (along the same subsequence) are tight, and, hence, the latter triplets converge to $(\mu^X,D,\mu^{\wt X})$ in distribution along a possibly different subsequence (recall that $\mu^{\wt X}$ is a deterministic measure). To simplify notation, we assume that the latter subsequence coincides with the original sequence. Further, in view of the Skorokhod Representation Theorem, we may assume that the convergence $(\mu^{X,N},D^N,\mu^{\wt X,N})\to(\mu^X,D,\mu^{\wt X})$ occurs almost surely on $(\Omega,\mathcal{F},\PP)$.~Our goal is then to show that, for any bounded and uniformly continuous function $f\!:\cD([0,T],\RR^d)\rightarrow\RR$, 
\begin{equation}\label{aim1}
\frac{1}{N}\sum_{i=1}^N f\big(\wt X^{i,N}_{\cdot\wedge\tau^{i,N}}\big) \underset{N\to\infty}{\longrightarrow} \int_{\widehat{\Omega}} f\big(\xi(\widehat{\omega})+B_{\cdot\wedge\tau(\omega\,,\widehat{\omega})}(\widehat{\omega})\big)\,\mathrm{d}\widehat{\PP}(\widehat{\omega})
\end{equation}
for $\PP$-almost every $\omega\in\Omega$, where $\{\wt X^{i,N}\}_{i=1}^N$ are the atoms of $\mu^{\wt X,N}$ (i.e., the underlying particle system of the $N$-th MDLA process), $\{\tau^{i,N}\}_{i=1}^N$ are the corresponding absorption times (see \eqref{eq.tau.i.N.def}), and $\xi$, $B$, $\tau$ are as in the statement of the theorem\footnote{Strictly speaking, each underlying particle path $\wt X^{i,N}$ is defined on $[0,T+1]$.~With a minor abuse of notation, we let $f$ act on the restrictions of the stopped paths to $[0,T]$ on the left-hand side~of~\eqref{aim1}.}. Indeed, the left-hand side of \eqref{aim1} tends $\PP$-almost surely  to the integral of $f$ with respect to $\mu^X$, so that \eqref{aim1} proves the theorem. 

\medskip

To invoke the stability of the crossing property (Corollary \ref{cor:stabCrossing}) below, we define, for each $N\!\in\!\nn$, a continuous version $\{Y^{i,N}\}_{i=1}^N$ of the underlying particle system~$\{\wt X^{i,N}\}_{i=1}^N$. For this purpose, we denote by $\rho^{i,N}_0< \rho^{i,N}_1<\cdots<\rho^{i,N}_{J^{i,N}}$ the ordered elements of the set which consists of $0$, $T+1$ and the jump times of $\wt X^{i,N}$ on $[0,T+1]$. We then let
\begin{equation}\label{eq.sec4.2.YiN.def}
\begin{split}
& Y^{i,N}_t := \wt X^{i,N}_0,\;\; t\in[\rho^{i,N}_0,\rho^{i,N}_1),
\quad Y^{i,N}_{T+1}:=\wt X^{i,N}_{\rho^{i,N}_{J^{i,N}-1}},\\
&Y_t^{i,N}:=\wt X^{i,N}_{\rho^{i,N}_j}+\frac{t-\rho^{i,N}_{j+1}}{\rho^{i,N}_{j+2}-\rho^{i,N}_{j+1}}
\Big(\wt X^{i,N}_{\rho^{i,N}_{j+1}}-\wt X^{i,N}_{\rho^{i,N}_j}\Big),\;\; t\in[\rho^{i,N}_{j+1},\rho^{i,N}_{j+2}),\;\; j=0,\,1,\,\ldots,\,J^{i,N}-2
\end{split}
\end{equation}
and introduce
$$
\sigma^{i,N}:=\inf\{t\ge0:\,D^N_t(Y^{i,N}_t)< N^{-\frac{1}{d}}/2\}.
$$

\smallskip


By Assumption \ref{ass:1}, the range of each $Y^{i,N}$ is comprised of segments connecting the neighboring sites of $\ZZ^d/N^{\frac{1}{d}}$ visited by $\widetilde{X}^{i,N}$, which allows us to deduce
\begin{equation}\label{eq.sec4.2.tauiN.eq.sigmaiN}
\tau^{i,N} \wedge T = \sigma^{i,N}\wedge T \quad\text{and}\quad
\sup_{t\in[0,T]}|\wt X^{i,N}_{t\wedge\tau^{i,N}}-Y^{i,N}_{t\wedge\sigma^{i,N}}| \leq N^{-\frac{1}{d}}
\end{equation}
by direct verification.~Thus, the M1 distance from $\wt X^{i,N}_{\cdot\wedge\tau^{i,N}}$ to $Y^{i,N}_{\cdot\wedge\sigma^{i,N}}$ is at most $N^{-\frac{1}{d}}$. Putting this together with the uniform continuity of $f$ we obtain
\begin{equation*}
\lim_{N\to\infty} \bigg|\frac{1}{N}\sum_{i=1}^N f\big(\wt X^{i,N}_{\cdot\wedge\tau^{i,N}}\big)-\frac{1}{N}\sum_{i=1}^N f\big(Y^{i,N}_{\cdot\wedge\sigma^{i,N}}\big)\bigg|=0.
\end{equation*}
Consequently, to prove \eqref{aim1} it is enough to show the almost sure convergence
\begin{equation}\label{aim2}
\frac{1}{N}\sum_{i=1}^N f\big(Y^{i,N}_{\cdot\wedge\sigma^{i,N}}\big) 
=\frac{1}{N}\sum_{i=1}^N f\big(\pi_{\Gamma^N}(Y^{i,N})\big)
\underset{N\to\infty}{\longrightarrow} 
\int_{C([0,T+1],\RR^d)}  f\circ\pi_\Gamma\,\mathrm{d}\mu^{\wt X},
\end{equation}
with the absorption mappings $\pi_{\Gamma^N},\pi_\Gamma\!:C([0,T+1],\RR^d)\to C([0,T],\RR^d)$ given by
\begin{align}
&\pi_{\Gamma^N}(b)=b(\cdot\wedge\beta^N(b))|_{[0,T]},\;\;
\beta^N(b)=\inf\{t\in[0,T+1]:\,D_t^N(b(t))<N^{-\frac{1}{d}}/2\}, \label{eq.sec4.2.betaN.def}\\
& \pi_\Gamma(b)=b(\cdot\wedge\beta(b))|_{[0,T]},\;\;
\beta(b)=\inf\{t\in[0,T+1]:\,D_t(b(t))=0\}.\label{eq.sec4.2.beta.def}
\end{align}
The main challenge in establishing \eqref{aim2} stems from the fact that, in general, $\pi_{\Gamma}(b)$ is discontinuous in both $\Gamma$ and $b$. The remainder of the proof relies on the stability of the crossing property (Corollary \ref{cor:stabCrossing}) to verify the desired continuity of $\pi_{\Gamma}(b)$ at almost every limiting growing aggregate $\Gamma(\omega)$ and Brownian path $b$. 

\medskip

First, we aim to demonstrate the almost sure convergence
\begin{equation}\label{eq.mainThm.Pf.absorbCont.inX}
\frac{1}{N}\sum_{i=1}^N f\big(\pi_\Gamma(Y^{i,N})\big)
\underset{N\to\infty}{\longrightarrow} 
\int_{C([0,T+1],\RR^d)}  f\circ\pi_\Gamma\,\mathrm{d}\mu^{\wt X}.
\end{equation}
To ease the notation, for any path $b\in C([0,T+1],\RR^d)$ and $r>0$, we let
$$
O_{\mathcal C}(b,r):=\big\{b'\in C([0,T+1],\RR^d):\,\max_{[0,T+1]} |b'-b|<r\big\}.
$$
If $\beta(b)>T$, we consider
\begin{equation}\label{case_a}
\eps_1:=\inf_{t\in[0,T]} D_t(b(t))=\min_{t\in[0,T]} D_t(b(t)) > 0.
\end{equation}
Since the functions $D_t(\cdot)$, $t\in[0,T]$ are $1$-Lipschitz, we have
\begin{equation*}
\inf_{t\in[0,T]} D_t(b'(t))=\min_{t\in[0,T]} D_t(b'(t))>0\quad\text{and}\quad\beta(b')>T,\quad\text{for all}\quad b'\in O_{\mathcal C}(b,\eps_1).
\end{equation*}
Thus, near the elements of $\{b\!\in\! C([0,T\!+\!1],\RR^d)\!\!:\!\beta(b)\!>\!T\}$ the mapping $\pi_\Gamma$ just restricts paths to $[0,T]$ and is therefore continuous on this set.

\medskip

To tackle the case $\beta(b)\leq T$, we recall from Corollary \ref{cor:stabCrossing} that, almost surely, for $\mu^{\wt X}$-almost every $b$ and all $\delta>0$, there exists an $\varepsilon=\varepsilon(b,\delta)>0$ such that
$$
\inf\{t\in[0,T+1]:\,D_t(b'(t))=0\} \leq \delta+\inf\{t\in[0,T]:\,D_t(b(t))=0\},\;\;b'\in O_{\mathcal C}(b,\varepsilon).
$$
On the other hand, the $1$-Lipschitz property of the functions $D_t(\cdot)$, $t\in[0,T]$ and the lower semicontinuity of $t\mapsto D_t(b(t))$ reveal that, for any path $b\in C([0,T+1],\RR^d)$ with $\beta(b)\le T$, and for any $\delta>0$, there exists an $\varepsilon'=\varepsilon'(b,\delta)>0$ for which
\begin{equation}\label{case_a.new}
\inf\{t\in[0,T+1]:\,D_t(b'(t))=0\} \geq \inf\{t\in[0,T+1]:\,D_t(b(t))=0\} - \delta,\;\;b'\in O_{\mathcal C}(b,\varepsilon').
\end{equation}
Hence, almost surely, for $\mu^{\wt X}$-almost every $b$ and all $\delta>0$, the bound $\beta(b)\le T$ implies $|\beta(b')-\beta(b)|\leq\delta$, $b'\in O_{\mathcal C}(b,\varepsilon\wedge\varepsilon')$. Combining this with the triangle inequality
\begin{equation*}
\max_{[0,T]} |\pi_\Gamma(b')-\pi_\Gamma(b)|
\leq \max_{[0,T]} |b(\cdot\wedge\beta(b'))-b(\cdot\wedge\beta(b))|+\max_{[0,T]} |b'-b|,
\end{equation*}
we see that, almost surely, the mapping $\pi_\Gamma$ is continuous at $\mu^{\wt X}$-almost every $b$. 

\medskip

The Continuous Mapping Theorem now reduces \eqref{eq.mainThm.Pf.absorbCont.inX} to the almost sure weak convergence of the empirical measure $\mu^{Y,N}$ of $\{Y^{i,N}\}_{i=1}^N$ to $\mu^{\wt X}$ in $C([0,T+1],\rr^d)$. Since $|Y^{i,N}_\cdot-\wt X^{i,N}_\cdot|\le N^{-\frac{1}{d}}$, we have $\mu^{Y,N}\to\mu^{\wt X}$ weakly in $\cD([0,T+1],\rr^d)$ almost surely, so it suffices to check the almost sure tightness of $\{\mu^{Y,N}\}_{N\in\NN}$ on $C([0,T+1],\rr^d)$. In view of the Arzel\`a-Ascoli Theorem, the latter is a consequence of the almost sure $C$-tightness of $(\mu^{\wt X,N})_{N\in\NN}$ (see the last paragraph in the proof of Proposition~\ref{prop:existence}), the implication \cite[Chapter VI, Proposition 3.26, (i)$\Rightarrow$(ii)]{JacodShiryaev} and $|Y^{i,N}_\cdot-\wt X^{i,N}_\cdot|\le N^{-\frac{1}{d}}$.

\medskip

To finish the proof we show the almost sure convergence
\begin{equation}\label{eq.mainThm.Pf.absorbCont.inN}
\frac{1}{N}\sum_{i=1}^N f\big(\pi_{\Gamma^N}(Y^{i,N})\big)
-\frac{1}{N}\sum_{i=1}^N f\big(\pi_\Gamma(Y^{i,N})\big) \underset{N\to\infty}{\longrightarrow} 0. 
\end{equation}
The Skorokhod Representation Theorem yields, almost surely, $C([0,T+1],\RR^d)$-valued $\{Y^{0,N}\}_{N\in\NN}$ and $\wt X$, living on some $(\Omega',\mathcal{F}',\PP')$, such that the law of $Y^{0,N}$ is $\mu^{Y,N}$, $N\in\NN$, the law of $\wt X$ is $\mu^{\wt X}$, and $\lim_{N\to\infty} Y^{0,N}=\wt X$ almost surely. Then, $\PP'$-almost~surely,
$$
\max_{[0,T+1]} |Y^{0,N}_{\cdot\wedge\beta^N(Y^{0,N})}-\wt X_{\cdot\wedge\beta^N(Y^{0,N})}|\underset{N\to\infty}{\longrightarrow}0 \quad\text{and}\quad
\max_{[0,T+1]} |Y^{0,N}_{\cdot\wedge\beta(Y^{0,N})}
-\wt X_{\cdot\wedge\beta(Y^{0,N})}|\underset{N\to\infty}{\longrightarrow}0.
$$
The above, the uniform continuity of $f$, and the boundedness of $f$ reduce \eqref{eq.mainThm.Pf.absorbCont.inN} to
\begin{align}
&\lim_{N\to\infty}\EE'\Big[\Big(f\big(\wt X_{\cdot\wedge\beta^N(Y^{0,N})}|_{[0,T]}\big)
-f\big(\wt X_{\cdot\wedge\beta(Y^{0,N})}|_{[0,T]}\big)\!\Big)\,\bone_{A_1}\Big]=0,\label{aim3}\\
&\lim_{N\to\infty}\EE'\Big[\Big(f\big(\wt X_{\cdot\wedge\beta^N(Y^{0,N})}|_{[0,T]}\big)
-f\big(\wt X_{\cdot\wedge\beta(Y^{0,N})}|_{[0,T]}\big)\!\Big)\,\bone_{A_2}\Big]=0,\label{aim3.2}
\end{align}
where we have partitioned $\Omega'$ into
\begin{equation*}
A_1:=\{\beta(\wt X)>T\}\quad\text{and}\quad A_2:=\{\beta(\wt X)\leq T\}.
\end{equation*}

\smallskip

Recalling the argument following \eqref{case_a} we see that on $A_1$ it holds $\beta(Y^{0,N})>T$ for all $N\in\nn$ large enough. Next, we claim that, for $\PP'$-almost every outcome in $A_1$,
\begin{equation}\label{eq.mainThm.Pf.case.a}
\liminf_{N\rightarrow\infty}\inf_{t\in[0,T]} D^N_t(Y^{0,N}_t)>0.
\end{equation}
This implies $\beta^N(Y^{0,N})>T$ for large enough $N\in\NN$, yielding \eqref{aim3}. To show~\eqref{eq.mainThm.Pf.case.a}, we argue by contradiction and let $A_0\subset A_1$ be an event  with $\PP'(A_0)>0$ on which 
$$\inf_{t\in[0,T]} D^N_t(Y^{0,N}_t) \rightarrow 0$$
along a subsequence of $N\in\NN$. Then, for all $\omega'\in A_0$, there exists a sequence $(t_{N_k})_{k\in\NN}$ in $[0,T]$ such that $\lim_{k\to\infty} D^{N_k}_{t_{N_k}}(Y^{0,N_k}_{t_{N_k}})=0$ and $t_\infty:=\lim_{k\to\infty} t_{N_k}\in[0,T]$. Further, by Remark \ref{rem:discont.Dbar}, there exists a decreasing sequence $(s_j)_{j\in\NN}$ of continuity points of $D$ with $s_j\downarrow t_\infty$ as $j\to\infty$. For any $j\in\NN$ and all large enough $k\in\NN$ (so that $t_{N_k}\le s_j$), 
\begin{equation*}
D^{N_k}_{s_j}(\wt X_{t_\infty})\leq D^{N_k}_{t_{N_k}}(\wt X_{t_\infty})
\leq D^{N_k}_{t_{N_k}}\big(Y^{0,N_k}_{t_{N_k}}\big)+|Y^{0,N_k}_{t_{N_k}}-\wt X_{t_{N_k}}|
+|\wt X_{t_{N_k}}-\wt X_{t_\infty}|\underset{k\to\infty}{\longrightarrow}0.
\end{equation*}
Hence, the right-continuity of $D_\cdot(\wt X_{t_\infty})$ and the convergences $D^{N_k}_{s_j}(\wt X_{t_\infty})\to D_{s_j}(\wt X_{t_\infty})$ as $k\to\infty$ for $j\in\NN$ lead to
$$
D_{t_\infty}(\wt X_{t_\infty}) = \lim_{j\to\infty} D_{s_j}(\wt X_{t_\infty}) 
= \lim_{j\to\infty}\,\lim_{k\to\infty} D^{N_k}_{s_j}(\wt X_{t_\infty}) = 0.
$$
This contradicts $\beta(\wt X)>T$, showing that \eqref{eq.mainThm.Pf.case.a} must hold $\PP'$-almost surely on $A_1$.

\medskip

For $\PP'$-almost every outcome in $A_2$, Corollary \ref{cor:stabCrossing} reveals that, for any $\delta\in(0,1)$,
\begin{equation}\label{est1}
\beta(Y^{0,N})\le\delta+\beta(\wt X)\;\;\text{for all large enough}\;\; N\in\NN. 
\end{equation}
Moreover, the construction of $x_N$ in the proof of Corollary \ref{cor:stabCrossing} results in
\begin{equation}\label{est2}
\beta^N(Y^{0,N})\leq \inf\{t\in[0,T\!+\!1]\!:D^N_t(Y^{0,N}_t)=0\}\leq\beta(\wt X)+\delta\;\text{for all large enough}\; N\in\NN.
\end{equation}
On the other hand, by the argument leading to \eqref{case_a.new}, for $\delta\in(0,\beta(\wt X))$,
\begin{equation}\label{est3}
\beta(Y^{0,N})>\beta(\wt X)-\delta\;\;\text{for all large enough}\;\; N\in\NN. 
\end{equation}
Further, by repeating the proof of \eqref{eq.mainThm.Pf.case.a}, we obtain, for $\delta\in(0,\beta(\wt X))$,
\begin{equation}\label{est4}
\liminf_{N\rightarrow\infty}\inf_{t\in[0,\beta(\wt X)-\delta]} D^N_t(Y^{0,N}_t)\!>\!0,\;\text{thus}\;
\beta^N(Y^{0,N})\!>\!\beta(\wt X)-\delta\;\text{for all large enough}\,N\in\NN. 
\end{equation}
By collecting \eqref{est1}--\eqref{est4}, we get $\beta(\wt X)-\delta<\beta^N(Y^{0,N}),\beta(Y^{0,N})\le\beta(\wt X)+\delta\le T+\delta$ for all large enough $N\in\NN$. This, the uniform continuity of $f$, and the boundedness of $f$ lead to \eqref{aim3.2}, completing the proof of Theorem  \ref{thm:char}. \qed

\begin{rmk}\label{rem:beta.conv}
The proof of \eqref{eq.mainThm.Pf.absorbCont.inN} presented above has another corollary that we use below. Namely, the proof of \eqref{eq.mainThm.Pf.absorbCont.inN} also reveals that it holds $\sup_{[0,T]}|Y^{0,N}_\cdot-\wt X_\cdot|\rightarrow 0$ and $\beta^N(Y^{0,N})\wedge T,\,\beta(Y^{0,N})\wedge T\rightarrow \beta(\wt X)\wedge T$, as $N\rightarrow\infty$, $\PP'$-almost surely. These, in turn, imply that, $\PP$-almost surely, as $N\rightarrow\infty$,
\begin{align*}
\mu^{Y,N}\circ\big(\pi_{\Gamma^N}(\cdot), \beta^N(\cdot)\wedge T\big)^{-1},
\;\mu^{Y,N}\circ\big(\pi_{\Gamma^N}(\cdot), \beta(\cdot)\wedge T\big)^{-1}
\longrightarrow\mu^{X}\circ\big(\,\cdot\,,\beta(\cdot)\wedge T\big)^{-1},
\end{align*}
in the topology of weak convergence on $\mathcal{P}(C([0,T],\RR^d)\times\RR)$.
\end{rmk}

\section{Connection with the supercooled Stefan problem}\label{se:Stefan}

In this section, we prove Theorem \ref{thm:main}(b.2), by investigating the relationship between the scaling limits of (external) MDLA processes and the single-phase supercooled Stefan problem (1SSP) for the heat equation under Assumptions \ref{ass:1}, \ref{ass:2}. Our analysis is motivated by the findings in \cite{ChSw}, \cite{DIRT2}, \cite{NadShk1}, \cite{DT}, \cite{LS}, \cite{DNS19}, \cite{Cuchiero}, which ultimately show that, for $d=1$, any limit point $(\mu^X,D)$ of a sequence of MDLA processes is given by the unique (in the appropriate sense) solution of the 1SSP for the heat equation.\footnote{Strictly speaking, most of the growth processes studied in these references do not fit Definition~\ref{def:extDLA} of an MDLA process, as the associated underlying particle systems do not take values in a discrete space, and the dynamics of the aggregates involve additional non-linear transformations. Nonetheless, it is not hard to see that, up to minor adjustments, the arguments in the aforementioned papers also apply to the MDLA processes considered herein, for $d=1$.} It turns out that, already for $d=2$, the limit points of MDLA processes may not solve the 1SSP for the heat equation. Nevertheless, there does exist a connection between the limit points of MDLA processes and the 1SSP for the heat equation. In the remainder of the section, we illustrate this connection and provide an example in which the limit of a sequence of MDLA processes fails to solve the 1SSP for the heat equation. 

\subsection{Single-phase Stefan problem for the heat equation} 

We start with the single-phase Stefan problem for the heat equation in its classical formulation. For a closed subset $\Gamma_0$ of $\RR^d$ and a function $w_0\in C(\RR^d,\RR)$ supported in $\RR^d\backslash\Gamma_0$, find a family $\{\Gamma_t\}_{t\in[0,T]}$ of closed subsets of $\RR^d$ and a function $w\in C^{1,2}(Q_T,\RR)\cap C(\overline{Q_T},\RR)$, with 
\begin{equation}\label{eq.sec5.QT.def}
Q_T=\{(t,x)\in (0,T)\times\RR^d:\, x\notin\Gamma_t\},
\end{equation}
such that $\nabla_x w\in C(\overline{Q_T}\setminus(\{0\}\times\overline{\RR^d\backslash\Gamma_0}),\RR^d)$ and
\begin{equation}\label{eq.stefan.PDE.def}
\begin{array}{ll}
\partial_t w = \Delta w/2 &\text{on}\;\; Q_T, \\
V = -(\nabla_x w\cdot\nu)/2 &\text{on}\;\;\{(t,x)\in(0,T]\times\RR^d:\,x\in\partial\Gamma_t\}, \\
w=0 &\text{on}\;\;\{(t,x)\in(0,T]\times\RR^d:\,x\in\partial\Gamma_t\}, \\
w(0,\cdot\,)=w_0&\text{on}\;\;\RR^d\setminus\Gamma_{0},
\end{array}
\end{equation}
where $\nu$ is the outward unit normal vector field on $\partial\Gamma_t$, $t\in(0,T]$, and $V$ is the normal growth speed on that same set.\footnote{In~\eqref{eq.stefan.PDE.def}, we have taken the ratios of the thermal conductivity to the heat capacity per unit volume and of the density of latent heat to the thermal conductivity as $1/2$.~This is done only to reduce the number of constants -- it is easy to see how these constants can be brought back by a time and space rescaling, both in~\eqref{eq.stefan.PDE.def} and in the MDLA process.~We also choose to consider the problem on the entire $\RR^d$, to avoid additional conditions at an external boundary.}~In~\eqref{eq.stefan.PDE.def}, $\Gamma_t$ represents the region occupied by the solid phase (e.g., ice) at time $t$, and $\RR^d\setminus\Gamma_t$ is occupied by the liquid phase (e.g., water).~The function $w(t,\cdot\,)$ stands for the temperature distribution in the liquid phase at time $t$.~We assume that $0$ is the temperature at which the phase transition occurs in equilibrium (hence the third line in~\eqref{eq.stefan.PDE.def}).~The temperature in the solid phase is assumed to always equal $0$, leading to a single-phase problem.~The second line in~\eqref{eq.stefan.PDE.def} constitutes a growth condition, saying that the speed at which the solid phase grows/shrinks is proportional to the rate of decrease/increase in the temperature of the liquid in the direction normal to the phase boundary.  

\medskip

In many important applications, e.g., in models of the freezing process in a supercooled liquid, and also in models of crystal growth, it is necessary to consider the Stefan problem \eqref{eq.stefan.PDE.def} with $w_0<0$. Such a situation is referred to as the supercooled regime, since then the liquid phase is cooled below its equilibrium freezing temperature initially. Note that the classical formulation of the Stefan problem does not change in this regime. Indeed, there is no a priori constraint on the sign of $w_0$ in \eqref{eq.stefan.PDE.def}.

\medskip

The downside of the classical formulation is its assumption of substantial  regularity on the solution. For example, we need $\partial\Gamma_t$ to be sufficiently smooth to define $\nu$ and for~$V$ to make sense.\footnote{More specifically, in order for $V$ to be well-defined, one may assume $\Gamma_t=\{x\in\RR^d:\, \Phi(t,x)\leq0\}$, for $\Phi\in C^1((0,T]\times\RR^d,\RR)$ with $|\nabla_x\Phi|>0$. In this case, $V=-\partial_t\Phi/|\nabla_x\Phi|$.} Establishing such a regularity is an issue in general (to date, there exist no general well-posedness results for the classical formulation \eqref{eq.stefan.PDE.def} when $d>1$), but it becomes a particular concern in the supercooled regime, where the solutions of the Stefan problem are known to exhibit singularities (see, e.g., \cite{Sher}, \cite{HOL}, \cite{DF}, \cite{HV}).
To avoid the regularity assumptions on the solution, a weak formulation of the Stefan problem was proposed: see, e.g., \cite{Ishii}, \cite{ChSw}, \cite{Vis}. Herein, we provide a modification of this definition that is well-suited for the \textit{supercooled} Stefan problem. Namely, for a closed $\Gamma_{0-}\subset\RR^d$ and a locally integrable $w_0\!:\RR^d\rightarrow(-\infty,0]$ essentially supported in $\RR^d\backslash\Gamma_{0-}$, find a non-decreasing family $\{\Gamma_t\}_{t\in[0,T]}$ of closed subsets of $\rr^d$ containing~$\Gamma_{0-}$ and locally integrable $w\!:[0,T]\times\RR^d\rightarrow\RR$, $\chi\!:[0,T]\times\RR^d\rightarrow[0,\infty)$ with
\begin{align}
&\,\forall\,t\!\in\![0,T],\varphi\!\in\! C^\infty_c([0,t]\!\times\!\RR^d)\!:
\int_{\RR^d} \!\varphi(t,x)\big(w(t,x)\!-\!\chi(t,x)\!\big) 
\!-\!\varphi(0,x)\big(w_0(x)\!-\!\bone_{\Gamma_{0-}}(x)\!\big)\mathrm{d}x
\label{eq.SecStefan.weakStefan.PDE}\\
&\qquad\qquad\qquad\qquad\qquad
= \int_0^t \int_{\RR^d} \partial_t\varphi(s,x)\big(w(s,x)\!-\!\chi(s,x)\!\big) + \frac{1}{2}\Delta \varphi(s,x)\, w(s,x)\,\mathrm{d}x\,\mathrm{d}s, \nonumber \\
&\chi(t,x)\,\bone_{\RR^d\setminus \bigcup_{0\le s\le t}\Delta\Gamma_s}(x)
=\bone_{\Gamma_t\setminus \bigcup_{0\le s\leq t}\Delta\Gamma_s}(x),
\;\; \int_{\Delta\Gamma_s} \chi(t,x)\,\mathrm{d}x = \text{Leb}(\Delta\Gamma_s),\;\; s\in[0,t], \label{eq.SecStefan.weakStefan.chi}
\end{align}
where $\Delta\Gamma_s=\Gamma_s\setminus\Gamma_{s-}$ and the union $\bigcup_{0\le s\leq t}\Delta\Gamma_s$ is taken over all $s\in[0,t]$ such that $\text{Leb}(\Delta\Gamma_s)>0$ (the set of such $s$ is at most countable),
\begin{align}
&\,\bone_{\Gamma_t}(x)\,w(t,x) = 0\;\;\text{for almost every}\;\;(t,x)\in[0,T]\times\RR^d.
\label{eq.SecStefan.weakStefan.bdryCond}
\end{align}
The displays \eqref{eq.SecStefan.weakStefan.PDE}, \eqref{eq.SecStefan.weakStefan.chi} are a weak form of the first two and the last line in \eqref{eq.stefan.PDE.def}. They ensure that the temperature in the liquid phase starts from the prescribed initial condition and evolves according to the heat equation, that the solid phase starts from the prescribed initial condition and grows with the correct speed (in a weak sense) at the continuity times, and that the energy is preserved at the jump times (i.e., the total enthalpy of a solid block that is attached at a jump time does not change at that time). The third display, \eqref{eq.SecStefan.weakStefan.bdryCond}, is a substitute for the boundary condition (the third line in \eqref{eq.stefan.PDE.def}) -- it enforces that the temperature equals to $0$ in the solid phase.~We note that, if $\text{Leb}(\Gamma_t\backslash\Gamma_{t-})\!=0$, $t\in[0,T]$, the display \eqref{eq.SecStefan.weakStefan.chi} reduces to $\chi(t,\cdot\,)=\mathbf{1}_{\Gamma_t}$, $t\in[0,T]$.

\medskip

The boundary condition \eqref{eq.SecStefan.weakStefan.bdryCond} is often strengthened (see, e.g., \cite{Ishii}, \cite{ChSw}, \cite{Vis}) to
\begin{equation}\label{eq.SecStefan.weakStefan.bdryCond.strong}
\bone_{\Gamma_t}(x)\,w(t,x) = 0,
\;\bone_{\RR^d\setminus\Gamma_t}(x)\,\bone_{\{w(t,x)\le 0\}}=0
\;\text{for almost every}\;(t,x)\in[0,T]\times\RR^d
\end{equation}
in the context of a general (not necessarily supercooled) single-phase Stefan problem.~The equations in \eqref{eq.SecStefan.weakStefan.bdryCond.strong} imply that the phase is determined by the temperature. If the temperature at a point is above freezing, the point is in the liquid phase;~and if a point is at the freezing temperature, the point belongs to the solid phase.~As a consequence, the phase $\bone_{\Gamma_t}(x)$ depends on the temperature $w(t,x)$ in a monotone way.~Using this monotonicity, a comparison principle for weak solutions of the Stefan problem, which yields its well-posedness, is established in \cite{Ishii} (see also \cite{KimVisc}, where the same monotonicity is used to develop a viscosity theory for the Stefan problem).~However, as noted in \cite{ChSw},~\cite{Vis}, the stronger version of the boundary condition, \eqref{eq.SecStefan.weakStefan.bdryCond.strong}, and its implication that the phase is almost everywhere determined by the temperature, excludes the supercooled regime.~Indeed, if the temperature of the liquid is below freezing, the second equation in \eqref{eq.SecStefan.weakStefan.bdryCond.strong} is violated.~For this reason, we impose the weaker boundary condition~\eqref{eq.SecStefan.weakStefan.bdryCond}.

\begin{rmk}\label{rem:weakStefan.tooWeak}
The weak formulation \eqref{eq.SecStefan.weakStefan.PDE}--\eqref{eq.SecStefan.weakStefan.bdryCond} is too weak to pin down the solution of the 1SSP for the heat equation uniquely. In fact, for a fixed $w_0$, adding a Lebesgue null set to $\Gamma_{0-}$ and all $\Gamma_t$, $t\in[0,T]$, and adjusting $\chi(t,\cdot\,)$, $t\in[0,T]$ on Lebesgue null sets via the first equation in \eqref{eq.SecStefan.weakStefan.chi}, results in another solution of \eqref{eq.SecStefan.weakStefan.PDE}--\eqref{eq.SecStefan.weakStefan.bdryCond}. 
At the same time, an extended initial aggregate, even if it is extended by a Lebesgue null set, can produce a different weak solution, with the difference between the aggregates having a positive Lebesgue measure at positive times.  
~This is particularly clear in the case $d=1$, with $\Gamma_{0-}=(-\infty,0]$ and $w_0$ being the negative of a ``nice'' positive probability density on~$(0,\infty)$, where the probabilistic approach of \cite{DIRT2},~\cite{NadShk1},~\cite{LS},~\cite{Cuchiero},~\cite{DNS19} shows the existence of a weak solution $(w,\Gamma,\chi)$ with $\Gamma_\cdot\!=\!(-\infty,\Lambda_\cdot]$ for a non-decreasing $\Lambda$ satisfying $\Lambda_0\!=\!0$, and with $\chi(t,\cdot)=\bone_{\Gamma_t}$.
Concurrently, for a constant $C>0$ with $\int_0^C w_0(x)\,\mathrm{d}x>-C$, using the probabilistic approach, one can construct a solution $(\widehat{w},\widehat{\Gamma},\widehat{\chi})$ of \eqref{eq.SecStefan.weakStefan.PDE}--\eqref{eq.SecStefan.weakStefan.bdryCond} with the initial condition $(\widehat{\Gamma}_{0-}:=(-\infty,0]\cup\{C\},w_0)$, with $\widehat{\Gamma}_t=(-\infty,\widehat{\Lambda}^{(1)}_t]\cup[C-\widehat{\Lambda}^{(2)}_t,C+\widehat{\Lambda}^{(3)}_t]$, and with $\widehat\chi(t,\cdot)=\bone_{\widehat\Gamma_t}$, for non-decreasing non-negative $\widehat{\Lambda}^{(1)}$, $\widehat{\Lambda}^{(2)}$, $\widehat{\Lambda}^{(3)}$ that are strictly positive on $(0,\infty)$.~As~$\widehat{\Gamma}_{0-}$ coincides with~$\Gamma_{0-}$ up to a Lebesgue null set, we see that $(\widehat{w},\widehat{\Gamma},\widehat{\chi})$ is also a weak solution of \eqref{eq.SecStefan.weakStefan.PDE}--\eqref{eq.SecStefan.weakStefan.bdryCond} with the initial condition $(\Gamma_{0-},w_0)$. The solution $(\widehat{\Gamma},\widehat{w},\widehat{\chi})$ describes the situation when an instantaneous (homogeneous) nucleation occurs at the point $C$, upon which the initially infinitesimal solid crystal at $C$ starts to grow. Both solutions $(\widehat{w},\widehat{\Gamma},\widehat{\chi})$ and $(w,\Gamma,\chi)$ obey the equations \eqref{eq.SecStefan.weakStefan.PDE}--\eqref{eq.SecStefan.weakStefan.bdryCond} with the same initial condition.
Such a non-uniqueness is the reason why the 1SSP is often referred to as ill-posed.
However, the non-uniqueness can be resolved by imposing a natural minimality constraint on the solutions of \eqref{eq.SecStefan.weakStefan.PDE}--\eqref{eq.SecStefan.weakStefan.bdryCond}.
To illustrate this, we note that, in the above example of non-uniqueness, the aggregate $\Gamma$ has the property that there exists no $t\geq0$ such that $\widehat{\Gamma}_s$ is included in $\Gamma_s$ for all $s\in[0,t]$, with a strict inclusion for at least one $s\in[0,t]$. On the other hand, $\widehat{\Gamma}$ does not satisfy this property as $\Gamma_0$ is strictly included in $\widehat{\Gamma}_0$. 
The existing well-posedness results for the supercooled Stefan problem, all of which are restricted to $d=1$ and $\Gamma_{0-}=(-\infty,0]$ (see \cite{FP2}, \cite{ChSw}, \cite{DNS19}), make additional structural assumptions on $\Gamma$, ultimately ensuring that $\Gamma$ takes its minimal form $\Gamma_\cdot=(-\infty,\Lambda_\cdot]$, with the smallest possible function $\Lambda$ (see the discussion of minimality in \cite{DNS19}). Thus, in order to achieve uniqueness for  \eqref{eq.SecStefan.weakStefan.PDE}--\eqref{eq.SecStefan.weakStefan.bdryCond}, this system needs to be equipped with an additional ``minimality" condition. Of course, to date, the existence of a minimal solution to \eqref{eq.SecStefan.weakStefan.PDE}--\eqref{eq.SecStefan.weakStefan.bdryCond}, or of any solution at all, is not known except for the cases with sufficiently strong symmetry which are reducible to one-dimensional 1SSPs.
\end{rmk}

\subsection{Probabilistic solutions of the 1SSP for the heat equation}\label{se:prob_sol}
In this sub-section, we introduce the notion of a probabilistic solution for the 1SSP and prove that it strikes a middle ground between the classical and the weak solutions. Such a notion is especially well-suited for the study of the connection between the 1SSP and the scaling limits of external MDLA processes, and for $d=1$ it has already appeared in \cite{DNS19}, where the uniqueness and regularity of the one-dimensional probabilistic solution are shown (under an additional minimality assumption on $\Gamma$). It is also worth mentioning that the probabilistic solutions obtained in \cite{DIRT2}, \cite{NadShk1}, \cite{LS}, \cite{Cuchiero} for $d=1$ do satisfy the desired minimality property, necessary for their uniqueness. This observation provides another reason for the use of probabilistic solutions. Herein, we extend the notion of a probabilistic solution for the 1SSP to arbitrary $d\geq1$.

\begin{definition}\label{def:probStefan.sol}
Let $\Gamma_{0-}\subset\RR^d$ be non-empty and closed, and let $u_0$ be a probability density essentially supported in $\RR^d\backslash\Gamma_{0-}$. A triplet $(\mu,\Gamma,\chi)$, with a non-decreasing right-continuous family $\Gamma=\{\Gamma_t\}_{t\in[0,T]}$ of closed subsets of $\rr^d$, with $\mu\in\mathcal{P}(C([0,T],\RR^d))$, and with a function $\chi$ satisfying \eqref{eq.SecStefan.weakStefan.chi}, is called a \emph{probabilistic solution} of the 1SSP \eqref{eq.stefan.PDE.def} with the initial condition $(u_0,\Gamma_{0-})$ if
\begin{itemize}
\item $\Gamma_{0-}\subset\Gamma_t$, $t\in[0,T]$\emph{;}
\item $\mu$ equals to the distribution of the stochastic process $(\xi+B_{t\wedge\widehat{\tau}})_{t\in[0,T]}$, with a random vector $\xi\sim u_0(x)\,\mathrm{d}x$, an independent standard Brownian motion $B$, and $\widehat{\tau}:=\inf\{t\in[0,T]\!:\xi+B_t\in\Gamma_t\}$\emph{;}
\item for the canonical process $X$ on $C([0,T],\RR^d)$ and $\tau:=\inf\{t\in[0,T]\!:X_t\in\Gamma_t\}$, 
\begin{equation}\label{grow_local}
\int_{\RR^d} \varphi(x)\big(\chi(t,x)-\bone_{\Gamma_{0-}}(x)\big)\,\mathrm{d}x	
= \EE^\mu\big[\varphi(X_\tau)\,\mathbf{1}_{\{\tau\le t\}}\big],\;\varphi\in C^\infty_c(\RR^d,\RR),\; t\in[0,T], 
\end{equation}
where $\EE^\mu$ stands for the expectation under $\mu$. 
\end{itemize}
\end{definition}

\begin{rmk}
The inclusion $\Gamma_{0-}\!\subset\!\Gamma_0$ may be strict, as emphasized by the subscript $0-$.
\end{rmk}

\begin{rmk}
The local growth condition \eqref{grow_local} is equivalent to saying that, on each event $\{\tau\le t\}$,
the distribution of the (representative) absorbed particle $X_\tau$  is given by $(\chi(t,x)-\bone_{\Gamma_{0-}}(x))\,\mathrm{d}x$.
Recall also that, if the function $t\mapsto\text{Leb}(\Gamma_t)$ is continuous, the condition \eqref{eq.SecStefan.weakStefan.chi} yields $\chi(t,\cdot\,)=\bone_{\Gamma_t}$, which, in turn, implies that the particles absorbed by time $t$ are distributed uniformly in $\Gamma_t\setminus\Gamma_{0-}$. Physically, the latter corresponds to a constant enthalpy in the solid phase.
\end{rmk}


The next two propositions show that our notion of a probabilistic solution is natural. The first one explains how classical solutions lead to probabilistic solutions.

\begin{proposition}\label{prop:classic.is.prob}
Let $(w,\Gamma:=\{\Gamma_t\}_{t\in[0,T]})$ be a classical solution of the single-phase Stefan problem \eqref{eq.stefan.PDE.def}, such that $w$ is bounded, $w_0:=w(0,\cdot\,)\le0$ integrates to $-1$ and vanishes at infinity, and
\begin{equation}\label{level_set}
\Gamma_t=\{x\in\RR^d:\,\Phi(t,x)\leq0\},\;\; t\in[0,T],
\end{equation}
with a function $\Phi\in C^1((0,T]\times\RR^d,\RR)\cap C([0,T]\times\RR^d,\RR)$ satisfying $\partial_t\Phi\le0$ and $|\nabla_x \Phi|>0$ on $(0,T]\times\RR^d$.~Then, for any stochastic process $X=(\xi+B_{t\wedge\widehat{\tau}})_{t\in[0,T]}$ constructed under a probability measure $\PP$, with $\xi$ being $-w_0(x)\,\mathrm{d}x$-distributed, with $B$ being an independent standard Brownian motion, and with $\widehat{\tau}:=\inf\{t\!\in\![0,T]\!:\xi+B_t\!\in\!\Gamma_t\}$, the triplet $(\PP\circ X^{-1},\Gamma,\bone_\Gamma)$ is a probabilistic solution of \eqref{eq.stefan.PDE.def} with the initial condition $(-w_0,\Gamma_0)$. 
\end{proposition}

\begin{proof}
The family $\Gamma$ is non-decreasing and right-continuous by $\partial_t\Phi\le 0$ on $(0,T]\times\RR^d$ and by the continuity of $\Phi$. Since the first two bullet points in Definition \ref{def:probStefan.sol} hold by construction, it remains to show the probabilistic local growth condition \eqref{grow_local}. As a preparation, we record the Feynman-Kac formula for $w$:
\begin{equation}\label{FK}
w(t,x)=\E\big[w_0(x+B_t)\,\bone_{\{x+B_s\notin\Gamma_{t-s},\,s\in[0,t)\}}\big],\;\;(t,x)\in Q_T, 
\end{equation}	
resulting from the first, third and fourth lines in \eqref{eq.stefan.PDE.def}, as well as \eqref{level_set}, the continuity of $\Phi$ and $w$, and the inequality $|\nabla_x \Phi|>0$ on $(0,T]\times\RR^d$. On the other hand, multiplying both sides of \eqref{FK} by test functions of $x$ and integrating, we conclude that $-w(t,\cdot)$ is the density of the restriction of the distribution of $X_t$ to $\RR^d\setminus\Gamma_t$. 

\medskip

Next, let us prove that, for all $\varphi\in C^2_c(\RR^d,\RR)$ and almost every $t\in(0,T)$, 
\begin{equation}\label{eq.sec5.Strong.is.Prob.eq1}
\begin{split}
\frac{\mathrm{d}}{\mathrm{d}t}\int_{\Gamma_t\backslash\Gamma_0} \varphi(x)\,\mathrm{d}x
= \int_{\partial\Gamma_t} V(t,x)\,\varphi(x)\,\sigma_t(\mathrm{d}x),
 \end{split}
 \end{equation}
where $\sigma_t$ denotes the surface measure on the boundary $\partial\Gamma_t$ and $V=-\partial_t\Phi/|\nabla_x\Phi|$.~As the support of $\varphi$ is compact and both sides in \eqref{eq.sec5.Strong.is.Prob.eq1} are linear in $\varphi$, we only need to prove \eqref{eq.sec5.Strong.is.Prob.eq1} for~$\varphi$ supported in a sufficiently small open cube $R$, centered at an arbitrary point in~$\RR^d$. Once the center of $R$ is fixed, we choose its side length and $\varepsilon>0$ to be small enough, so that, for $(s,x)\in(t-\varepsilon,t+\varepsilon)\times R$, the equality $\Phi(s,x_1,\ldots,x_d)=0$ holds iff $x_i=\kappa(s,x_1,\ldots,x_{i-1},x_{i+1},\ldots,x_d)$ and $(x_1,\ldots,x_{i-1},x_{i+1},\ldots,x_d)\in U$, with $i\in\{1,\ldots,d\}$, $\kappa\in C^1((t-\varepsilon,t+\varepsilon)\times U,\RR)$ and an open $U\subset \RR^{d-1}$ (the same for all $(s,x)\in(t-\varepsilon,t+\varepsilon)\times R$). The existence of such a choice follows from $|\nabla_x \Phi|>0$ and the Implicit Function Theorem.~Assuming $i\!=\!1$ without loss of generality, we find
$$
\partial_{\alpha}\kappa(s,x_2,\ldots,x_d) = -\frac{\partial_{\alpha}\Phi}{\partial_{x_1}\Phi}\big(s,\kappa(s,x_2,\ldots,x_d),x_2,\ldots,x_d\big),\;\;
\alpha=s,x_2,\ldots,x_d.
$$
The map $(s,x_2,\ldots,x_d)\mapsto (\kappa(s,x_2,\ldots,x_d),x_2,\ldots,x_d)$ is of class~$C^1((t\!-\!\varepsilon,t\!+\!\varepsilon)\!\times\! U,R)$, with a Jacobian determinant $\partial_s\kappa$ non-negative or non-positive everywhere (note that $\partial_s\Phi\le0$ and $\partial_{x_1}\Phi\neq0$ in $(t-\varepsilon,t+\varepsilon)\times R$), and consequently an injection on the set $((t-\varepsilon,t+\varepsilon)\times U)\cap\{\partial_s\kappa\neq0\}$. Further, for any $t-\varepsilon<t_1<t_2<t+\varepsilon$, the image of $((t_1,t_2]\times U)\cap\{\partial_s\kappa\neq0\}$ under this map is contained in $(\Gamma_{t_2}\backslash\Gamma_{t_1})\cap R$ which, in turn, is contained in the image of $(t_1,t_2]\times U$. Then, by Sard's Theorem, the image of $((t_1,t_2]\times U)\cap\{\partial_s\kappa\neq0\}$ equals to $(\Gamma_{t_2}\backslash\Gamma_{t_1})\cap R$ up to a Lebesgue null set. Thus,
\begin{equation*}
\begin{split}
&\,\int_{\Gamma_{t_2}\backslash\Gamma_{t_1}} \varphi(x)\,\mathrm{d}x 
= \int_{(\Gamma_{t_2}\backslash\Gamma_{t_1})\cap R} \varphi(x)\,\mathrm{d}x\\
&= \int_{((t_1,t_2]\times U)\cap\{\partial_s\kappa\neq0\}}
\big|\partial_s\kappa(s,x_2,\ldots,x_d)\big|\,\varphi\big(\kappa(s,x_2,\ldots,x_d),x_2,\ldots,x_d\big)
\,\mathrm{d}s\,\mathrm{d}x_2\,\ldots\,\mathrm{d}x_d\\
&= \int_{(t_1,t_2]\times U}
\bigg|\frac{\partial_{s}\Phi}{\partial_{x_1}\Phi}(s,\kappa,x_2,\ldots,x_d)\bigg|
\,\varphi(\kappa,x_2,\ldots,x_d)
\,\mathrm{d}s\,\mathrm{d}x_2\,\ldots\,\mathrm{d}x_d.
 \end{split}
 \end{equation*}
Moreover, for any $s\in[t_1,t_2]$, the map $(x_2,\ldots,x_d)\mapsto (\kappa(s,x_2,\ldots,x_d),x_2,\ldots,x_d)$ is a bijection between $U$ and $\partial\Gamma_s$. Denoting its Jacobian by $J$, we recall that the associated metric tensor is given by $J^\top J$,
$$
\text{det}(J^\top J) = 1+ |\nabla_x\kappa|^2=\frac{|\nabla_x\Phi|^2}{(\partial_{x_1}\Phi)^2}\quad\text{and}\quad
\sigma_s(\mathrm{d}x)=  \bigg|\frac{\nabla_x\Phi}{\partial_{x_1}\Phi}(s,\kappa,x_2,\ldots,x_d)\bigg|\,\mathrm{d}x_2\,\ldots\,\mathrm{d}x_d.
$$
We now derive \eqref{eq.sec5.Strong.is.Prob.eq1}, using that $\partial_s\Phi\le0$, via
\begin{equation*}
\begin{split}
\int_{\Gamma_{t_2}\backslash\Gamma_{t_1}} \varphi(x)\,\mathrm{d}x 
&= \int_{t_1}^{t_2} \int_U
\bigg|\frac{\partial_{s}\Phi}{\partial_{x_1}\Phi}(s,\kappa,x_2,\ldots,x_d)\bigg|
\,\varphi(\kappa,x_2,\ldots,x_d)\,\mathrm{d}x_2\,\ldots\,\mathrm{d}x_d\,\mathrm{d}s\\
&= \int_{t_1}^{t_2}\int_{\partial \Gamma_s}
\frac{-\partial_{s}\Phi}{|\nabla_x\Phi|}(s,x)\,\varphi(x)\,\sigma_s(\mathrm{d}x) \,\mathrm{d}s.
 \end{split}
 \end{equation*}

\smallskip

Next, from \eqref{eq.sec5.Strong.is.Prob.eq1} and the second line in \eqref{eq.stefan.PDE.def} we obtain
\begin{equation}\label{eq.sec5.Strong.is.Prob.limBdry} 
\begin{split}
\frac{\mathrm{d}}{\mathrm{d}t}\int_{\Gamma_t\backslash\Gamma_0} \varphi(x)\,\mathrm{d}x	
= \int_{\partial\Gamma_t} V(t,x)\,\varphi(x)\,\sigma_t(\mathrm{d}x)
&= -\frac{1}{2} \int_{\partial\Gamma_t} (\nabla_x w\cdot\nu)(t,x)\,\varphi(x)\,\sigma_t(\mathrm{d}x)\\
& = -\frac{1}{2} \lim_{\varepsilon\downarrow0} \int_{\partial\Gamma^{\varepsilon}_t} (\nabla_x w\cdot\nu^\varepsilon)(t,x)\,\varphi(x)\,\sigma^{\varepsilon}_t(\mathrm{d}x),
\end{split}
\end{equation}
where $\Gamma^{\varepsilon}_t:=\{x\in\RR^d\!:\Phi(t,x)\leq\varepsilon\}$, while $\nu^\varepsilon$ and $\sigma^{\varepsilon}_t$ denote, respectively, the outward unit normal vector field and the surface measure on the boundary $\partial\Gamma^{\varepsilon}_t$. Indeed, for a local parameterization of~$\partial\Gamma^\varepsilon_t$, without loss of generality given by $(x_2,\ldots,x_d)\mapsto (\kappa^{\varepsilon}(t,x_2,\ldots,x_d),x_2,\ldots,x_d)$ with~$\kappa^{\varepsilon}$ of class $C^1$, it holds
$$
\sigma^{\varepsilon}_t(\mathrm{d}x) = \bigg|\frac{\nabla_x\Phi}{\partial_{x_1}\Phi}(t,\kappa^{\varepsilon},x_2,\ldots,x_d)\bigg|\, \mathrm{d}x_2\,\ldots\,\mathrm{d}x_d.
$$
The last equality in~\eqref{eq.sec5.Strong.is.Prob.limBdry} is due to $\kappa^{\varepsilon}\!\to\!\kappa$ as $\varepsilon\downarrow0$ and the continuity of $\nabla_x w$, $\nabla\Phi$.
~By Green's first identity, the first line in \eqref{eq.stefan.PDE.def}, and the analogue of \eqref{eq.sec5.Strong.is.Prob.eq1} for $\Gamma^\varepsilon_t$, we have
\begin{equation}\label{eq.sec5.Strong.is.Prob.epsEq.Fubini.Just}
\begin{split}
&\,-\frac{1}{2} \int_{\partial\Gamma^{\varepsilon}_t} (\nabla_x w\cdot \nu^\varepsilon)(t,x)\,\varphi(x)\,\sigma^{\varepsilon}_t(\mathrm{d}x) \\
& =\frac{1}{2} \int_{\RR^d\backslash\Gamma^\varepsilon_t} \Delta w(t,x)\,\varphi(x)\,\mathrm{d}x 
+ \frac{1}{2} \int_{\rr^d\backslash\Gamma^{\varepsilon}_t}  (\nabla_x w\cdot\nabla\varphi)(t,x)\,\mathrm{d}x \\
& =\int_{\rr^d\backslash\Gamma^{\varepsilon}_t} \partial_t w(t,x)\,\varphi(x)\,\mathrm{d}x
+\frac{1}{2}\int_{\rr^d\backslash\Gamma^{\varepsilon}_t}(\nabla_x w\cdot\nabla\varphi)(t,x)\,\mathrm{d}x \\
&=\frac{\mathrm{d}}{\mathrm{d}t}\!\int_{\rr^d\backslash\Gamma^{\varepsilon}_t} \!\! w(t,x)\varphi(x)\,\mathrm{d}x 
\!+\!\int_{\partial\Gamma^\varepsilon_t}\!\! V^\varepsilon(t,x) w(t,x) \varphi(x)\,\sigma^{\varepsilon}_t(\mathrm{d}x) 
\!+\!\frac{1}{2}\!\int_{\rr^d\backslash\Gamma^{\varepsilon}_t} \!\! (\nabla_x w\cdot\nabla\varphi)(t,x)\,\mathrm{d}x.
\end{split}
\end{equation}

\smallskip

Using \eqref{eq.sec5.Strong.is.Prob.limBdry}, \eqref{eq.sec5.Strong.is.Prob.epsEq.Fubini.Just}, the third line in \eqref{eq.stefan.PDE.def}, and Green's first identity, we continue:
\begin{equation*}
\begin{split}
\frac{\mathrm{d}}{\mathrm{d}t}\int_{\Gamma_t\backslash\Gamma_0} \varphi(x)\,\mathrm{d}x
& = \lim_{\varepsilon\downarrow0} \bigg(\frac{\mathrm{d}}{\mathrm{d}t}\int_{\rr^d\backslash\Gamma^{\varepsilon}_t} w(t,x)\,\varphi(x)\,\mathrm{d}x 
+ \frac{1}{2}\int_{\rr^d\backslash\Gamma^{\varepsilon}_t} \nabla_x w(t,x)\cdot\nabla\varphi(x)\,\mathrm{d}x\bigg) \\
& = \lim_{\varepsilon\downarrow0} \bigg(\frac{\mathrm{d}}{\mathrm{d}t}\int_{\RR^d\backslash\Gamma^\varepsilon_t}  
w(t,x)\,\varphi(x)\,\mathrm{d}x 
- \frac{1}{2}\int_{\rr^d\backslash\Gamma^{\varepsilon}_t} w(t,x)\,\Delta\varphi(x)\,\mathrm{d}x\bigg) \\
& = \frac{\mathrm{d}}{\mathrm{d}t}\int_{\RR^d\backslash\Gamma_t}  
w(t,x)\,\varphi(x)\,\mathrm{d}x 
- \frac{1}{2}\int_{\rr^d\backslash\Gamma_t} w(t,x)\,\Delta\varphi(x)\,\mathrm{d}x,
\end{split}
\end{equation*}
where we have relied on the Dominated Convergence Theorem to interchange $\lim_{\varepsilon\downarrow0}$ and $\mathrm{d}/\mathrm{d}t$ in the last equality (hence, we obtain it for almost every $t\in(0,T)$). The application of the Dominated Convergence Theorem is justified by the fact that the term on the first line in \eqref{eq.sec5.Strong.is.Prob.epsEq.Fubini.Just} and the last two summands on the last line in \eqref{eq.sec5.Strong.is.Prob.epsEq.Fubini.Just} are absolutely bounded uniformly in $\varepsilon$.~Recalling that $-w(t,\cdot)$ is the density of the restriction of the distribution of $X_t$ to  $\RR^d\setminus\Gamma_t$, we obtain, for almost every $t\in(0,T)$, 
\begin{equation}\label{eq.sec5.Strong.is.Prob.eqMain.1}
\begin{split}
\frac{\mathrm{d}}{\mathrm{d}t}\int_{\Gamma_t\backslash\Gamma_0} \varphi(x)\,\mathrm{d}x 
= -\,\frac{\mathrm{d}}{\mathrm{d}t}\,\EE\big[\varphi(X_t)\,\mathbf{1}_{\{\widehat{\tau}>t\}}\big]
+\frac{1}{2}\,\EE\big[\Delta\varphi(X_t)\,\mathbf{1}_{\{\widehat{\tau}>t\}}\big].
\end{split}
\end{equation}

\smallskip


On the other hand, for almost every $t\in(0,T)$, it holds 
\begin{equation}\label{eq.sec5.Strong.is.Prob.eqMain.2}
\begin{split}
\frac{\mathrm{d}}{\mathrm{d}t}\,\EE\big[\varphi(X_{\widehat{\tau}})\,\mathbf{1}_{\{\widehat{\tau}\le t\}}\big]
& = \frac{\mathrm{d}}{\mathrm{d}t}\,\EE\big[\varphi(X_{t\wedge\widehat{\tau}})\big]
- \frac{\mathrm{d}}{\mathrm{d}t}\,\EE\big[\varphi(X_t)\,\mathbf{1}_{\{\widehat{\tau}>t\}}\big] \\
& = \frac{1}{2}\,\EE\big[\Delta\varphi(X_t)\,\mathbf{1}_{\{\widehat{\tau}>t\}}\big]
- \frac{\mathrm{d}}{\mathrm{d}t}\,\EE\big[\varphi(X_t)\,\mathbf{1}_{\{\widehat{\tau}>t\}}\big],
\end{split}
\end{equation}
where the second equality follows by taking the expectation and then the derivative with respect to $t$ on both sides of It\^{o}'s formula
\begin{equation*}
\varphi(X_{t\wedge\widehat{\tau}})
=\varphi(\xi)+\frac{1}{2}\int_0^t \mathbf{1}_{\{\widehat{\tau}>s\}}\,\Delta\varphi(X_s)\,\mathrm{d}s
+\int_0^t \mathbf{1}_{\{\widehat{\tau}>s\}}\,\nabla\varphi(X_s)\cdot\mathrm{d}B_s.
\end{equation*}
By equating the left-hand sides of \eqref{eq.sec5.Strong.is.Prob.eqMain.1} and \eqref{eq.sec5.Strong.is.Prob.eqMain.2} one reaches \eqref{grow_local} with~$\chi(t,\cdot\,)\!:=\!\bone_{\Gamma_t}$, $t\in[0,T]$ via $\int_{\Gamma_0\backslash\Gamma_{0}} \varphi(x)\,\mathrm{d}x=0=\EE[\varphi(X_{\widehat{\tau}})\,\mathbf{1}_{\{\widehat{\tau}\le 0\}}]$.
\end{proof}

The next proposition shows how probabilistic solutions translate into weak solutions.

\begin{proposition}\label{prop:prob.is.weak}
Let $(\mu,\Gamma,\chi)$ be a probabilistic solution of the single-phase supercooled Stefan problem (see Definition \ref{def:probStefan.sol}) for $t\!\in\![0,T]$, with an initial condition $(u_0,\Gamma_{0-})$. Let $X$ denote the canonical process on $C([0,T],\RR^d)$ and let the density $u(t,\cdot)$ of the restriction of $\mu\circ X^{-1}_t$ to $\RR^d\backslash\Gamma_t$ be extended into $\Gamma_t$ by $0$. 
Then, $(-u,\Gamma,\chi)$ is a weak solution of the single-phase supercooled Stefan problem with the initial condition $(-u_0,\Gamma_{0-})$ in the sense of \eqref{eq.SecStefan.weakStefan.PDE}--\eqref{eq.SecStefan.weakStefan.bdryCond}.
\end{proposition}

\begin{proof}
Note that \eqref{eq.SecStefan.weakStefan.chi} is a part of Definition \ref{def:probStefan.sol}, and that $\Gamma$ and $w:=-u$ satisfy~\eqref{eq.SecStefan.weakStefan.bdryCond} by the construction of $u$. Thus, it remains to prove \eqref{eq.SecStefan.weakStefan.PDE}. We rewrite \eqref{eq.SecStefan.weakStefan.PDE} for a test function $\varphi(s,x)=\phi(s)\,\psi(x)$, with $\phi\in C^\infty([0,t],\RR)$ and $\psi\in C^{\infty}_c(\RR^d,\RR)$:
\begin{equation}\label{eq.sec5.Prob.is.Weak.Pf.eq1}
\int_\RR \int_{\RR^d} \psi(x)\,\big(\widehat{u}(s,x)+\widehat{\chi}(s,x)\big)\,\mathrm{d}x\,\mathrm{d}\phi(s)
+\frac{1}{2}\int_\RR \int_{\RR^d} \phi(s)\,\Delta\psi(x)\,u(s,x)\,\mathrm{d}x\,\mathrm{d}s=0,
\end{equation}
where $\widehat{u}(s,\cdot\,)\!:=\!u_0\,\bone_{\{0\}}(s)+u(s,\cdot\,)\,\bone_{(0,T]}(s)$, $\widehat{\chi}(s,\cdot\,)\!:=\!\bone_{\Gamma_{0-}}\,\bone_{\{0\}}(s)+\chi(s,\cdot\,)\,\bone_{(0,T]}(s)$, and we have taken $\phi\equiv0$ on $\RR\setminus[0,t)$. Observe that, once \eqref{eq.SecStefan.weakStefan.PDE} is established for the test functions of the form $\varphi(s,x)=\phi(s)\,\psi(x)$, a standard approximation argument shows that \eqref{eq.SecStefan.weakStefan.PDE} holds for all $\varphi\in C^\infty_c([0,t]\times\RR^d,\RR)$. Hence, it suffices to prove \eqref{eq.sec5.Prob.is.Weak.Pf.eq1}.

\medskip

To verify \eqref{eq.sec5.Prob.is.Weak.Pf.eq1}, we record that, for $0<s_1<s_2<t$,
\begin{align*}
\bigg|\int_{\RR^d} \psi(x)\,\big(u(s_2,x)-u(s_1,x)\big)\,\mathrm{d}x\bigg|
&=\big|\EE^{\mu}\big[\psi(X_{s_2})\,\bone_{\{\tau>s_2\}} - \psi(X_{s_1})\,\bone_{\{\tau>s_1\}}\big]\big| \\
&\leq \EE^{\mu} \big[|\psi(X_{s_1})|\,\bone_{\{\tau\in(s_1,s_2]\}}\big] + 
\EE^{\mu} \big[|\psi(X_{s_2})\!-\!\psi(X_{s_1})|\big].
\end{align*}
The limit of the latter expression as $s_1\uparrow s_2$ does not exceed the $(\sup_{\rr^d} |\psi|)$-multiple of the jump in $s\mapsto\mu(\tau\le s)$ at $s_2$. By invoking \eqref{grow_local}, we conclude that the jumps of 
$$
(0,t)\to\rr,\;\; s\mapsto \int_{\RR^d} \psi(x)\,\big(\widehat{u}(s,x)+\widehat{\chi}(s,x)\big)\,\mathrm{d}x
$$
do not exceed the $(2\sup_{\rr^d} |\psi|)$-multiples of the jumps in $s\mapsto\mu(\tau\le s)$. For any $\varepsilon>0$, \\ the latter function has at most a finite number of jumps that exceed $\varepsilon$.~Then, the Continuous Mapping Theorem implies that the left-hand side in \eqref{eq.sec5.Prob.is.Weak.Pf.eq1} can be approximated to any precision by replacing~$\phi$ with an appropriately chosen function of the form $\sum_{j=1}^J \phi(jt/J)\,\bone_{[(j-1)t/J,\,jt/J)}$.~Thus, it suffices to check \eqref{eq.sec5.Prob.is.Weak.Pf.eq1} for $\phi=\bone_{[s_1,s_2)}$ with arbitrary $0\leq s_1<s_2\leq t$. For such $\phi$, \eqref{eq.sec5.Prob.is.Weak.Pf.eq1} becomes
\begin{equation*}\label{eq.sec5.Prob.is.Weak.Pf.eq2}
\int_{\RR^d} \psi(x)\,\big(u(s_2,x)+\chi(s_2,x)-\widehat{u}(s_1,x)-\widehat{\chi}(s_1,x)\big)\,\mathrm{d}x
=\frac{1}{2}\int_{s_1}^{s_2}\!\!\int_{\RR^d} \Delta\psi(x)\,u(s,x)\,\mathrm{d}x\,\mathrm{d}s.
\end{equation*}

\smallskip

In view of the connection between $u$ and $X$, the last equation is equivalent to
\begin{equation}\label{23}
\begin{split}
&\,\int_{\RR^d} \psi(x)\,\big(\chi(s_2,x)-\widehat{\chi}(s_1,x)\big)\,\mathrm{d}x \\
&=-\EE\big[\psi(X_{s_2})\,\mathbf{1}_{\{\tau>s_2\}}\big]
+\EE\big[\psi(X_{s_1})\,\mathbf{1}_{\{\tau>\widehat{s_1}\}}\big]
+\frac{1}{2}\int_{s_1}^{s_2} \EE\big[\Delta\psi(X_s)\,\mathbf{1}_{\{\tau>s\}}\big]\,\mathrm{d}s,
\end{split}
\end{equation}
where $\widehat{s_1}:=s_1$ for $s_1>0$, $\widehat{s_1}:=0-$ for $s_1=0$, and we take $0>0-$.
By \eqref{grow_local},
\begin{equation}\label{24}
\int_{\RR^d} \psi(x)\,\big(\chi(s_2,x)-\widehat{\chi}(s_1,x)\big)\,\mathrm{d}x
=\EE\big[\psi(X_\tau)\,\mathbf{1}_{\{\tau\le s_2\}}\big]
-\EE\big[\psi(X_\tau)\,\mathbf{1}_{\{\tau\le \widehat{s_1}\}}\big],
\end{equation}
so it is enough to verify that the right-hand sides in \eqref{23} and \eqref{24} coincide. For this purpose, we calculate their difference to
\begin{align*}
&\EE\big[\psi(X_{s_2\wedge\tau})\big] - \EE\big[\psi(X_{s_1\wedge\tau})\big]
-\frac{1}{2}\int_{s_1}^{s_2} \EE\big[\Delta\psi(X_s)\,\mathbf{1}_{\{\tau>s\}}\big]\,\mathrm{d}s,
\end{align*}
which evaluates to $0$ thanks to It\^{o}'s formula for $\psi(X_{\cdot\wedge\tau})$.
\end{proof}

\subsection{Limit points of MDLA processes and probabilistic solutions of 1SSP}
\label{subse:DLA.vs.Stefan}

The analysis of this subsection is restricted to $d\in\{1,2\}$.~Consider a limit point~$(\mu^X,D)$ of a sequence of external MDLA processes fulfilling Assumption \ref{ass:1}, and the associated $\mu^X_0$,~$\Gamma_{0-}$. Assume that $\mu^X_0(\mathrm{d}x)=u_0(x)\,\mathrm{d}x$ for a probability density $u_0$ essentially supported in $\rr^d\backslash\Gamma_{0-}$. Then, by Theorem \ref{thm:char}, almost every realization of $(\mu^X,D)$ results in a pair $(\mu^X,\Gamma)$ (with $\Gamma:=(\Gamma_t)_{t\in[0,T]}$ and $\Gamma_t:=\{x\in\RR^d\!:D_t(x)=0\}$) satisfying all properties of a probabilistic solution of the 1SSP \eqref{eq.stefan.PDE.def} with the initial condition $(u_0,\Gamma_{0-})$ and with $\chi(t,\cdot\,)\!=\!\bone_{\Gamma_t}$, $t\in[0,T]$, possibly except for the growth condition \eqref{grow_local}. The next proposition shows that, under the additional Assumption~\ref{ass:2}, a version of \eqref{grow_local} with an inequality holds for $(\mu^X,\Gamma)$ almost surely.

\begin{proposition}\label{prop:limDLA.probSol.ineq}
\!For $d\!\in\!\{1,2\}$ and under Assumptions \ref{ass:1} and \ref{ass:2},~let $(\mu^X,D)$~be a limit point of a sequence of external MDLA processes and $\Gamma_t:=\{x\in\RR^d\!:D_t(x)=0\}$, $t\in[0,T]$. Then, for almost every realization of $(\mu^X,D)$, 
\begin{align}
&\int_{\Gamma_t\backslash\Gamma_{0-}} \varphi(x)\,\mathrm{d}x
\geq\EE^{\mu^X}\big[\varphi(X_{\tau})\,\bone_{\{\tau\leq t\}}\big] 
,\;\;\varphi\in C^\infty_c(\RR^d,[0,\infty)),\;\; t\in[0,T),\label{eq.FB.ineq}
\end{align}
where $X$ is the canonical process on $C([0,T],\RR^d)$, and $\tau:=\inf\{t\in[0,T]\!: X_t\in\Gamma_t\}$.
\end{proposition}

\begin{proof}
As at the beginning of Subsection \ref{subse:proof}, we use the Skorokhod Representation Theorem to assume that $(\mu^X,D,\mu^{\wt X})$ is the almost sure limit of $(\mu^{X,N},D^N,\mu^{\wt X,N})_{N\in\NN}$ on $(\Omega,\mathcal{F},\PP)$. Moreover, we extend $X$ to be the canonical process on $\mathcal{D}([0,T],\RR^d)$, which does not change \eqref{eq.FB.ineq} as $\mu^X$ is $\PP$-almost surely supported in $C([0,T],\RR^d)$. Next, we notice that, for any $N\in\NN$, under Assumption \ref{ass:2}, the number of cubes from~$\mathcal{C}_N$ that are added to $\Gamma^N_{0-}$ by a time $t\in[0,T]$ equals the number of particles absorbed by time $t$, and every added cube contains exactly one absorbed particle.~Recalling $\Gamma^N_t=\{x\in\RR^d\!:D^N_t(x)=0\}$, one has $\PP$-almost surely, 
\begin{equation}\label{eq.sec5.DLA.is.Supersol.Pf.eq0}
\int_{\Gamma^N_t\backslash\Gamma^N_{0-}} \varphi(x)\,\mathrm{d}x
=\frac{1}{N}\sum_{i=1}^N \varphi(X^{i,N}_t)\,\mathbf{1}_{\{\tau^{i,N}\le t\}}
+\big(\sup_{\RR^d} |\nabla\varphi|\big) \,O(N^{-\frac{1}{d}})
\end{equation}
for all $t\in[0,T]$ and $\varphi\in C^{\infty}_c(\RR^d,[0,\infty))$.~Now, we employ the continuous interpolations $Y^{i,N}$ of $\wt X^{i,N}$, constructed via \eqref{eq.sec4.2.YiN.def}. We first use \eqref{eq.sec4.2.tauiN.eq.sigmaiN}, i.e., $\tau^{i,N}\wedge T=\beta^N(Y^{i,N})\wedge T$ and $\sup_{t\in[0,T]} |X^{i,N}_t-\pi_{\Gamma^N}(Y^{i,N})_t| \leq N^{-\frac{1}{d}}$. Then, writing $\mu^{Y,N}$ for the empirical measure of $\{Y^{i,N}\}_{i=1}^N$, we deduce from \eqref{eq.sec5.DLA.is.Supersol.Pf.eq0} that, $\PP$-almost surely,
\begin{equation*}\label{eq.sec5.DLA.is.Supersol.Pf.eq0.2}
\int_{\Gamma^N_t\backslash\Gamma^N_{0-}} \varphi(x)\,\mathrm{d}x
=\EE^{\mu^{Y,N}}\big[\varphi(\pi_{\Gamma^N}(X)_t)\,\bone_{\{\beta^N(X)\leq t\}}\big]
+\big(\sup_{\RR^d} |\nabla\varphi|\big) \,O(N^{-\frac{1}{d}})
\end{equation*}
for all $t\in[0,T)$ and $\varphi \in C^{\infty}_c(\RR^d,[0,\infty))$.~Remark \ref{rem:beta.conv} and the Continuous Mapping Theorem demonstrate further that, $\PP$-almost surely, for every continuity point of the function $t\mapsto\mu^X(\beta(X)\leq t)$ in $[0,T)$,
\begin{equation}\label{eq.sec5.DLA.is.Supersol.Pf.eq0.3}
\lim_{N\rightarrow\infty}\int_{\Gamma^N_t\backslash\Gamma^N_{0-}} \varphi(x)\,\mathrm{d}x
=\EE^{\mu^{X}}\big[\varphi(X_{t})\,\bone_{\{\beta(X)\leq t\}}\big],\;\;\varphi\in C^{\infty}_c(\RR^d,[0,\infty)).
\end{equation}
Note that $\mathbf{1}_{\{\beta(X)\leq t\}}=\mathbf{1}_{\{\tau\le t\}}$ since $\beta(X)\wedge T=\tau \wedge T$.

\medskip

Next, to analyze the behavior of the left-hand side in \eqref{eq.sec5.DLA.is.Supersol.Pf.eq0.3}, we derive
$$
\limsup_{N\rightarrow\infty}\int_{\RR^d} \varphi(x)\,\bone_{\Gamma^N_t}(x)\,\mathrm{d}x
\leq \int_{\RR^d} \varphi(x)\,\bone_{\Gamma_{t,\varepsilon}}(x)\,\mathrm{d}x,\;\;t\in[0,T),\;\;\varepsilon>0,
$$
$\PP$-almost surely, with the open $\varepsilon$-neighborhood $\Gamma_{t,\varepsilon}$ of $\Gamma_t$. The above is first obtained for (a countable subset of) the continuity times $t$ of $D$ in $[0,T)$, and then for all $t\in[0,T)$ by an approximation with continuity times from the right. Passing to the limit $\varepsilon\downarrow0$, we get, $\PP$-almost surely,
\begin{equation}\label{eq.sec5.DLA.is.Supersol.Pf.eq1}
\limsup_{N\rightarrow\infty}\int_{\RR^d} \varphi(x)\,\bone_{\Gamma^N_t}(x)\,\mathrm{d}x
\le\int_{\RR^d} \varphi(x)\,\bone_{\Gamma_t}(x)\,\mathrm{d}x,\;\;t\in[0,T).
\end{equation}
In addition, the first part of Assumption \ref{ass:2} implies
\begin{equation*}
\lim_{N\rightarrow\infty} \int_{\RR^d} \varphi(x)\,\bone_{\Gamma^N_{0-}}(x)\,\mathrm{d}x
=\int_{\RR^d} \varphi(x)\,\bone_{\Gamma_{0-}}(x)\,\mathrm{d}x.
\end{equation*}
Upon subtracting the above equation from \eqref{eq.sec5.DLA.is.Supersol.Pf.eq1}, we infer from \eqref{eq.sec5.DLA.is.Supersol.Pf.eq0.3} that
\begin{equation}\label{prop5.7eq1}
\int_{\Gamma_t\backslash\Gamma_{0-}} \varphi(x)\,\mathrm{d}x
\geq\EE^{\mu^{X}}\big[\varphi(X_\tau)\,\bone_{\{\tau\leq t\}}\big],\;\;\varphi\in C^{\infty}_c(\RR^d,[0,\infty)),
\end{equation}
for (a countable subset of) the continuity points of $t\mapsto\mu^X(\beta(X)\leq t)$ in $[0,T)$, $\PP$-almost surely. Approximating any $t\!\in\![0,T)$ from the right by such continuity points and taking the limit on both sides in~\eqref{prop5.7eq1} we establish \eqref{eq.FB.ineq} for all $t\in[0,T)$.
\end{proof}


We note that equation \eqref{eq.sec5.DLA.is.Supersol.Pf.eq0} in the latter proof gives a version of \eqref{eq.FB.ineq}  with an equality for the external MDLA process $(\mu^{X,N},D^N)$, and that the error term in \eqref{eq.sec5.DLA.is.Supersol.Pf.eq0} even vanishes for functions $\varphi$ which are constant on the interiors of the cubes in $\mathcal{C}_N$. Nonetheless, it turns out that, when $d=2$, the local growth condition \eqref{grow_local} may indeed hold with a strict inequality, for all admissible choices of $\chi$, see Example \ref{counterexample} below. This means that, in general, the limit points of external MDLA processes do not solve the 1SSP. It is important to record the physical interpretation of a strict inequality in~\eqref{grow_local}: it indicates that the density of the absorbed particles in the aggregate is too low, which occurs because of the accumulation of microscopic holes created when cubes get attached to the aggregate in a ``disorderly" fashion. Thus, the density of the aggregate is lower and its enthalpy is higher than normal. Such an object is sometimes referred to as a ``mushy region" (see, e.g., \cite{Vis}).

\begin{example}\label{counterexample}
In this example, we construct a limit point $(\mu^X,D)$ of a sequence of MDLA processes, such that, for any $\chi$ satisfying \eqref{eq.SecStefan.weakStefan.chi}, the growth condition \eqref{grow_local} holds with an inequality for all admissible test functions $\varphi$ and this inequality is strict for some $\varphi$.
	To construct the desired example, we take $d=2$ and consider the initial aggregates $\Gamma^N_{0-}:=\RR^2\backslash(-N^{-\frac{1}{2}}/2,\,-N^{-\frac{1}{2}}/2+R)^2$, for an arbitrary (fixed) integer $R>1$ and $N=n^2$, $n\in\NN$ (so that each $\Gamma^N_{0-}$ is a union of squares from $\mathcal C_N$). Clearly, 
	$$
	\Gamma^{n^2}_{0-}\underset{n\to\infty}{\longrightarrow}
	\RR^2\setminus(0,R)^2=:\Gamma_{0-},\quad \text{Leb}(\Gamma^{n^2}_{0-}\,\Delta\,\Gamma_{0-})\underset{n\to\infty}{\longrightarrow}0.
	$$
	The distribution of the initial particle locations $(\widetilde{X}^{1,n^2}_0,\,\ldots,\,\widetilde{X}^{n^2,n^2}_0)$, denoted by $\upsilon^{n^2}_0$, is chosen as follows. First, we define 
	$$
	P^n:=\big\{(x,y)\in \mathbf{Z}_{n^2}\backslash\Gamma^{n^2}_{0-}:\,xn\,\text{mod}\,R^2 = 0\big\}
	$$
	and $C^n$ as the union of all cubes in $\mathcal C_{n^2}$ whose centers belong to $P^n$. We then have
	$$
	|P^n| = Rn \lfloor n/R \rfloor \in[n^2-Rn, n^2],\quad\text{Leb}(C^n) = (R/n)\lfloor n/R \rfloor \in [1-R/n,1].
	$$
	The distribution $\upsilon^{n^2}_0$ is obtained by assigning $|P^n|$ particles, selected uniformly at random from the total of $n^2$ particles, to distinct elements of $P^n$ uniformly at random, and subsequently assigning the remaining $n^2-|P^n|$ particles uniformly at random to distinct elements of $\mathbf{Z}_{n^2}\backslash (\Gamma^{n^2}_{0-}\cup C^n)$ (so that no two particles occupy the same site).
	
	\medskip
	
	It is clear that $\upsilon^{n^2}_0$ is symmetric.
	Using $(n^2-|P^n|)/n^2\leq R/n$ and the convergence of Riemann sums to the corresponding integral, we deduce, for any bounded continuous function $f\!:\RR^2\rightarrow\RR$, 
	$$
	\frac{1}{n^2} \sum_{i=1}^{n^2} f\big(\wt X^{i,n^2}_0\big)
	\underset{n\to\infty}{\longrightarrow}
	\frac{1}{R^2}\int_{(0,R)^2} f(x)\,\mathrm{d}x\quad\text{almost surely}. 
	$$
	Hence, \cite[Remark 2.3]{Sznitman} yields that $(\upsilon^{n^2}_0)_{n\in\nn}$ is $u_0(x)\,\mathrm{d}x$-chaotic, where $u_0:=\frac{1}{R^2}\,\bone_{(0,R)^2}$, in the sense that the first $(\RR^2)^k$-marginal of $\upsilon_0^{n^2}$ tends to $(u_0(x)\,\mathrm{d}x)^{\otimes k}$ for every~$k\in\nn$. 
	We let the underlying particle system $(\wt X^{1,n^2},\,\ldots,\,\wt X^{n^2,n^2})$ perform a bond exclusion process, so that Assumptions \ref{ass:1} and \ref{ass:2} are satisfied with $\xi\sim u_0(x)\,\mathrm{d}x$ (see the discussion after Assumption \ref{ass:2} and Proposition \ref{prop:brownian_limit}). 
	
	\medskip
	
	Consider a limit point $(\mu^X,D)$ of the proposed sequence of MDLA processes (which exists by Proposition \ref{prop:existence}),  the canonical process $X$ on $C([0,T],\RR^2)$, and its absorption time $\tau:= \inf\{t\in[0,T]\!:X_t\in\Gamma_t\}$. Then, for $t=0$, the right-hand side in \eqref{grow_local} reads
	\begin{equation*}
		\EE^{\mu^X}\big[\varphi(X_{\tau})\,\bone_{\{\tau\leq 0\}}\big]
		= \EE^{\mu^X}\big[\varphi(X_0)\,\bone_{\{\tau=0\}}\big]
		= \frac{1}{R^2} \int_{\Gamma_0\backslash\Gamma_{0-}}\,\varphi(x)\,\mathrm{d}x,
	\end{equation*}
	where the second equality is due to Theorem \ref{thm:char}. We claim that $\Gamma_0=\RR^2$, so that the right-hand side of the above equation equals $\int_{(0,R)^2} \varphi(x)\,\mathrm{d}x/R^2$. Indeed, each particle initially located in~$P^n$ is immediately absorbed, and hence $(\Gamma^{n^2}_{0-}\cup C^n)\subset\Gamma^{n^2}_t$, $t\in[0,T]$. Consequently, the distance from any point in $\RR^2$ to $\Gamma^{n^2}_t$ does not exceed $R^2/n$. Since $\Gamma_0=\bigcap_{0<t\le T} \Gamma_t$ and $\Gamma^{n^2}_t\to\Gamma_t$ in the sense of the distance function for all continuity times $t$ of $D$, we infer that $\Gamma_0=\RR^2$. Finally, we can make $\int_{(0,R)^2} \varphi(x)\,\mathrm{d}x/R^2$ arbitrarily close to~$1$ by picking a non-decreasing sequence of $\varphi$ in $C^\infty_c(\RR^2,[0,\infty))$ tending to $1$ pointwise. However, for such a sequence, the left-hand side in \eqref{grow_local} converges~to 
	\begin{equation*}
		\int_{\RR^2} \chi(0,x)-\bone_{\Gamma_{0-}}(x)\,\mathrm{d}x=\int_{\Gamma_0\backslash\Gamma_{0-}} \chi(0,x)\,\mathrm{d}x=\text{Leb}(\Gamma_0\backslash\Gamma_{0-})=R^2>1
	\end{equation*}
	by \eqref{eq.SecStefan.weakStefan.chi}. So, \eqref{grow_local} eventually holds with the inequality ``$>$'' for the $\varphi$'s in our sequence.  
\end{example}

\begin{rmk}
	In Example \ref{counterexample}, put $u(t,\cdot\,)$ for the density of the restriction of the distribution of $X_t$ to $\RR^2\backslash\Gamma_t$ under~$\mu^X$, extended into $\Gamma_t$ by $0$. Then, there is no locally integrable $\chi\!:[0,T]\times\RR^2\to[0,\infty)$ that would render $(-u,\Gamma,\chi)$ a weak solution of the 1SSP with the initial condition $(-u_0,\Gamma_{0-})$, in the sense of \eqref{eq.SecStefan.weakStefan.PDE}--\eqref{eq.SecStefan.weakStefan.bdryCond}. Indeed, since $\Gamma_t=\RR^2$, $t\in[0,T]$ almost surely under $\mu^X$, we have $u(t,\cdot\,)\equiv0$, $t\in[0,T]$. For $t=0$ and $\varphi\in C^\infty_c(\RR^2,\RR)$ with $\varphi=1$ on~$\Gamma_0\backslash\Gamma_{0-}$, the condition \eqref{eq.SecStefan.weakStefan.PDE} thus becomes
	\begin{align*}
		&\int_{\RR^2} -\varphi(x)\,\chi(0,x)
		+\varphi(x)\big(u_0(x)+\bone_{\Gamma_{0-}}(x)\big)\,\mathrm{d}x = 0,
	\end{align*}
	which under the condition \eqref{eq.SecStefan.weakStefan.chi} amounts to
	$$
	\int_{\Gamma_0\backslash\Gamma_{0-}} u_0(x)\,\mathrm{d}x = \text{Leb}(\Gamma_0\backslash\Gamma_{0-}).
	$$
	This is false, as the left-hand side equals to $1$, while the right-hand side is $R^2>1$.
\end{rmk}

\appendix

\section{}\label{appendixA}

The appendix is concerned with the behavior of the underlying particle system 
$(\wt X^{1,N}_t,\,\wt X^{2,N}_t,\,\ldots,\,\wt X^{N,N}_t)_{t\in[0,T+1]}$ that follows the bond exclusion process described after Assumption \ref{ass:2}.~Proposition \ref{prop:brownian_limit} is a $d$-dimensional generalization of \cite[Theorem~3.3]{Sznitman}. Although the proof therein can be adapted to any $d>1$ with only a few changes, we give here an alternative, somewhat more direct, proof.

\begin{proposition}\label{prop:brownian_limit}
Let the joint distribution $\upsilon_0^N$ of $(\wt X_0^{1,N},\,\ldots,\,\wt X_0^{N,N})$ be symmetric and the empirical measure of the initial locations $\{\wt X^{i,N}_0\}_{i=1}^N$ converge to an $\upsilon_0\in\mathcal{P}(\RR^d)$ weakly in probability as $N\to\infty$. Then, as $N\to\infty$, the empirical measure $\mu^{\wt X,N}$ of the paths~$\{\wt X^{i,N}\}_{i=1}^N$ converges to $\mu^{\wt X}$, the distribution on $\mathcal{D}([0,T+1],\RR^d)$ of a Brownian motion with the initial law $\upsilon_0$, weakly in probability.
\end{proposition}

\begin{proof} 
Note that the first particle $\wt X^{1,N}$ is a simple symmetric random walk on $\mathbf{Z}_N$ with the jump rate $N^{\frac{2}{d}}d$. Since the first marginal of $\upsilon_0^N$ converges to $\upsilon_0$ as $N\to\infty$ (see \cite[Remark 2.3]{Sznitman}), the law of $\wt X^{1,N}$ on $\mathcal{D}([0,T+1],\RR^d)$ tends to the Wiener measure with the initial distribution $\upsilon_0$. In addition, \cite[Proposition~2.2~ii)]{Sznitman} yields the tightness of the sequence $(\mu^{\wt X,N})_{N\in\NN}$ in $\mathcal{P}(\mathcal{D}([0,T+1],\RR^d))$. Upon restricting $N$ to a subsequence along which $\mu^{\wt X,N}$ converges in distribution to a random element $\widetilde{\mu}$ of $\mathcal{P}(\mathcal{D}([0,T+1],\RR^d))$, we use the Skorokhod Representation Theorem to assume, without loss of generality, that this convergence holds almost surely.  
Our goal is then to show that $\widetilde{\mu}=\mu^{\wt X}$ almost surely. 

\medskip

By the symmetry of the joint distribution of $\{\wt X^{i,N}\}_{i=1}^N$ we have, for all bounded and uniformly continuous functions $f\!:\mathcal{D}([0,T+1],\RR^d)\to\RR$,
\begin{equation*}
	\EE\bigg[\int_{\mathcal{D}([0,T+1],\RR^d)} \! f\,\mathrm{d}\mu^{\wt X,N} \bigg] \!
	=\EE\bigg[\frac{1}{N}\sum_{i=1}^N f(\wt X^{i,N})\bigg] \!
	= \EE\big[f(\wt X^{1,N})\big]
	\underset{N\to\infty}{\longrightarrow} \int_{\mathcal{D}([0,T+1],\RR^d)} \! f\,\mathrm{d}\mu^{\wt X}.
\end{equation*}
On the other hand, the leftmost expression in this display tends to the expectation of
\begin{equation*}
	\int_{\mathcal{D}([0,T+1],\RR^d)} \! f\,\mathrm{d}\wt\mu
\end{equation*}
along the subsequence in consideration. Thus, it suffices to prove that the variance of the above random variable is $0$, which can be reduced further, via the Dominated Convergence Theorem, to the convergence to $0$ of the variances of the integrals of $f$ with respect to $\mu^{\wt X,N}$, $N\in\NN$. 
Using symmetry again, we find that 
\begin{equation*}
	\text{Var}\bigg(\!\int_{\mathcal{D}([0,T+1],\RR^d)} f\,\mathrm{d}\mu^{\wt X,N}\!\bigg) 
	\!=\!\frac{1}{N^2}\sum_{i,j=1}^N \!\text{Cov}\big(f(\wt X^{i,N}),f(\wt X^{j,N})\!\big)
	\!=\!\mathrm{Cov}\big(f(\wt X^{1,N}),f(\wt X^{2,N})\!\big).
\end{equation*}
To prove that the obtained covariance tends to $0$ as $N\to\infty$, it is enough to verify that the joint law of $(\wt X^{1,N},\wt X^{2,N})$ converges to $\mu^{\wt X}\otimes\mu^{\wt X}$ as $N\to\infty$.

\medskip

To this end, we introduce the rescaled basis vectors $\wt e_k:=N^{-\frac{1}{d}}e_k$, $k=1,\,\ldots,\,d$, in $\RR^d$, and notice that the Markov process $(\wt X^{1,N},\wt X^{2,N})$ on $(\mathbf{Z}_N)^2$ has the generator
\begin{align*}
	(\mathcal{L}_Ng)(x^1,x^2)=&\,\frac{N^{\frac{2}{d}}}{2}\sum_{k=1}^d \big(g(x^1+\wt e_k,x^2)+g(x^1-\wt e_k,x^2)-2g(x^1,x^2) \\
	&\qquad\quad\;\;\;+g(x^1,x^2+\wt e_k)+g(x^1,x^2-\wt e_k)-2g(x^1,x^2)\big) \\
	&\,-\frac{N^{\frac{2}{d}}}{2}\,\bone_{\{|x^1-x^2|_1=N^{-\frac{1}{d}}\}}\lb g(x^1,x^1)-g(x^1,x^2)-g(x^2,x^1)+g(x^2,x^2)\rb\\
	=:&\,\frac{1}{2}(\Delta^1_Ng+\Delta^2_Ng)(x^1,x^2)-\frac{N^{\frac{2}{d}}}{2}\,\bone_{\{|x^1-x^2|_1=N^{-\frac{1}{d}}\}}(\Theta g)(x^1,x^2)
\end{align*}
for all $g\in C_c^2(\RR^{2d},\RR)$, with $|x|_1:=\sum_{k=1}^d |x_k|$. In particular, for such $g$, the process
\begin{equation*}
	g(\wt X_t^{1,N},\wt X_t^{2,N})-g(\wt X_0^{1,N},\wt X_0^{2,N})-\int_0^t(\mathcal{L}_Ng)(\wt X_s^{1,N},\wt X_s^{2,N})\,\mathrm{d}s,\;\;t\in[0,T+1]
\end{equation*}
is a martingale. As the laws of $\wt X^{1,N}$, $\wt X^{2,N}$ tend to $\mu^{\wt X}$, the sequence $(\wt X^{1,N},\wt X^{2,N})_{N\in\NN}$ is $C$-tight. To identify the limit points with $\mu^{\wt X}\otimes\mu^{\wt X}$, we recall the martingale problem on $C([0,T+1],\rr^{2d})$ characterizing $\mu^{\wt X}\otimes\mu^{\wt X}$, and thus aim to check that, for all $k\in\NN$, $0\leq t_1<\cdots<t_k\leq T_1<T_2\le T+1$ and $h_1,\,\ldots,\,h_k\in C_c(\RR^{2d},\RR)$,
\begin{align*}
	\EE\bigg[\bigg(g(\wt X_{T_2}^{1,N},\wt X_{T_2}^{2,N})
	\!-\! g(\wt X_{T_1}^{1,N},\wt X_{T_1}^{2,N})
	\!-\!\int_{T_1}^{T_2} \frac{1}{2}(\Delta g)(\wt X_t^{1,N},\wt X_t^{2,N})\,\mathrm{d}t\bigg)
	\prod_{l=1}^k h_l(\wt X^{1,N}_{t_l},\wt X^{2,N}_{t_l})\bigg]
\end{align*}
converges to $0$ as $N\to\infty$. This statement is equivalent to 
\begin{equation*}
	\EE\bigg[\bigg(\int_{T_1}^{T_2} (\mathcal{L}_Ng)(\wt X_t^{1,N},\wt X_t^{2,N})-\frac{1}{2}(\Delta g)(\wt X_t^{1,N},\wt X_t^{2,N})\,\mathrm{d}t\bigg)
	\prod_{l=1}^k h_l(\wt X^{1,N}_{t_l},\wt X^{2,N}_{t_l})\bigg]
	\underset{N\to\infty}{\longrightarrow}0.
\end{equation*}
Moreover, the Lagrange form of Taylor's Theorem reveals that
\begin{equation*}
	\Delta^1_Ng\!+\!\Delta^2_Ng\underset{N\to\infty}{\longrightarrow}\Delta g\;\;\text{uniformly},
	\quad\sup_{N\in\NN}\,\sup_{x^1,x^2\in\mathbf{Z}_N:\,|x^1-x^2|_1=N^{-\frac{1}{d}}} N^{\frac{2}{d}}(\Theta g)(x^1,x^2)<\infty. 
\end{equation*}
Therefore, it is enough to show that  
\begin{equation}\label{eq.AppendixA.mainProp.IN.def}
	I^{T_1,T_2}_N:=\EE\bigg[\int_{T_1}^{T_2} \bone_{\{|\wt X_t^{1,N}-\wt X_t^{2,N}|_1=N^{-\frac{1}{d}}\}}\,\mathrm{d}t\bigg]\underset{N\to\infty}{\longrightarrow}0.
\end{equation}

\smallskip

In order to establish \eqref{eq.AppendixA.mainProp.IN.def}, we consider the process
\begin{equation*}
	Y_t^N:=N^{\frac{1}{d}} |\wt X_t^{1,N}-\wt X_t^{2,N}|_1-1,\;\;t\in[0,T+1],
\end{equation*}
on $\{0,\,1,\,\ldots\}$. Note that, when $Y^N_t=n\ge 1$, it jumps to $(n+1)$ at the rate $(d+m)N^{\frac{2}{d}}$ and to $(n-1)$ at the rate $(d-m)N^{\frac{2}{d}}$, where $(d-m)$ is the Hamming distance\footnote{The Hamming distance between two vectors is the number of coordinates in which they differ.} between $\wt X_t^{1,N}$ and $\wt X_t^{2,N}$. When $Y^N_t=0$, it jumps to $1$ at the rate $(2d-1/2)N^{\frac{2}{d}}$. Next, we construct a Markov process $\widetilde{Y}^N\le Y^N$ on $\{0,\,1,\,\ldots\}$ with the generator
\begin{equation}\label{gen_on_N}
	\bone_{\{n\ge 1\}}\,dN^{\frac{2}{d}}\big(g(n+1)+g(n-1)-2g(n)\big)
	+\bone_{\{n=0\}}(2d-1/2)N^{\frac{2}{d}}\big(g(1)-g(0)\big).
\end{equation}
Specifically, $\widetilde{Y}^N_0:=Y^N_0$; whenever $\widetilde{Y}^N_{t-}=Y^N_{t-}=n\ge1$ and $Y^N_t=n-1$, we set $\widetilde{Y}^N_{t}=n-1$; whenever $\widetilde{Y}^N_{t-}=Y^N_{t-}=n\ge1$ and $Y^N_t=n+1$, we set $\widetilde{Y}^N_t=n-1$ with the probability $m/(d+m)$ and $\widetilde{Y}^N_t=n+1$ with the probability $d/(d+m)$, where $(d-m)$ is the Hamming distance between $\wt X_{t-}^{1,N}$ and $\wt X_{t-}^{2,N}$; whenever $\widetilde{Y}^N_{t-}=Y^N_{t-}=0$ and $Y^N_t=1$, we set $\widetilde{Y}^N_t=1$; and, whenever $\widetilde{Y}^N_t<Y^N_t$, the process $\widetilde{Y}^N$ evolves independently of $Y^N$ (according to its generator).
Finally, we let $S^N$ be a simple symmetric random walk on $\ZZ$ with the jump rate $dN^{\frac{2}{d}}$ and
$$
r(t):=\int_0^t \bone_{\{S^N_{r'}>0\}} + \frac{2d-1/2}{2d}\,\bone_{\{S^N_{r'}=0\}}\,\mathrm{d}r',\;\;t\ge0.
$$
Then, $|S_{r(\cdot)}^N|$ is a Markov process on $\{0,\,1,\,\ldots\}$ with the generator \eqref{gen_on_N}. Hence, 
\begin{equation*}
	I^{T_1,T_2}_N=\EE\bigg[\int_{T_1}^{T_2} \bone_{\{Y_t^N=0\}}\,\mathrm{d}t\bigg]
	\le\EE\bigg[\int_{T_1}^{T_2} \bone_{\{S_{r(t)}^N=0\}}\,\mathrm{d}t\bigg]
	\leq\frac{2d}{2d-1/2}\,\EE\bigg[\int_0^{T_2} \bone_{\{S^N_r=0\}}\,\mathrm{d}r\bigg].
\end{equation*}
Moreover, for any $\eps>0$,
\begin{equation*}
	\EE\bigg[\int_0^{T_2} \bone_{\{S^N_r=0\}}\,\mathrm{d}r\bigg]
	\leq\int_0^{T_2}\PP\big(N^{-\frac{1}{d}}\,S_r^N\in[-\eps,\eps]\big)\,\mathrm{d}r
	\underset{N\to\infty}{\longrightarrow} \int_0^{T_2} \PP\big(\sqrt{d}\,B_r\in[-\eps,\eps]\big)\,\mathrm{d}r,
\end{equation*}
where $B$ is a standard Brownian motion. Taking $\varepsilon\downarrow0$, we complete the proof.
\end{proof}

\bigskip

\bibliography{bibliography}

\providecommand{\bysame}{\leavevmode\hbox to3em{\hrulefill}\thinspace}
\providecommand{\MR}{\relax\ifhmode\unskip\space\fi MR }
\providecommand{\MRhref}[2]{%
  \href{http://www.ams.org/mathscinet-getitem?mr=#1}{#2}
}
\providecommand{\href}[2]{#2}
\begin{thebibliography}{10}

\bibitem{AlSh}
David Aldous and Paul Shields, \emph{A diffusion limit for a class of
  randomly-growing binary trees}, Probab. Theory Related Fields \textbf{79}
  (1988), no.~4, 509--542. \MR{966174}

\bibitem{BPP}
Martin~T. Barlow, Robin Pemantle, and Edwin~A. Perkins, \emph{Diffusion-limited
  aggregation on a tree}, Probab. Theory Related Fields \textbf{107} (1997),
  no.~1, 1--60. \MR{1427716}

\bibitem{BDLP}
Felipe Barra, Benny Davidovitch, Anders Levermann, and Itamar Procaccia,
  \emph{Laplacian growth and diffusion limited aggregation: different
  universality classes}, Physical Review Letters \textbf{87} (2001), no.~13,
  134501.

\bibitem{BeYa}
Itai Benjamini and Ariel Yadin, \emph{Diffusion limited aggregation on a
  cylinder}, Comm. Math. Phys. \textbf{279} (2008), no.~1, 187--223.
  \MR{2377633}

\bibitem{Bri}
M.~Brillouin, \emph{Sur quelques probl\`emes non r\'{e}solus de la {P}hysique
  {M}ath\'{e}matique classique. {P}ropagation de la fusion}, Ann. Inst. H.
  Poincar\'{e} \textbf{1} (1930), no.~3, 285--308. \MR{1507990}

\bibitem{CaMa}
L.~Carleson and N.~Makarov, \emph{Aggregation in the plane and {L}oewner's
  equation}, Comm. Math. Phys. \textbf{216} (2001), no.~3, 583--607.
  \MR{1815718}

\bibitem{ChSw}
L.~Chayes and G.~Swindle, \emph{Hydrodynamic limits for one-dimensional
  particle systems with moving boundaries}, Ann. Probab. \textbf{24} (1996),
  no.~2, 559--598. \MR{1404521}

\bibitem{Cuchiero}
C.~Cuchiero, S.~Rigger, and S.~Svaluto-Ferro, \emph{Propagation of minimality
  in the supercooled {S}tefan problem}, arXiv:2010.03580v1, 2020.

\bibitem{DIRT2}
F.~Delarue, J.~Inglis, S.~Rubenthaler, and E.~Tanr\'{e}, \emph{Particle systems
  with a singular mean-field self-excitation. {A}pplication to neuronal
  networks}, Stochastic Process. Appl. \textbf{125} (2015), no.~6, 2451--2492.
  \MR{3322871}

\bibitem{DNS19}
F.~Delarue, S.~Nadtochiy, and M.~Shkolnikov, \emph{Global solutions to the
  supercooled {S}tefan problem with blow-ups: regularity and uniqueness},
  arXiv:1902.05174v2 (2019).

\bibitem{DT}
A.~Dembo and L.-C. Tsai, \emph{The criticality of a randomly-driven front},
  Arch Rational Mech Anal \textbf{233} (2019), 643--699.

\bibitem{Dev}
Luc Devroye, \emph{A note on the height of binary search trees}, J. Assoc.
  Comput. Mach. \textbf{33} (1986), no.~3, 489--498. \MR{849025}

\bibitem{DF}
E.~DiBenedetto and A.~Friedman, \emph{The ill-posed {H}ele-{S}haw model and the
  {S}tefan problem for supercooled water}, Trans. Amer. Math. Soc. \textbf{282}
  (1984), no.~1, 183--204. \MR{728709}

\bibitem{Dudley}
R.~M. Dudley, \emph{Real analysis and probability}, {Cambridge Studies in
  Advanced Mathematics}, vol.~74, Cambridge University Press, 2002, 2nd
  edition.

\bibitem{EW}
Dorothea~Maria Eberz-Wagner, \emph{Discrete growth models}, ProQuest LLC, Ann
  Arbor, MI, 1999, Thesis (Ph.D.)--University of Washington. \MR{2699374}

\bibitem{ElNaSl}
Dor Elboim, Danny Nam, and Allan Sly, \emph{The critical one-dimensional
  multi-particle dla}, 2020.

\bibitem{El}
Ronen Eldan, \emph{Diffusion-limited aggregation on the hyperbolic plane}, Ann.
  Probab. \textbf{43} (2015), no.~4, 2084--2118. \MR{3353822}

\bibitem{FP2}
A.~Fasano and M.~Primicerio, \emph{A critical case for the solvability of
  {S}tefan-like problems}, Math. Methods Appl. Sci. \textbf{5} (1983), no.~1,
  84--96. \MR{690897}

\bibitem{LeGall}
J.-F.~Le Gall, \emph{Some properties of planar brownian motion}, \'{E}cole
  d'\'{E}t\'{e} de {P}robabilit\'{e}s de {S}aint-{F}lour {XX}---1990, Lecture
  Notes in Math., vol. 1527, Springer, Berlin, Heidelberg, 1992, pp.~111--229.

\bibitem{HaLe}
M.~Hastings and L.~Levitov, \emph{Laplacian growth as one-dimensional
  turbulence}, Physica D: Nonlinear Phenomena \textbf{116} (1998), 244--252.

\bibitem{HV}
Miguel~A. Herrero and Juan J.~L. Vel\'{a}zquez, \emph{Singularity formation in
  the one-dimensional supercooled {S}tefan problem}, European J. Appl. Math.
  \textbf{7} (1996), no.~2, 119--150. \MR{1388108}

\bibitem{HOL}
S.~D. Howison, J.~R. Ockendon, and A.~A. Lacey, \emph{Singularity development
  in moving-boundary problems}, Quart. J. Mech. Appl. Math. \textbf{38} (1985),
  no.~3, 343--360. \MR{800769}

\bibitem{Ishii}
H.~Ishii, \emph{On a certain estimate of the free boundary in the {S}tefan
  problem}, J. Differ. Equ. \textbf{42} (1981), 106--115.

\bibitem{JacodShiryaev}
J.~Jacod and A.~Shiryaev, \emph{Limit theorems for stochastic processes},
  Springer-Verlag, Berlin Heidelberg, 2003.

\bibitem{Ka}
K.~Kassner, \emph{Pattern formation in diffusion-limited crystal growth}, 1996.

\bibitem{Kelley}
J.~L. Kelley, \emph{General topology}, {Graduate Texts in Mathematics},
  vol.~27, Springer-Verlag New York, 1975.

\bibitem{Ke1}
Harry Kesten, \emph{How long are the arms in {DLA}?}, J. Phys. A \textbf{20}
  (1987), no.~1, L29--L33. \MR{873177}

\bibitem{Ke2}
\bysame, \emph{Upper bounds for the growth rate of {DLA}}, Phys. A \textbf{168}
  (1990), no.~1, 529--535. \MR{1077203}

\bibitem{KeSi2}
Harry Kesten, Vladas Sidoravicius, et~al., \emph{A problem in one-dimensional
  diffusion-limited aggregation (dla) and positive recurrence of markov
  chains}, The Annals of Probability \textbf{36} (2008), no.~5, 1838--1879.

\bibitem{KimVisc}
I.~C. Kim and N.~Poz\'ar, \emph{Viscosity solutions for the two-phase {S}tefan
  problem}, Commun. Part. Differ. Eq. \textbf{36} (2010), no.~1, 42--66.

\bibitem{LS}
S.~Ledger and A.~Sojmark, \emph{At the mercy of the common noise: blow-ups in a
  conditional {M}c{K}ean--{V}lasov problem}, arXiv:1807.05126, 2018.

\bibitem{Ma}
S\'{e}bastien Martineau, \emph{Directed diffusion-limited aggregation}, ALEA
  Lat. Am. J. Probab. Math. Stat. \textbf{14} (2017), no.~1, 249--270.
  \MR{3633231}

\bibitem{McK}
H.~P. McKean, \emph{Stochastic integrals}, AMS Chelsea Publishing, Providence,
  RI, 2005, Reprint of the 1969 edition, with errata. \MR{2169626}

\bibitem{Me1}
Paul Meakin, \emph{Formation of fractal clusters and networks by irreversible
  diffusion-limited aggregation}, Physical Review Letters \textbf{51} (1983),
  no.~13, 1119.

\bibitem{Me2}
\bysame, \emph{Multiparticle diffusion-limited aggregation with strip
  geometry}, Physica A: Statistical Mechanics and its Applications \textbf{153}
  (1988), no.~1, 1--19.

\bibitem{Me3}
\bysame, \emph{Fractals, scaling and growth far from equilibrium}, vol.~5,
  Cambridge university press, 1998.

\bibitem{NadShk1}
S.~Nadtochiy and M.~Shkolnikov, \emph{Particle systems with singular
  interaction through hitting times: application in systemic risk modeling},
  Ann. Appl. Probab. \textbf{29} (2019), no.~1, 89--129.

\bibitem{NoTu}
James Norris and Amanda Turner, \emph{Hastings--levitov aggregation in the
  small-particle limit}, Communications in Mathematical Physics \textbf{316}
  (2012), no.~3, 809--841.

\bibitem{Pi}
Boris Pittel, \emph{On growing random binary trees}, J. Math. Anal. Appl.
  \textbf{103} (1984), no.~2, 461--480. \MR{762569}

\bibitem{Rocka}
R.~T. Rockafellar and R.~J.-B. Wets, \emph{Variational analysis}, Grundlehren
  der mathematischen Wissenschaften, vol. 317, Springer, 1997.

\bibitem{RoMa}
Herbert~B. Rosenstock and Charles~L. Marquardt, \emph{Cluster formation in
  two-dimensional random walks: Application to photolysis of silver halides},
  Phys. Rev. B \textbf{22} (1980), 5797--5809.

\bibitem{Rudin}
W.~Rudin, \emph{Real and complex analysis}, McGraw-Hill, New York, 1966.

\bibitem{SaTa}
P.~Saffman and G.~I. Taylor, \emph{The penetration of a fluid into a porous
  medium or hele-shaw cell containing a more viscous liquid}, Proceedings of
  the Royal Society of London. Series A. Mathematical and Physical Sciences
  \textbf{245} (1958), 312 -- 329.

\bibitem{Sa}
Leonard~M. Sander, \emph{Diffusion-limited aggregation: A kinetic critical
  phenomenon?}, Contemporary Physics \textbf{41} (2000), no.~4, 203--218.

\bibitem{Sher}
B.~Sherman, \emph{A general one-phase {S}tefan problem}, Quart. Appl. Math.
  \textbf{28} (1970), 377--382. \MR{0282082}

\bibitem{SiSt}
Vladas Sidoravicius and Alexandre Stauffer, \emph{Multi-particle diffusion
  limited aggregation}, Invent. Math. \textbf{218} (2019), no.~2, 491--571.
  \MR{4011705}

\bibitem{Si}
Vittoria Silvestri, \emph{Fluctuation results for {H}astings-{L}evitov planar
  growth}, Probab. Theory Related Fields \textbf{167} (2017), no.~1-2,
  417--460. \MR{3602851}

\bibitem{Sly}
Allan Sly, \emph{On one-dimensional multi-particle diffusion limited
  aggregation}, arXiv preprint arXiv:1609.08107 (2016).

\bibitem{Stefan1}
J.~Stefan, \emph{\"{U}ber einige {P}robleme der {T}heorie der
  {W}\"{a}rmeleitung}, Sitzungber., Wien, Akad. Mat. Natur. \textbf{98} (1889),
  473--484.

\bibitem{Stefan2}
\bysame, \emph{\"{U}ber die {T}heorie der {E}isbildung}, Monatsh. Math. Phys.
  \textbf{1} (1890), no.~1, 1--6. \MR{1546138}

\bibitem{Stefan3}
\bysame, \emph{{\"U}ber die {V}erdampfung und die {A}ufl\"{o}sung als
  {V}org\"{a}nge der {D}iffusion}, Ann. Physik \textbf{277} (1890), 725--747.

\bibitem{Stefan4}
\bysame, \emph{\"{U}ber die {T}heorie der {E}isbildung, insbesondere \"{u}ber
  die {E}isbildung im {P}olarmeere}, Ann. Physik Chemie \textbf{42} (1891),
  269--286.

\bibitem{Sznitman}
A.-S. Sznitman, \emph{Topics in propagation of chaos}, \'{E}cole d'\'{E}t\'{e}
  de {P}robabilit\'{e}s de {S}aint-{F}lour {XIX}---1989, Lecture Notes in
  Math., vol. 1464, Springer, Berlin, 1991, pp.~165--251. \MR{1108185}

\bibitem{Vi}
Tamas Vicsek, \emph{Pattern formation in diffusion-limited aggregation},
  Physical review letters \textbf{53} (1984), no.~24, 2281.

\bibitem{VST}
Fredrik~Johansson Viklund, Alan Sola, and Amanda Turner, \emph{Small-particle
  limits in a regularized laplacian random growth model}, Communications in
  Mathematical Physics \textbf{334} (2015), no.~1, 331--366.

\bibitem{Vis}
A.~Visintin, \emph{Stefan problem with a kinetic condition at the free
  boundary}, Ann. Mat. Pura Appl. (4) \textbf{146} (1987), 97--122. \MR{916689}

\bibitem{Vo}
R.~F. Voss, \emph{Multiparticle fractal aggregation}, Journal of Statistical
  Physics \textbf{36} (1984), 861--872.

\bibitem{Vo2}
Richard~F Voss, \emph{Multiparticle diffusive fractal aggregation}, Physical
  Review B \textbf{30} (1984), no.~1, 334.

\bibitem{Whitt}
W.~Whitt, \emph{Stochastic-process limits: An introduction to
  stochastic-process limits and their application to queues}, Springer Series
  in Operations Research, Springer-Verlag, New York, 2002.

\bibitem{WiSa2}
Thomas~A Witten and Leonard~M Sander, \emph{Diffusion-limited aggregation},
  Physical Review B \textbf{27} (1983), no.~9, 5686.

\bibitem{WiSa}
TA~Witten~Jr and Leonard~M Sander, \emph{Diffusion-limited aggregation, a
  kinetic critical phenomenon}, Physical review letters \textbf{47} (1981),
  no.~19, 1400.

\end{thebibliography}
\bibliographystyle{amsplain}

\bigskip\bigskip

\end{document}